\documentclass[12pt, a4paper]{amsart}
\usepackage[margin = 2.8 cm]{geometry}
\usepackage{setspace}
\usepackage{hyperref}
\usepackage{amssymb}
\usepackage{mathrsfs}
\usepackage{enumitem}
\usepackage{amsthm}
\usepackage{xcolor, soul}
\usepackage{parskip}
\usepackage{tikz}
\usepackage{tikz-cd}
\usepackage{tabularray}
\usepackage[all]{xy}

\usepackage{cleveref}

\usepackage{relsize}
\usepackage[bbgreekl]{mathbbol}
\usepackage{amsfonts}
\DeclareSymbolFontAlphabet{\mathbb}{AMSb}
\DeclareSymbolFontAlphabet{\mathbbl}{bbold}
\newcommand{\prism}{{\mathlarger{\mathbbl{\Delta}}}}

\theoremstyle{plain}

\newtheorem{theorem}{Theorem}[section]
\newtheorem{prop}[theorem]{Proposition}
\newtheorem{lemma}[theorem]{Lemma}
\newtheorem{cor}[theorem]{Corollary}
\newtheorem{conj}[theorem]{Conjecture}

\theoremstyle{definition}
\newtheorem{assumption}[theorem]{Assumption}
\newtheorem{defn}[theorem]{Definition}
\newtheorem{def/prop}[theorem]{Definition/Proposition}
\newtheorem{example}[theorem]{Example}
\newtheorem{example/def}[theorem]{Example/Definition}
\newtheorem{remark}[theorem]{Remark}
\newtheorem{notation}[theorem]{Notation}

\newcommand{\AV}{\mathcal{A}}

\newcommand{\Af}{\mathbb{A}_f}

\newcommand{\Aut}{\mathrm{Aut}}
\newcommand{\Bun}{\mathrm{Bun}}
\newcommand{\Bdr}{B_\mathrm{dR}}

\newcommand{\C}{\mathbb{C}}
\newcommand{\cris}{\mathrm{cris}}

\newcommand{\cycl}{\mathrm{cycl}}

\newcommand{\F}{\mathbb{F}}
\newcommand{\Fl}{\mathcal{F}l}
\newcommand{\G}{\mathcal{G}}

\newcommand{\GL}{\mathrm{GL}}
\newcommand{\Gm}{{\mathbb{G}_m}}
\newcommand{\Gr}{\mathrm{Gr}}
\newcommand{\Hom}{\mathrm{Hom}}

\newcommand{\Igs}{\mathrm{Igs}}

\newcommand{\N}{\mathbb{N}}
\newcommand{\CO}{\mathcal{O}}
\newcommand{\Perf}{\mathrm{Perf}}
\newcommand{\Q}{\mathbb{Q}}
\newcommand{\R}{\mathbb{R}}

\newcommand{\Shi}{\mathcal{S}}
\newcommand{\Spre}{\Shi_{K^p}^{\circ,\mathrm{pre}}}
\newcommand{\Spa}{\operatorname{Spa}}
\newcommand{\Spd}{\operatorname{Spd}}
\newcommand{\Spec}{\operatorname{Spec}}
\newcommand{\Spf}{\operatorname{Spf}}

\newcommand{\Shio}{\Shi_{K^p}^\circ}
\newcommand{\Shib}{\Shi_{K^p}^\ast}
\newcommand{\Ig}{\mathrm{Ig}^b_C}
\newcommand{\Igb}{\mathrm{Ig}^{b,\ast}_C}
\newcommand{\Z}{\mathbb{Z}}

\newcommand{\CP}{\mathcal{P}}

\newcommand{\SK}{\mathscr{S}_K^\diamond}

\newcommand{\Sht}{\mathrm{Sht}_\G}

\newcommand*{\triple}[2][.1ex]{%
  \mathrel{\vcenter{\offinterlineskip%
   \hbox{$#2$}\vskip#1\hbox{$#2$}\vskip#1\hbox{$#2$}}}}

\newcommand*{\double}[2][.1ex]{%
  \mathrel{\vcenter{\offinterlineskip%
  \hbox{$#2$}\vskip#1\hbox{$#2$}}}}

\title[A PEL-type Igusa stack]{A PEL-type Igusa stack and the $p$-adic geometry of Shimura varieties}
\author{Mingjia Zhang}

\makeatletter
\def\l@subsection{\@tocline{2}{0pt}{2pc}{3pc}{}}
\makeatother

\begin{document}
\begin{abstract}
Let $(G,X)$ be a PEL-Shimura datum  of type AC in Kottwitz's classification. Assume $G_{\Q_p}$ is unramified. We show that the good reduction locus of the infinite $p$-level Shimura variety attached to this datum, considered as a diamond, can be described as the fiber product of a certain v-stack (which we call ``Igusa stack") with a Schubert cell of the corresponding $\Bdr^+$-affine Grassmannian, over the stack of $G_{\Q_p}$-torsors on the Fargues-Fontaine curve. We also construct a minimal compactification of the Igusa stack and show that this fiber product structure extends to the minimal compactification of the Shimura variety. When the Schubert cell of the affine Grassmannian is replaced by a certain bounded substack of $\G$-shtukas, where $\G$ is a reductive model of $G_{\Q_p}$ over $\Z_p$, we show that this fiber product recovers the integral model of the Shimura variety. This result on integral models, if specialized to a Newton polygon stratum, recovers the fiber product formula of Mantovan. Similar fiber product structures are conjectured by Scholze to exist on general Shimura varieties.
\end{abstract}

 \maketitle
 
{\setlength{\parskip}{1 pt plus 0pt}
\setlength{\parindent}{6pt}
\tableofcontents}

\setlength{\parskip}{2.5pt plus2pt}
\setlength{\parindent}{15pt}

\section{Introduction: Scholze's fiber product conjecture}
    The motivating question of this paper is to understand the geometry of Shimura varieties as $p$-adic analytic objects and the relation to that of their local counterparts. Instances of such relations can be dated back to the $p$-adic uniformization of Rapoport and Zink \cite{RZ}: it relates an open part (the basic Newton stratum) of PEL type Shimura varieties as $p$-adic rigid analytic spaces to simpler rigid spaces (Rapoport-Zink spaces). The formula for this uniformization formally resembles the complex uniformization expressing the Shimura varieties as adelic double quotients. 

As for a general Newton stratum labeled by an element $b$ in the corresponding Kottwitz set, Mantovan \cite{Mantovan} (cf. \cite{CS17} for this reformulation and notation) discovered that up to quotienting by the action of a certain group $\widetilde{G}_b$, it is a product of a corresponding Rapoport-Zink space $\mathcal{M}^b_\infty$ and a so-called Igusa variety $(\mathrm{Ig}^b_{\CO_K})^\text{ad}_\eta$. On the basic stratum, this takes the form of a $p$-adic ``uniformization", since in that case the Igusa variety is merely a profinite set. 

We show that for some PEL type Shimura varieties, via constructing a $p$-adic analytic stack, which we call ``Igusa stack", it is possible to interpolate between the strata and obtain a similar ``product structure" on the whole Shimura variety. In order to do this, we need to work relatively over a stack that interpolates the classifying stacks for the groups $\widetilde{G}_b$. This base turns out to be provided by the classifying stack of $G$-bundles on the Fargues-Fontaine curve, which appeared in the work of Fargues-Scholze \cite{FS}. Correspondingly the role of a $p$-adic symmetric space is played by a minuscule Schubert cell of the $\Bdr^+$-affine Grassmannian of Scholze-Weinstein \cite{Berkeley}, which interpolates the quotients $[\mathcal{M}^b_\infty/\widetilde{G}_b]$. Since a general formalism of stacks fibered over adic spaces is not available and might not behave well at all, we work in the category of small v-stacks on perfectoid spaces in characteristic $p$ in the framework of Scholze \cite{Sch18}. 

In very rough terms, the fiber product structure we seek for is a separation of the geometric information of a $p$-adic Shimura variety into a $p$-part and a prime-to-$p$ part, where the minuscule Schubert cell models the local geometry of the Shimura variety at $p$, while the Igusa stack records the global prime-to-$p$ information. Although in this work we only deal with certain PEL-type Shimura varieties, a similar fiber product structure is conjectured by Scholze\footnote{While writing up this paper, I learned that Patrick Daniels, Pol van Hoften and Daniel Kim have independently conjectured the integral version of the cartesian diagram in Conjecture \ref{conjecture}. This has led to the joint work \cite{DvHKZ}. Laurent Fargues and Matteo Tamiozzo have independently got similar ideas. A version of the Igusa stack for the modular curve appeared in the thesis of Tamiozzo.} to exist on general Shimura varieties. Let us give a precise formulation of this conjecture, before stating our results towards it.

\subsection{The fiber product conjecture}
Let $(G/\Q,X)$ be a Shimura datum, which determines a $G(\C)$-conjugacy class of minuscule cocharacters $[\mu^{-1}]$ with field of definition $E_0$. Fix a rational prime $p$ and let $E$ be the completion of $E_0$ at a prime above $p$ with residue field $\F_q$. Take a compact open subgroup $K=K_pK^p\subset G(\Af)$. Consider the category $\Perf$ of perfectoid spaces in characteristic $p$ and equip it with the v-topology. Let $\Shi_{K_pK^p}$ denote the diamond over $\Spd E$ attached to the corresponding Shimura variety at level $K_pK^p$ and $\Shi_{K^p}:=\varprojlim_{K_p}\Shi_{K_pK^p}$. Let $\Gr_G$ be the $\Bdr^+$-affine Grassmannian attached to $G_{\Q_p}$, considered as a diamond over $\Spd E$. Fix an isomorphism $\C\cong \bar{\Q}_p$ over $E_0$, where $\bar{\Q}_p$ is an algebraic closure of $\Q_p$ containing $E$. Fixing a maximal torus inside a Borel subgroup of $G_{\bar{\Q}_p}$, we choose a dominant cocharacter $\mu$ representing the $G(\bar{\Q}_p)$-conjugacy class $[\mu]$. Denote by $\Gr_{G,\mu}$ the Schubert cell labeled by $\mu$. Let $\Bun_G:=\Bun_{G_{\Q_p}}$ be the small v-stack on $\Perf$ of $G_{\Q_p}$-bundles on the Fargues-Fontaine curve. The affine Grassmannian maps to $\Bun_G$ via the Beauville-Laszlo map $BL:\Gr_G\rightarrow \Bun_G$. The Shimura variety $\Shi_{K^p}$ maps to the affine Grassmannian via the Hodge-Tate period map $\pi_{HT}: \Shi_{K^p}\rightarrow \Gr_G$, with image lying in $\Gr_{G,\mu}$. 

\begin{conj}{(Scholze)}\label{conjecture}
There exists a system of small v-stacks $\{\Igs_{K^p}\}_{K^p}$ on $\Perf$, together with maps $\Shi_{K^p}/\Spd E \rightarrow \Igs_{K^p}$ and $\Igs_{K^p}\xrightarrow{\bar{\pi}_{HT}} \Bun_{G}$ such that 
\begin{enumerate}[leftmargin=0.7cm]
\item (Cartesian diagram) For each $K^p$, the diagram
    \[
    \begin{tikzcd}
        \Shi_{K^p}\ar[r,"\pi_{HT}"] \ar[d] & \Gr_{G,\mu}\ar[d,"BL"]\\
        \Igs_{K^p}\ar[r, "\bar{\pi}_{HT}"] & \Bun_G
    \end{tikzcd}
    \]
is cartesian. 
\item (Hecke action) There exists a $G(\Af)$-action on $\{\Igs_{K^p}\}_{K^p}$ (where $G(\Q_p)$ acts trivially) descending that on 
$\{S_{K^p}\}_{K^p}$. In particular for any compact open subgroup $K_p$ of $G(\Q_p)$, we have a similar Cartesian diagram at level $K_p$, with the top row replaced by
\[\Shi_{K_pK^p} \xrightarrow{\pi_{HT,K_p}} [\Gr_{G,\mu}/K_p].\]
\item (Minimal compactification) There exist compactifications $\Igs_{K^p}\hookrightarrow \Igs^*_{K^p}$ over $\Bun_G$, extending the above cartesian diagram to the minimal compactification $\Shi_{K^p}^*$'s of the Shimura varieties.
    \item (Integral model) For $\G$ being a smooth parahoric model of $G$ over $\Z_p$, the cartesian diagram at level $K_p=\G(\Z_p)$ has a canonical integral model
 \[
    \begin{tikzcd}
    S_{K}^\Diamond\ar[r,"\pi_\text{crys}"] \ar[d] & \mathrm{Sht}_{\G,\mu}\ar[d]\\
    \Igs_{K^p}\ar[r, "\bar{\pi}_{HT}"] & \Bun_G
    \end{tikzcd}
    \]
    where $S_K^{\Diamond}$ is the v-sheaf\footnote{There are two ways to attach a v-sheaf to an $\CO_E$-scheme, see Definition \ref{BigDiamond} and here we are using the one that views a test perfectoid space as a locally ringed space with its structure sheaf (instead of the integral structure sheaf) as sheaf of rings.} associated with the (conjectural) schematic canonical integral model of the Shimura variety at level $K$ over $\CO_E$, uniquely characterized by Conjecture 4.2.2 of \cite{PR}, $\mathrm{Sht}_{\G,\mu}/\Spd \CO_E$ is the moduli stack of ($p$-adic) $\G$-shtukas with one leg bounded by $\mu$, and the map $\pi_\text{crys}$ is given by the (conjectural) universal $\G$-shtuka\footnote{Part b) of \cite[Conjecture 4.2.2]{PR} only asks there to be a universal $\mathcal{G}$-shtuka over $S_K^{\Diamond/}$, see Definition 2.3.2 of \textit{loc. cit.} for this notation. Yet it seems reasonable to expect that it extends to $S_K^\Diamond$ with some uniqueness properties.} on $S_K^{\Diamond}$. 
    
    \item (Functoriality) The construction is functorial in Shimura data.
\end{enumerate}
\end{conj}

Here $\Bun_G$, the affine Grassmannian $\Gr_G$ and the Beauville-Laszlo map are explained in detail in \cite[Chapter III]{FS}, \cite[Lecture 19]{Berkeley}. The construction of the Hodge-Tate period map is originally due to Scholze \cite{Sch15} and rewritten in \cite{CS17} for Hodge type Shimura varieties. In this generality, it is recorded in a preliminary draft of Hansen \cite{Hansen16}, cf.\cite[Section 2]{PR}, whose existence relies on the fact that the tautological $G(\Q_p)$-local system on the Shimura variety is de Rham, established by the work of Liu and Zhu \cite{LZ}.

To comment on the motivation and some features of conjecture \ref{conjecture}, we mention that it arises in the context of the geometrization of the local Langlands conjectures due to Fargues \cite{Far16} and Fargues-Scholze \cite{FS}. Conjecturally, to any discrete local Langlands parameter, there is a certain corresponding perverse sheaf on $\Bun_G$, thus realizing the local Langlands correspondence as a geometric Langlands correspondence on the Fargues-Fontaine curve. It is expected that the complex $R\pi_{HT,!}\overline{\Q}_l$, obtained by pushing forward a constant local system from the Shimura variety along the Hodge-Tate map, descends to $\Bun_G$ and relates to the conjectural perverse sheaves in some form of compatibility to the global Langlands correspondence \cite[Remark 1.18]{CS17}, \cite[Section 7]{Far16}. Conjecture~\ref{conjecture} is a geometric and hence more robust version of the weaker conjecture that $R\pi_{HT,!}\overline{\Q}_l$ descends.

Part $(4)$ of the conjecture can also be formulated by saying that the Igusa stack constructed from part $(1)$ using the generic fiber, when pulled back to the moduli stack of $p$-adic $\G$-shtukas bounded by $\mu$ is representable by a flat normal $\CO_E$-scheme with certain properties\footnote{One may hope for similar result for the minimal compactification, though due to stackiness at the boundary, the best reasonable conjecture is that the coarse quotient, i.e. the universal $v$-sheaf under the fiber product $\Igs_{K^p}^\ast\times_{\Bun_G}\mathrm{Sht}_{\mathcal{G},\mu}$, is representable by the minimal compactification of $S_K$ over $\CO_E$.}. This seems to provide a new way of constructing canonical integral models of Shimura varieties, even though our current approach to this part of the conjecture in the PEL case uses the existence of integral models as an input. It also supports the idea that shtukas in the sense of Scholze-Weinstein are some incarnation of motives in $p$-adic situations. 

Also, having the construction of Igusa stacks at hand, we can take their fiber products with various objects over $\Bun_G$, not necessarily the affine Grassmannian. This provides new semi-global companions of Shimura varieties. As pointed out to me by Tamiozzo, the conjecture could be potentially applied to a local version of the plectic conjectures by taking the fiber product of the Igusa stack with a moduli stack of shtukas with several legs. This idea will be pursued in a later work.

\begin{remark}
Our formulation of the conjecture does not uniquely characterize the system of v-stacks $\{\Igs_{K^p}\}_{K^p}$. For example it does not predict their images in $\Bun_G$ under $\overline{\pi}_{HT}$. A more ideal version of the conjecture would ask $\Igs_{K^p}$ to surject onto $\Bun_G$. Yet given that currently we can only approach the construction of $\Igs_{K^p}$ via Shimura varieties, we will be content with having a v-stack covered by the Shimura variety, whose image in $\Bun_G$ is hence bounded by $\mu$. At first sight this would lead to the Igusa stacks being dependent on $[\mu]$ and hence being defined over the residue field of $E$. However it is expected that they only depend on the Kottwitz set $B(G,\mu)$, not the conjugacy class $[\mu]$ itself, and hence $\Igs_{K^p}$ should already be defined over $\Spd \F_p$.\footnote{During the time this paper is reviewed, a unique characterization of the Igusa stacks has been given by Kim \cite{Kim}, where he also proved the functoriality of Igusa stacks, provided that they exist. Moreover, he extended Theorem~\ref{Main} beyond the good reduction locus.}     
\end{remark}

\subsection{Main results and organization}
The aim of this paper is to prove Conjecture \ref{conjecture} for PEL Shimura varieties of type AC in the classification of Kottwitz, and the main result is the following. The functoriality is carried out in \cite{DvHKZ}.
\begin{theorem}[\Cref{cartesian}, \Cref{cpf}, \Cref{Hecke}, \Cref{integral}]\label{Main}
    If $(G,X)$ is a PEL Shimura datum of type AC, assuming $G_{\Q_p}$ is unramified and $\G$ is reductive (see Assumption \ref{assumption}), then part $(1)(2)(4)$ of Conjecture \ref{conjecture} is true on the good reduction locus $\Shi_{K^p}^\circ$ of the Shimura variety. If we further assume that the minimal compactification of the Shimura variety has boundary codimension at least two, then part $(3)$ of the conjecture is true.\footnote{We do not obtain the optimal base field. In our construction the Igusa stacks live over the residue field of $E$. See Proposition \ref{classification} for a classification of the (simple) Shimura varieties that are excluded by the codimension condition.}
\end{theorem}

Our proof relies heavily on the fact that the Shimura variety in concern is a moduli space of abelian varieties with additional structures. In short, in this case the Igusa stack can be constructed as a moduli stack of abelian varieties up to isogenies in characteristic $p$. Upon relating points of the Schubert cell $\Gr_{G,\mu}$ (respectively $\Bun_G$) to $p$-divisible groups with additional structure via Dieudonn\'e theory, the desired cartesian property of the diagrams in part $(1), (4)$ of the conjecture follows from Serre-Tate theory of lifting abelian varieties. 

In Section 2 to Section 7 we review known results and present our global PEL setup.

In Section 8 we give a construction of the Igusa stack and show part $(1)$ of the conjecture on the good reduction locus. We define:

\begin{defn}[\Cref{DefIgusa}]
    Let $\Igs_{K^p}^\circ$ be the v-stackification of $\Igs^\text{pre}$, which takes an affinoid perfectoid space $T=\Spa(R,R^+) \in \Perf_{\F_q}$ to the groupoid $\Igs^\text{pre}(T)$ whose objects are $\Spec(R^+/\varpi)$-points of $S_K$, where $\varpi$ is a pseudo-uniformizer of $R^+$ and $S_K$ is the schematic Shimura variety at level $K$ over $\CO_{E}$. Isomorphisms between two objects are quasi-isogenies between abelian schemes compatible with extra structures.
\end{defn}

The map $\Shi_{K^p}^\circ/\Spd E \rightarrow \Igs_{K^p}^\circ$ is constructed by taking the reduction of abelian schemes over $R^+$ to $R^+/\varpi$, and we denote it by $\mathrm{red}$. The map $\bar{\pi}_{HT}^\circ: \Igs_{K^p}^\circ\rightarrow \Bun_{G}$ is constructed by taking the $G$-bundle on the Fargues-Fontaine curve attached to the rational Dieudonn\'e module of the objects in $\Igs_{K^p}^\circ$. With these we show:

\begin{prop}[\Cref{cartesian}]\label{rational}
For PEL Shimura varieties of type AC, with the above definitions, part $(1)$ of the conjecture is true on the good reduction locus $\Shi_{K^p}^\circ\subset \Shi_{K^p}$.  
\end{prop}

The proof uses the moduli interpretation as alluded to earlier, though the relation of $\Gr_{G,\mu}$ and $\Bun_G$ to $p$-divisible groups is only clean on rank one geometric points. Hence some effort is paid to extend the result from rank one points to a basis of the v-topology named ``product of points", see Example~\ref{Example: productofpoints}.

A sheaf theoretic corollary is the following.

\begin{cor}[\Cref{PBC}]
For a coefficient ring $\Lambda$ with $n\Lambda=0$ for some $n$ prime to $p$, we have natural base change equivalence
\[{BL}^\ast R\bar{\pi}^\circ_{HT,\ast}\cong R\pi^\circ_{HT,\ast}\mathrm{red}^\ast\]
of functors $D_\text{\rm {\'et}}(\Igs_{K^p}^\circ,\Lambda)\rightarrow D_\text{\rm {\'et}}(\Gr_{G,\mu},\Lambda)$. In particular, the complex $R\pi^\circ_{HT,\ast}\Lambda$ on $\Gr_{G,\mu}$ descends to the complex $R\bar{\pi}^\circ_{HT,\ast}\Lambda$ on $\Bun_G$. The same statement for pushforward with compact support is true.
\end{cor}

Section 9 deals with the minimal compactification and along the way investigates the geometry of Igusa varieties. This section is more technical, but arguably novel.

The idea (due to Scholze) of constructing a minimal compactification $\Igs^\ast_{K^p}$ is based on the fact that the fibers of the Hodge-Tate period map are affinoid. So upon imposing the condition that the boundary of the minimal compactification of the Shimura variety has codimension at least two, we can mimic a construction of a relative spectrum ``$\underline{\Spa}_{\Bun_G}(\overline{\pi}^\circ_{\text{HT},\ast}(\CO,\CO^+))$" over $\Bun_G$. That this relative spectrum, when taken fiber product with $\Gr_{G,\mu}$, recovers the minimal compactification of the Shimura variety would be a consequence of the algebraic Hartogs' extension lemma. To carry this out, we define the affinization of a small v-stack $X$ to be the v-sheaf
\[X_0: S\mapsto \mathrm{Hom}((\CO_X,\CO_X^+)(X), (\CO_S,\CO_S^+)(S)).\]

\begin{prop}[\Cref{construction}]
The functor on strictly totally disconnected perfectoid spaces over $\Bun_G$
\[\Igs^\ast: T\mapsto \mathrm{Hom}_T (T,\overline{(T\times_{\Bun_G}\Igs^\circ_{K^p})_0}^{/T}),\] 
where $^{/T}$ denotes the canonical compactification towards $T$, is a sheaf for the v-topology, and hence extends to a v-stack $\Igs_{K^p}^\ast$ with a $0$-truncated map to $\Bun_G$. It contains $\Igs_{K^p}^\circ$ as an open substack and the fiber product $\Igs_{K^p}^\ast\times_{\Bun_G}\Gr_{G,\mu}$ is isomorphic to the minimal compactification $\Shib$ of $\Shi_{K^p}$. Its structure morphism to $\Bun_G$ pulls back to the Hodge-Tate period map on $\Shib$ under this identification. 
\end{prop}

The main effort here is to show the pullback of $\Igs^\ast_T$ along a map of strictly totally disconnected spaces $T'\rightarrow T$ is indeed isomorphic to $\Igs^\ast_{T'}$. Write $\Igs_T$ for $T\times_{\Bun_G}{\Igs_{K^p}^\circ}$. Using perfectoid machinery and almost mathematics, this eventually boils down to a comparison between the global sections of the sheaf $\CO^+/\varpi$ for some pseudo-uniformizer $\varpi$ on $\Igs_T$ with $\CO^+(\Igs_T)/\varpi$. We first made a reduction to the case $T=\Spa(C,C^+)$ is a geometric point. Then using the comparison between the fibers of the Hodge-Tate period map with Igusa varieties due to \cite{CS19} and \cite{Mafalda}, we are reduced to show the natural map
\[\CO^+(\Ig)/\varpi\rightarrow (\CO^+/\varpi)(\Ig)\]
is an almost isomorphism. Here $\Ig$ is a perfectoid Igusa variety corresponding to some element $b$ in the Kottwitz set. This is constructed as the adic generic fiber of the deformation to $\Spf\CO_C$ of a perfect scheme $\mathrm{Ig}^b$ over the residue field of $C$. Using the short exact sequence for multiplication by $\varpi$ on the integral structure sheaf, what we need to show becomes the almost vanishing of the $\varpi$-torsion in $H^1(\Ig,\CO^+)$. This is almost isomorphic to the Witt vector cohomology of the perfect scheme $\mathrm{Ig}^b$. We found interestingly that in the generality of any perfect scheme, we have torsion-vanishing in its first Witt vector cohomology:

\begin{prop}[\Cref{torsionfree}]
Let $X$ be a perfect scheme in characteristic $p$. Denote by $W(\cdot)$ the $p$-typical Witt vectors. Then the Witt vector cohomology $H^1(X,W\CO_X)$ on the Zariski site of $X$ is $p$-torsionfree.
\end{prop}

This fulfills our purpose. The rest, namely to check that the fiber product recovers the minimal compactification of the Shimura variety, is easy and is again reduced to geometric points as test objects. Here we need to compare the global sections of the structure sheaves on $\Ig$ and its partial minimal compactification $\mathrm{Ig}_C^{b,\ast}$, which reduces to comparing those on their special fibers $\mathrm{Ig}^b$, $\mathrm{Ig}^{b,\ast}$. We thus make the assumption that the codimension of the boundary of $\mathrm{Ig}^{b,\ast}$ is at least two, so that affiness and normality of $\mathrm{Ig}^{b,\ast}$ allows us to apply algebraic Hartogs lemma. As a side result, we classify the cases we exclude. The assumption on codimension turns out to be rather mild.

\begin{prop}[\Cref{classification}]
If the boundary of the partial minimal compactification of an Igusa variety on a (simple) Shimura variety of PEL-type $AC$ has codimension one, then the Igusa variety must lie over the ordinary locus and the Shimura variety is either the modular curve, or a unitary Shimura curve attached to an imaginary quadratic extension of $\Q$ as in Example \ref{unitary}.
\end{prop}

The short Section 10 deals with the Hecke action. 

In the last Section 11, we establish an integral model of the cartesian diagram for $\Igs_{K^p}^\circ$ for $\G$ being a reductive model of $G_{\Q_p}$. Here we first define the moduli stack $\Sht$ of $\G$-shtukas and study its geometry. The main result is 

\begin{theorem}[\Cref{qsness},\ref{generic fiber}, \Cref{ProperDiag}, \Cref{qcqsSchubert}]
The structure map $\Sht\rightarrow\Spd \Z_p$ is quasi-separated, with proper diagonal, and for any affinoid perfectoid Tate ring $(R, R^+)$ and any commutative diagram with solid arrows
    \[
    \begin{tikzcd}
    \Spa(R,R^\circ)\ar[r,"f"]\ar[d]
    & \Sht \ar[d]\\
    \Spa(R,R^+)\ar[r]\ar[ru,dashed]
    & \Spd \Z_p,
    \end{tikzcd}
    \]
there is a unique (up to isomorphism) dotted arrow making the whole diagram commute up to a natural automorphism of $f$. For a dominant cocharacter $\lambda$ of $G_{\overline{\Q}_p}$, the bounded substack $\mathrm{Sht}_{\G,\lambda}$ is quasi-compact. Moreover, the generic fiber $\mathrm{Sht}_{\G,\Q_p}$ identifies with the quotient $[\Gr_G/\underline{K}_p]$ of the affine Grassmannian.
\end{theorem}

The proof relies on a recent result of Gleason-Ivanov-Zillinger \cite{GIZ} on extending shtukas to Breuil-Kisin-Fargues modules over products of rank one geometric points, as well as a result of Ansch\"utz about triviality of torsors on the spectrum of the ring $W(R^+)[1/p]$, where $R^+$ is the integral subring of such a test object.

We next introduce the crystalline period map on the formal integral model of the Shimura variety. The existence of the map is a consequence of the existence of a universal $\G$-shtuka on it. Pappas-Rapoport \cite{PR} showed this for Hodge-type Shimura varieties and we rephrased their construction in our situation. The desired cartesian diagram is easy to establish in this case, using qcqsness of the map $\overline{\pi}^\circ_{HT}$ and the quasi-separatedness of $\Sht$.

Finally we discuss the Newton stratification on our cartesian diagram. This recovers Mantovan's product formula, and in the special case of a basic stratum, the $p$-adic uniformization of Rapoport--Zink.

\subsection{Example of the modular curve}

Let us discuss the example of the modular curve in detail to illustrate the content of the conjecture. This also clarifies our conventions on Dieudonn\'e theory. The general case shares great similarity.

Consider the Shimura datum $G=\GL_2$, $X=\mathfrak{h}^+\coprod \mathfrak{h}^-$ the union of the complex upper and lower half plane, identified with the $\GL_2(\R)$-conjugacy class of the map
    \[h: \C\rightarrow \GL_2(\R): 
    a+bi\mapsto 
    \begin{pmatrix}
    a &b\\
    -b &a
    \end{pmatrix}.\] 
We also fix the diagonal torus $T$ and standard (upper triangular) Borel $B$ of $\GL_{2,\bar{\Q}}$. The root datum is
\[(\Z^2, \{\pm \alpha\}, \Z^2, \{\pm \alpha^\vee\}),\]
where the character lattice is trivialized by a basis $e_1$, $e_2$ with dual basis $e^\vee_1$, $e^\vee_2$ and $\alpha=e_1-e_2$, $\alpha^\vee=e_1^\vee-e_2^\vee$.
Then the minuscule cocharacter $\mu^{-1}$ can be chosen to be $(1,0)$ and a dominant cocharacter representing its inverse is $\mu=(0,-1)$. 
    
Fix the level subgroup $K_p=\GL_2(\Z_p)$ at $p$ and a prime-to-$p$ principal level $K^p=K(N)$, $p \nmid N\geq 3$. We let $K:=K_pK^p\subset \GL_2(\Af)$. The Shimura variety is the modular curve at level $K$. It is defined over $\Q$ and parametrizes isomorphism classes of elliptic curves with trivialized $N$-torsion points. We consider its base change to $\Q_p$ and take the attached diamond $\Shi_K$. By trivializing the Tate module of the universal elliptic curve, we obtain $\Shi_{K^p}$, the modular curve with infinite level at $p$. Fix an isomorphism $\C\cong \overline{\Q}_p$. In this case the Schubert cell $\Gr_{\GL_2,\mu}$ for $\GL_2/\Q_p$ is the diamond over $\Q_p$ attached to the flag variety for the opposite of the standard Borel, which is a projective line $\mathbb{P}^1$.

Here the Hodge-Tate period map measures the relative position of the Hodge-Tate filtration on the Tate-module of the universal elliptic curve $\mathcal{E}$, which is of the form
    \[{Lie}\mathcal{E}\hookrightarrow T_p\mathcal{E}\otimes_{\underline{\Z}_p} \CO_{\Shi_{K^p}}\cong \CO_{\Shi_{K^p}}^{\oplus 2}.\]
Here we use that Tate module is tautologically trivialized on $\Shi_{K^p}$. Hence this defines a map 
$\Shi_{K^p}\rightarrow\mathbb{P}^{1,\diamond}.$

The stack $\Bun_G=\Bun_2$ classifies rank two vector bundles on the Fargues-Fontaine curve. To define the Beauville-Laszlo map
\[BL:\mathbb{P}^{1,\diamond}\rightarrow \Bun_2,\] 
consider a test object $S\in \Perf$ with an untilt $S^\sharp$ over $\Spa \Q_p$, a map $x: S\rightarrow \mathbb{P}^{1,\diamond}/\Spd \Q_p$ gives a injection $\mathcal{L}\hookrightarrow\CO_{S^\sharp}^{\oplus 2}$ for some line bundle $\mathcal{L}$. The untilt $S^\sharp$ defines a closed Cartier divisor on the relative Fargues-Fontaine curve over $S$ and we denote the closed immersion by $i: S^\sharp\hookrightarrow X_S$. Then we define the image $BL(x)$ to be the limit $\mathcal{K}$ of the diagram
\[i_\ast\mathcal{L}\hookrightarrow i_\ast\CO_{S^\sharp}^{\oplus 2}\twoheadleftarrow \CO_{X_S}(1)^{\oplus 2}.
\]

We construct the stack $\Igs_{K^p}^\circ$ on $\Perf$ by sheafifying the presheaf of groupoids of isogeny classes of elliptic curves with $N$-level structures:
\[\Spa(R,R^+)\mapsto \{E/(R^+/\varpi), E[N]\cong (\Z/N)^2\}/\sim.\]
This maps to $\Bun_2$ by taking the (rational) crystalline Dieudonn\'e module of $E$, which is a rank two projective $B_\mathrm{cris}^+(R^+/\varpi)$-module with Frobenius and hence a rank two vector bundle on the Fargues-Fontaine curve $X_S$. We have a cartesian diagram
 \[
    \begin{tikzcd}
    \Shi^\circ_{K^p}\ar[r,"\pi_{HT}^\circ"] \ar[d] & \mathbb{P}^{1,\diamond}\ar[d,"BL"]\\
    \Igs_{K^p}^\circ\ar[r, "\bar{\pi}^\circ_{HT}"] & \Bun_2.
    \end{tikzcd}
    \]
We use a table below to describe its Newton stratification. We first introduce some notation: Let $k=\overline{\F}_p$. The Kottwitz set $B(\GL_2)$ for $\GL_2/\Q_p$ is in bijection to the dominant cocharacters and can be described by a pair of half integers (slopes) with nonincreasing order. The subset $B(G,\mu)$ of $\mu$-admissible elements consists of two points $[b_0], [b_1]$ with $[b_0]\leq [b_1]$ under the partial order, whose images under the Newton map are respectively $(-\frac{1}{2},-\frac{1}{2})$ and $(0,-1)$. 

Let $\mathbb{X}/k$ be the unique (up to isomorphisms) formal $p$-divisible group of height two and dimension one. Let $D_p$ be the nonsplit quaternion algebra over $\Q_p$ and $D$ be the endomorphism ring of a supersingular elliptic curve over $k$, tensored with $\Q$. This is a division $\Q$-algebra with $p$-adic completion $D_p$. Consider the special fiber $S_{K,k}$ of the integral model of the modular curve over $\Z_p$. We write $S_{K,k}^b$ for the usual Newton strata on $S_{K,k}$ and write $\Shi_{K^p}^{\circ,b}$ for the Newton strata on the good reduction locus on the infinite $p$-level diamond Shimura variety, defined by pullback from $\Bun_G$. We caution the reader that latter is not the same as the strata on $\Shi_{K^p}^{\circ}$ obtained by pullback from $S_{K,k}^b$ along the specialization map. They agree on rank one points but not in general. Thus our notation deviates from \cite[Section 3]{CS17}.

For any $b\in B(\GL_2,\mu)$, $\mathrm{Ig}^b$ denotes the corresponding Igusa variety, a perfect $k$-scheme. We denote by $\overline{\mathrm{Ig}^{b,\diamond}}$ the canonical compactification of its attached v-sheaf towards $\Spd k$. Let $\mathcal{BC}(\CO(1))$ be the Banach-Colmez space as in \cite[Chapter II]{FS} that sends a perfectoid space $S$ to the global sections of $\CO(1)$ on the relative Fargues-Fontaine curve $X_S$. 

In the last row of the table, we use $\overline{\Igs_{K^p}^{\circ,b}}$ to denote the canonical compactification of the stratum on $\Igs_{K^p}^\circ$ labeled by $b$ towards $\Bun_G$, cf. Proposition~\ref{stratum}.
\smallskip

\begin{center}
\begin{tblr}{
  colspec = {|ccc|},
  stretch = 0,
  rowsep = 7pt,
  hlines = {0.3pt},}
    & $[b_0]$ & $[b_1]$  \\
    slopes   & $-\frac{1}{2},-\frac{1}{2}$ & $0,-1$  \\
    isocrystal & $(\Breve{\Q}_p^2, F=\begin{pmatrix}
        0 & 1\\
        p^{-1} & 0
    \end{pmatrix}$) & $(\Breve{\Q}_p^2, F=\mathrm{diag}\{1,p^{-1}\})$  \\
    {isogeny class of\\ $p$-divisible groups} & $\mathbb{X}$& $\Q_p/\Z_p\oplus \mu_{p^\infty}$\\
    {vector bundle\\$\mathscr{E}_b$} & $\CO(\frac{1}{2})$& $\CO\oplus\CO(1)$\\
    $\widetilde{G}_b= \underline{\Aut}(\mathscr{E}_b)$ & $\underline{D}_p^\times$& $\begin{pmatrix}
        \underline{\Q}_p^\times & \mathcal{BC}(\CO(1))\\
        0 & \underline{\Q}_p^\times
    \end{pmatrix}$\\
    $\mathbb{P}^{1,b}$ & $\Omega:=\mathbb{P}^1\backslash \mathbb{P}^1(\Q_p)$& $\mathbb{P}^1(\Q_p)$\\
    $S_{K,k}^b$ & supersingular locus & ordinary locus\\
    $\Shi^{\circ,b}_{K^p}$ &{the residue discs of the\\supersingular points (open)}&  {complement of $\Shi^{\circ,b_0}_{K^p}$\\ (closed)}\\
    $\mathrm{Ig}^b$ & {the profinite set \\$D^\times\backslash D_p^\times\times \GL_2(\Af^p)/K^p$\\considered as a $k$-scheme} & {a $\Z_p^\times \times \Z_p^\times$-torsor over the\\ perfection of the ordinary locus} \\
    $\overline{\Igs_{K^p}^{\circ,b}}$ & $[\underline{D}^\times\backslash \underline{\GL_2
    (\mathbb{A}_f^p)}/\underline{K}^p]$  & {$[\overline{\mathrm{Ig}^{b,\diamond}}/\widetilde{G}_b]$}\\
\end{tblr}    
\end{center}

    \newpage
    \section*{Notations and conventions}

{\setlength{\parskip}{1.5 pt}
\setlength{\parindent}{0 pt}
\begin{itemize}[leftmargin=0.7cm]

    \item $\Af$: the ring of finite adeles of $\Q$
    \item $\Breve{\Q}_p$, $\Breve{\Z}_p$: the completed maximal unramified extension of $\Q_p$ and its ring of integers
    \item A non-archimedean field is a nondiscrete topological field $K$ whose topology is induced by a nonarchimedean norm $|\cdot|: K\rightarrow \R_{\geq 0}$. We denote by $\CO_K$ its ring of integers, i.e. where the norm is no more than one.
    \item For a complete non-archimedean field $K$, we write $\Spa K$ for the adic space $\Spa(K,\CO_K)$ and $\Spa \CO_K$ for $\Spa(\CO_K,\CO_K)$.
    \item We use covariant Dieudonn\'e theory and follow the convention of \cite{CS17} to divide the Frobenius in the usual convention by $p$. So the covariant Dieudonn\'e module of $\Q_p/\Z_p$ is $(\Z_p, F=1)$.
    \item For a geometric object $X$ (e.g. scheme, formal scheme, diamond etc.), we use $|X|$ to mean its underlying topological space.
    \item Underlined objects denote sheaves, e.g. $\underline{\Hom}$, $\underline{\Aut}$. For a topological space $X$, $\underline{X}$ means we view it as a sheaf on some site that sends a test object $S$ to continuous maps from $|S|$ to $X$.
    \item Our definition of Breuil-Kisin-Fargues module follows \cite[Definition 2.2.4]{PR}, which differs from \cite[Definition 11.4.3]{Berkeley}.
    \item Our sign convention for the Cartan decomposition on the $\Bdr^+$-affine Grassmannians follows \cite[Section 19.2]{SW}, and hence differ from \cite[p.30, cf. Footnote 14]{CS17}. For the various sign conventions in literature and justification for our choices, see Remark~\ref{rem: SignHT} and Remark~\ref{rem: SignShtuka}.
    
\end{itemize}}

    {\setlength{\parskip}{1 pt}
\setlength{\parindent}{10 pt}
\section*{Acknowledgements}
This paper is modified from my PhD thesis. I thank my supervisor Peter Scholze heartily for his constant support. In particular, the idea of using a ``relative spectrum" for the minimal compactification is due to him. Special thanks go to Matteo Tamiozzo and Dongryul Kim for very helpful comments on drafts of this paper. I also wish to thank Ian Gleason, Alexander Ivanov, Johannes Ansch\"utz, Andreas Mihatsch, Linus Hamann, Ben Heuer, Si-Ying Lee, David Schwein, Ana Caraiani, Zhiyu Zhang for helpful conversations, and the referee for a careful reading of this paper and detailed comments. This work is completed during my PhD at the Hausdorff Center for Mathematics and the Max-Planck Institute for Mathematics, during which time I was partly supported by DFG via the Leibniz prize of Peter Scholze. I thank both institutes for providing a nice working environment.

\newpage 
\section{Diamonds and v-stacks}
     
The objects in consideration will be stacks on the v-site of perfectoid spaces in characteristic $p$. We review basics of v-sheaf theory, setting up some necessary vocabulary, for details, see \cite[Sections 6-9,17]{Berkeley} and \cite[Sections 3,5-9,18]{Sch18}.

An affinoid (perfectoid) Tate ring is a pair of the form $(R,R^+)$, where $R$ is a (perfectoid) Tate ring, and $R^+\subset R^\circ$ is an open bounded and integrally closed subring. A morphism $(R,R^+)\rightarrow (R',R'^{+})$ between affinoid Tate rings is a map of topological rings $R\rightarrow R'$, carrying $R^+$ into $R'^{+}$. The tilt of an affinoid perfectoid Tate ring $(R,R^+)$ is the affinoid perfectoid Tate ring $(R^\flat,R^{\flat+})$. 

By considering Huber's adic spaces attached to affinoid perfectoid Tate rings, one has the notion of affinoid perfectoid spaces and their tilts. This construction is compatible with taking rational open subsets and hence globalizes. For a perfectoid space $X$, we denote by $X^\flat$ its tilt, which is a perfectoid space in characteristic $p$. There is a homeomorphism $|X|\cong |X^\flat|$, compatibly with passing to rational opens.

\begin{example}[{Geometric points}]
    Let $C$ be a complete algebraically closed non-archimedean field of characteristic zero or $p$ and $C^+\subset C$ an open and bounded valuation subring. Then $\Spa(C,C^+)$ is a perfectoid space. We call a perfectoid space of such form a geometric point. If $C^+=\CO_C$ is the ring of integers of $C$, we say that it is of rank one.
\end{example}

\begin{defn}\label{untilt}
    Let $X$ be a perfectoid space in characteristic $p$. An \textit{untilt} of $X$ is a pair $(X^\sharp, \iota)$, consisting of a perfectoid space $X^\sharp$ and an isomorphism $\iota: X^{\sharp\flat}\cong X$. We sometimes drop $\iota$ and simply write $X^\sharp$ for an untilt.
\end{defn}

For a morphism of perfectoid spaces, there are notions of the morphism being quasi-compact, quasi-separated, injection, (closed or open) immersion, separated, see \cite[Section 5]{Sch18}.

\subsection{Pro-\'etale and v-topology}
We refer to \cite[Definition 6.2, 7.8]{Sch18} for the notion of \'etale and pro-\'etale morphisms between perfectoid spaces.

\begin{defn}
Let $\Perf$ be the category of perfectoid spaces in characteristic $p$. 
\begin{enumerate}[label=(\roman*), leftmargin=0.5cm]
\item The pro-\'etale topology on $\Perf$ is the Grothendieck topology for which a collection of morphisms $\{f_i:Y_i\rightarrow X\}_{i\in I}$ is a covering, if all $f_i$ are pro-\'etale, and for each quasicompact open subset $U\subset X$, there exists a finite subset $J\subset I$ and quasicompact open subsets $V_i\subset Y_i, i\in J$, such that $U=\bigcup_{i\in J}f_i(V_i)$. The category $\Perf$, endowed with this topology, is called the big pro-\'etale site.
\item The v-topology on $\Perf$ is the Grothendieck topology where a collection of morphisms $\{f_i:Y_i\rightarrow X\}_{i\in I}$ is a covering, if for each quasicompact open subset $U\subset X$, there exists a finite subset $J\subset I$ and quasicompact open subsets $V_i\subset Y_i, i\in J$, such that $U=\bigcup_{i\in J}f_i(V_i)$. The category $\Perf$, endowed with this topology, is called the $v$-site.\footnote{To avoid using universe, one first takes cutoff cardinals and then takes a limit over all possible cutoffs to define the category of small sheaves on this site, as discussed in \cite[Section 4,8]{Sch18}. We ignore this issue here.}  
\end{enumerate}
\end{defn}

It is proven in \cite[Section 8.6, 8.7]{Sch18} that the big pro-\'etale site and the v-site, are subcanonical, i.e. the functor $\mathrm{Hom}(-,X)$ for $X\in \Perf$ is a sheaf on the big pro-\'etale and the v-site of $\mathrm{Perf}$. We will sometimes not distinguish a perfectoid space and the v-sheaf represented by it, and this is justified here.

\begin{example}[{Product of points, see \cite[Definition 1.2]{Ian}}]\label{Example: productofpoints}
Let $S=\Spa(A,A^+)$ be an affinoid perfectoid space in $\Perf$ with a pseudo-uniformizer $\varpi\in A^+$. For any point $x:(A,A^+)\rightarrow (K,K^+)$, let $\varpi_x$ be the image of $\varpi$ in $k(x)$, $k(x)^+$ the $\varpi_x$-adic completion of $K^+$ and $k(x):=k(x)^+[\frac{1}{\varpi_x}]$ the completed residue field. Define $R^+:=\prod_{x\in |S|}k(x)^+$, with a pseudo-uniformizer $\varpi':=(\varpi_x)$, and $R:=R^+[\frac{1}{\varpi'}]$. Then $\tilde{S}:=\Spa(R,R^+)$ is perfectoid and $\tilde{S} \rightarrow S$ is a v-cover. 

Generally we call an affinoid perfectoid space a product of (geometric) points if it is of the shape $\Spa(R,R^+)$, where $R^+=\prod_iK_i^+$, $R=R^+[\frac{1}{\varpi}]$, and each $(K_i,K_i^+)$ is an (algebraically closed) affinoid perfectoid field, $\varpi_i\in K_i$ a pseudo-uniformizer. Each $s_i:=\Spa(K_i, K_i^+)$ is called a principal component of $S$.
\end{example}

A product of points is an example of a \textit{totally disconnected} perfectoid space. By using geometric points in the above construction, one gets a v-cover by a \textit{strictly totally disconnected} perfectoid space. These spaces are important as they provide a basis of v-topology and are structurally simple.

\begin{defn}
A perfectoid space $X$ is called (strictly) totally disconnected if it is quasi-compact quasi-separated and every (\'etale) open cover of it splits.
\end{defn}

\begin{prop}[{\cite[Proposition 1.15]{Sch18}}]
A perfectoid space $X$ is (strictly) totally disconnected if and only if it is affinoid, and every connected component of $X$ is of the form $\Spa(K,K^+)$ for $K$ being a perfectoid field (resp. an algebraically closed perfectoid field) with an open and bounded valuation subring $K^+$.
\end{prop}

One can define and study stacks in this context.

\begin{defn}
A v-stack $F$ is a contravariant 2-functor from the v-site $\Perf$ to the 2-category of groupoids (whose objects are groupoids and morphisms are functors), satisfying descent for v-covers, i.e. for a v-cover $Y\rightarrow X$, the natural functor
\[F(X)\rightarrow F(Y/X),\]
is an equivalence of categories. Here $F(Y/X)$ is the category of descent data, i.e. the objects are couples $(s,\alpha)$, with $s\in F(Y)$ and $\alpha: p_1^\ast s\cong p_2^\ast s$, satisfying the cocycle condition $p_{23}^*\alpha \circ p_{12}^* \alpha = p_{13}^* \alpha$, where $p_1,p_2:Y\times_XY\double{\rightarrow}Y$, $p_{12},p_{23},p_{13}:Y\times_XY\times_XY\triple{\rightarrow}Y\times_XY$ are the projections.
\end{defn}

We will work exclusively with the following class of v-stacks that are more geometric in nature, in the sense that, using charts of perfectoid spaces, one can define underlying topological spaces for them.

\begin{defn}
A small v-stack is a v-stack $X$ on $\Perf$ admitting a presentation
\[R=Y\times_X Y \double{\rightarrow} Y\rightarrow X,\]
	with $Y$ being the v-sheaf represented by some perfectoid space (not necessarily in characteristic $p$), and $R$ is a $\mathrm{small}$ v-$\mathrm{sheaf}$, i.e. a v-sheaf admitting a surjection (of v-sheaves) from a perfectoid space. 
\end{defn}

For a small v-stack $X$ with presentation $R\double{\rightarrow} Y$, where $Y$ is a perfectoid space and $R$ is a small v-sheaf admitting a surjection from a perfectoid space $\tilde{R}\rightarrow R$, its underlying topological space is the quotient space $|X|=|Y|/|\tilde{R}|$. 
As a set, this is in bijection to 
\[\{\Spa(K,K^+)\xrightarrow{s} Y\}/\sim,\]
where $\Spa(K,K^+)$ runs through all affinoid perfectoid fields, and the equivalence relation is defined by $s_1\sim s_2$ if there is a commutative diagram 
    \[
    \begin{tikzcd}
        \Spa(K_3,K_3^+) \ar[r,two heads]\ar[d,two heads]\ar[dr,"s_3"] & \Spa(K_1,K_1^+)\ar[d,"s_1"]\\
        \Spa(K_2,K_2^+)\ar[r,"s_2"] & X,
    \end{tikzcd}
    \]
for some third affinoid perfectoid field $(K_3,K_3^+)$. The topological space $|X|$ is independent of the choice of presentation \cite[Proposition 12.7, Definition 12.8]{Sch18}.

\begin{example}
For $T$ a topological space, we denote by $\underline{T}$ the v-sheaf on $\Perf$ of continuous homomorphisms into $T$, i.e.
\[S\mapsto \mathrm{Hom}_\mathrm{cts}(|S|,T).\] 

Let $X\in \Perf$ be a perfectoid space in characteristic $p$, with an action by a topological group $G$, one can consider the v-sheaf theoretic coequalizer $[X/\underline{G}]$ of the projection and action maps $X\times \underline{G} \double{\rightarrow} X$. This is a v-stack. 
\end{example}

\begin{def/prop}[{Fiber product of v-stacks}]
    Given a diagram $X\xrightarrow{f} Z\xleftarrow{g} Y$ of small v-stacks. The fiber product $X\times_Z Y$ is the presheaf of groupoids that sends $S\in \Perf$ to the groupoid whose objects are triples
    \[(x,y,\varphi: f(x)\cong g(y))\]
    and morphisms between $(x,y,\varphi)$ and $(x',y',\varphi')$ are pairs of maps $(x\xrightarrow{\alpha} x', y\xrightarrow{\beta} y')$ such that $\varphi'\circ f(\alpha)=g(\beta)\circ \varphi$. This is again a small v-stack by \cite[Proposition 12.10]{Sch18}.
\end{def/prop}
For the universal property of a fiber product, see \cite[3.4.13]{Olsson}.

\begin{defn}
A v-stack $X$ is quasi-compact if it admits a surjection of v-stacks from an affinoid perfectoid space. In particular, if $X$ is quasi-compact, then it is small and its underlying topological space $|X|$ is quasi-compact.
\end{defn}

\begin{defn}[{Morphism of v-stacks, cf. \cite[Definition 10.7]{Sch18}}]
Let $f: Y\rightarrow X$ be a morphism of v-stacks. 
	\begin{enumerate}[label=(\roman*),leftmargin=0.5cm]
    \item $f$ is $0$-truncated if for all $S\in \Perf$, the map of groupoids $f(S): Y(S)\rightarrow X(S)$ is faithful, or equivalently the diagonal map $\Delta_f: Y \rightarrow Y\times_X Y$ is fully faithful.
	\item $f$ is quasi-compact if for any affinoid perfectoid space $S$ mapping to $X$, the fiber product $Y\times_X S$ is quasi-compact. 
	\item $f$ is quasi-separated if the diagonal, which is $0$-truncated, is quasi-compact quasi-separated (qcqs).
	\item $f$ is an open (resp. closed) immersion if for every (totally disconnected) perfectoid space $T$ mapping to $X$, the pullback $Y\times_XT\rightarrow T$ is represented by an open (closed) immersion.
	\item $f$ is separated if the diagonal $\Delta_{Y/X}:Y\rightarrow Y\times_X Y$ is a closed immersion (hence $f$ is automatically $0$-truncated).
	\item $f$ is partially proper if it is separated and for every diagram
	\[\begin{tikzcd}
	\Spa(R,R^\circ)\ar[r]\ar[d]			  & Y\ar[d,"f"]\\
	\Spa(R,R^+)\ar[r] 	\ar[ur,dashed]  & X
	\end{tikzcd}\]
where $(R,R^+)$ is any affinoid perfectoid Tate ring, there exists a unique dotted arrow making it commute.
	\end{enumerate}
\end{defn}

We give a criterion for small v-stacks to be qcqs. The proof is adapted from the proof of \cite[Theorem 21.2.1]{Berkeley}

\begin{prop}\label{qcqs}
    Let $X$ be a small v-sheaf. Let $Y$ be a small v-stack on the slice category $\Perf_{/X}$, such that the structure map to $X$ has quasi-separated diagonal. If for any product of geometric points $S\in \Perf_{/X}$ with principal components $s_i$, $i\in I$, the restriction
    \[res: Y(S)\rightarrow \prod_{i\in I} Y(s_i)\]
    is an equivalence of groupoids, then $f$ is qcqs. The converse implication holds if $f$ is representable in diamonds.
\end{prop}
\begin{proof}
    We first prove quasi-compactness assuming $f$ is quasi-separated. Take any affinoid perfectoid space $S$ with a map to $X$ and denote by $T$ the fiber product $Y\times_X S$. It suffices to show $T$ is quasi-compact. We fix a representative $\Spa(C_t,C_t^+)$ for each $t\in |T|$ (recall that $t$ is an equivalence class of maps) and choose a pseudo-uniformizer $\varpi$ on $S$. The map $t\rightarrow T\rightarrow S$ 
    pulls $\varpi$ back to a pseudo-uniformizer $\varpi_t\in C_t^+$. Define $R^+=\prod_{t\in|T|} C^+_t$, $\varpi=(\varpi_t)$ and $R=R^+[1/\varpi]$. Then $\widetilde{T}:= \Spa(R,R^+)$ is a product of geometric points and the collection of maps $t\rightarrow S$ determines a unique map $g: \widetilde{T}\rightarrow S$. Hence we obtain commutative diagrams
     \[\begin{tikzcd}
         t \ar[d, hook]\ar[r] &  T \ar[d,"\tilde{f}"] \ar[r] & Y\ar[d,"f"]\\
         \widetilde{T} \ar[r, "g"] \ar[ru, dashed, "\tilde{g}"]
         & S\ar[r] & X
     \end{tikzcd}\]
     By assumption, the outer commutative squares give a unique (up to automorphisms) map $\widetilde{T}\rightarrow Y$, and hence a unique (up to automorphisms) dotted arrow $\tilde{g}$ by universal property of $T$. By construction, it is surjective on topological spaces. As $g$ is qcqs and $\tilde{f}$ is quasi-separated, $\tilde{g}$ is qcqs by cancellation. This shows that it is in fact a surjection of v-stacks and $T$ is quasi-compact as wished.
    
     Now for a general $f$, we take an affinoid perfectoid space $S$ with a map to $Y\times_X Y$ and consider the pullback $T$ of the diagonal. Note that the map $\widetilde{\Delta}_f: T\rightarrow S$ is a quasi-separated map satisfying the condition in the proposition. Indeed, for any product of points $\widetilde{S}$ with a map to $S$, assume we have commutative diagrams for all principal components $\tilde{s}\in \widetilde{S}$
     \[\begin{tikzcd}
         \tilde{s} \ar[r, hook]\ar[d] & \widetilde{S} \ar[d] \ar[ld, dashed, swap, "\tilde{g}"] \ar[ddr]& \\
         T \ar[d] \ar[r,"\widetilde{\Delta}_f"] & S\ar[d]  & \\
         Y \ar[r,"\Delta_f"] & Y\times_X Y \ar[r] & X.
     \end{tikzcd}\]
     Then by assumption the outer commutative diagrams determine a unique (up to automorphisms) map $\widetilde{S}\rightarrow Y$. By uniqueness, its composition with $\Delta_f$ agrees with $\widetilde{S}\rightarrow S\rightarrow Y\times_X Y$ up to a natural transform. This induces a unique (up to automorphisms) map $\tilde{g}$ by the universal property. It makes the diagram commute, up to an automorphism in $T$ in the upper left triangle. Hence we can apply the argument in the first paragraph to $\widetilde{\Delta}_f$ and deduce that it is quasi-compact. Since this works for any $S$ mapping to $Y\times_X Y$, it shows $\Delta_f$ is quasi-compact and hence $f$ is quasi-separated. Now apply the first paragraph again we see that $f$ is qcqs.

     Conversely, if $f$ is qcqs and representable in diamonds, assume for some product of geometric points $S\in \Perf_{/X}$ with principal components $s_i$, $i\in I$, we are given lifts of $s_i\rightarrow X$ to $Y$. Consider the fiber product $T=Y\times_X S$. This is a spatial diamond. Each $s_i$ maps to $T$ by the universal property. Take a pro-\'etale surjection $\widetilde{T}\twoheadrightarrow T$ from an affinoid perfectoid space. The maps $s_i \rightarrow T$ lift to $\widetilde{T}$, which determines a section $S\rightarrow \widetilde{T}$.
     Composing with the projection to $T$, we get a unique (up to automorphisms) section $\tilde{g}: S\rightarrow T$. This constructs an inverse to $res: Y(S)\rightarrow \prod_{i\in I} Y(s_i)$.
\end{proof}
    
\begin{remark}
    For a map $f$ between small v-stacks with quasi-separated diagonal, one can show $f$ is qcqs by testing the above criterion on any pullback of $f$ to an affinoid perfectoid space. 
\end{remark}

\subsection{Diamonds}

\begin{defn}
A diamond is a pro-\'etale sheaf on $\Perf$ that can be written as $X/R$ with $X, R$ being representable by perfectoid spaces and $R \subset X\times X$ an equivalence relation, such that the two projections $s,t: R\rightarrow X$ are pro-\'etale.
\end{defn}

Diamonds are in fact (small) v-sheaves, see \cite[Proposition 11.9]{Sch18}. 

\begin{example}[{$\Spd E$}]
Let $E/\Q_p$ be a finite extension. Joining all $p$-power roots of unity and then taking completion, one gets the perfectoid field $E^{\cycl}$. Define 
\[\Spd E := \mathrm{coeq}(\Spa(E^\cycl)^\flat\times \underline{\mathrm{Gal}(E^\cycl/E)}\double{\rightarrow} \Spa(E^\cycl)^\flat).\]
This is a diamond whose underlying topological space is just a point.
\end{example}

The following theorem describes the category of perfectoid spaces over $\Q_p$ in terms of those in characteristic $p$ in aid of diamonds. 

\begin{theorem}\cite[Theorem 8.4.2]{Berkeley}
The category of perfectoid spaces over $\Q_p$ is equivalent to the category of perfectoid spaces $X$ of characteristic p with a structure morphism $X \rightarrow \Spd\Q_p$ as sheaves on $\Perf$.
\end{theorem}

\subsubsection{Diamonds attached to adic spaces}

\begin{defn}[{The functor ``$\diamond$"}]
Let $X$ be an analytic adic space over $\Spa \Z_p$. Define a presheaf $X^\diamond$ on $\Perf$ by:
\[T\mapsto X^\diamond(T)= \{(T^\sharp,T^\sharp\rightarrow X)\}/\sim,\]
where $T^\sharp$ is an untilt of $T$, $T^\sharp \rightarrow X$ is a map of adic spaces and the equivalence relation is given by isomorphisms of such pairs.
\end{defn}

According to \cite[Theorem 10.1.5]{Berkeley}, the presheaf $X^\diamond$ is a diamond. And in particular, if $X$ is perfectoid, then $X^\diamond$ is represented by $X^\flat$. If $X=\Spa(R,R^+)$, we also write $X^\diamond$ as $\Spd(R,R^+)$, or $\Spd R$ if $R^+=R^\circ$. 

In general for any pre-adic space (as in \cite[Appendix to lecture 3]{Berkeley}) over $\Spa \Z_p$, the same functor as in the exhibited formula above always defines a v-sheaf \cite[Lemma 18.1.1]{Berkeley}, though not necessarily representable by a diamond. This includes formal schemes over $\Spf \Z_p$. For schemes over $\Z_p$, there are two different ways of attaching v-sheaves to it constructed in \cite[Section 2.2]{AGLR}, according to whether we want to view the test object as a ringed space with sheaf of rings given by the structure or the integral structure sheaf, recalled below.

\begin{example}\label{vsheafFormal}
    For an affine $p$-adic formal scheme $\mathfrak{X}=\Spf A$ over $\Spf \Z_p$, the v-sheaf $\mathfrak{X}^\diamond$ is the v-sheaf attached to the pre-adic space $\Spa(A,A)$. This construction is compatible with localization and hence globalizes and defines a functor from $p$-adic formal schemes to v-sheaves.
\end{example}

\begin{defn}[{\cite[Definition 2.10]{AGLR}}]\label{BigDiamond}
Let $A$ be a $\Z_p$-algebra and $X=\Spec(A)$. 
    \begin{enumerate}
        \item The \textit{small diamond} functor $X^\diamond$ of $X$ is the v-sheaf on $\Perf$
        \[S\mapsto \{(S^\sharp, f: A\rightarrow \CO_{S^\sharp}^+(S^\sharp))\},\]
        where $S^\sharp$ is an untilt of $S$ and $f$ is a ring homomorphism.
        
        \item The \textit{big diamond} functor $X^\Diamond$ of $X$ is the v-sheaf on $\Perf$
        \[S\mapsto \{(S^\sharp, f: A\rightarrow \CO_{S^\sharp}(S^\sharp))\},\]
        where $S^\sharp$ is an untilt of $S$ and $f$ is a ring homomorphism.
    \end{enumerate}
\end{defn}
\begin{remark}
    On proper schemes, the two diamond functors agree by the valuative criterion for properness.
\end{remark}
\hfill

\section{\texorpdfstring{$p$}{}-divisible groups}
	 We review facts about $p$-divisible groups, following \cite[Section I.2]{Messing}, \cite[Section 4.1]{CS17}, \cite{SW}, \cite{AL}. 

\subsection{Basic definitions}
Let $S$ be a scheme.
\begin{defn}
A sheaf of abelian groups $\G$ on the fpqc site of $S$ is said to be $p^\infty$-torsion if it is the colimit of its $p^n$-torsion points (denoted by $\G[p^n]$). It is $p$-divisible if multiplication by $p$ on $\G$ is an epimorphism.
\end{defn}

\begin{defn}
Let $h$ be an integer $\geq 0$. A $p$-divisible group $\G$ over $S$ of height $h$ is a fpqc sheaf of abelian groups on $S$, which is $p^\infty$-torsion, $p$-divisible and each $\G[p^n]$ is representable by a finite locally free group scheme of order $p^{nh}$. Morphisms between $p$-divisible groups are morphisms of sheaves of groups on $S_\text{fpqc}$.
\end{defn}

The dual $p$-divisible group $\G^\vee$ of $\G$ is the fpqc sheaf $T\mapsto \varinjlim_n \G[p^n]^\vee(T)$ over $S$, where $\G[p^n]^\vee$ is the Cartier dual of $\G[p^n]$ and the transition maps are the duals of multiplication by $p$. This is clearly a $p$-divisible group.

\begin{defn}
An isogeny between two $p$-divisible groups is a surjection of fpqc sheaves whose kernel is representable by a finite locally free group scheme.
\end{defn}
For two $p$-divisible groups $\G,\G'$ on a scheme $S$, we write $\underline{\mathrm{Hom}}(\G,\G')$ for the sheaf of homomorphisms between them.
\begin{defn}
    Let $\G,\G'$ be two $p$-divisible groups over a scheme $S$. A quasi-isogeny is a global section $\rho$ of the sheaf $\underline{\mathrm{Hom}}(\G,\G')\otimes \Q$ such that Zariski locally on $S$, $p^n\rho$ is an isogeny for some integer $n$.
\end{defn}

\begin{defn}
A polarization on a $p$-divisible group $\G$ is a quasi-isogeny
\[\lambda: \G\rightarrow \G^\vee,\]
such that the Cartier dual of $\lambda$ equals $-\lambda$. It is called a principal polarization if it is an isomorphism.
\end{defn}


\begin{example}
	\begin{enumerate}
	\item $\Q_p/\Z_p$ is a $p$-divisible group of height one.
	\item $\mu_{p^\infty}:=\varinjlim_n \Gm[p^n]$ is a $p$-divisible group of height one.
	\item Let $A/S$ be a $d$-dimensional abelian scheme. The colimit of its $p$-power torsion points $A[p^\infty]:=\varinjlim_n A[p^n]$ is a $p$-divisible group of height $2d$.
	\end{enumerate}
\end{example}

In the above examples, $\Q_p/\Z_p$ and $\mu_{p^\infty}$ are dual to each other and $A[p^\infty]$ is dual to the $p$-divisible group of the dual abelian variety $A^\vee$. 
The pairing between them (or rather the duality pairings on $A^\vee[p^n]\times A[p^n]$ for each $n$) is called the Weil pairing. In particular when $A$ is principally polarized, $A[p^\infty]$ is self-dual via the principal polarization.

\begin{remark}
Note that a polarization $\lambda$ on an abelian variety induces a polarization on its $p$-divisible group. Although slightly confusingly, on the abelian variety $\lambda$ agrees with its dual isogeny, yet on the $p$-divisible group it is the inverse of its dual. This is a consequence of the expression of the Weil pairing as a commutator of two translation operators on the sheaf $([p^n]\times id)^\ast\mathscr{P}_A$ on $A\times A^\vee$, for each integer $n$, where $[p^n]$ denotes the multiplication by $p^n$ map and $\mathscr{P}_A$ denotes the Poincar\'e bundle on $A\times A^\vee$. For details, see \cite[Corollary 1.3]{Oda}.
\end{remark}
 
\begin{defn}
The fpqc sheaf $T_p\G=\varprojlim_n\G[p^n]$ (where the transition maps are multiplication by $p$) on $S$ is called the (integral) Tate module of $\G$. It is a sheaf of $\Z_p$-modules and identifies with the internal Hom $\mathcal{H}om(\Q_p/\Z_p,\G)$ in the category of sheaves of abelian groups over $S_{fpqc}$. It is representable by a scheme that is affine and flat over $S$, see \cite[Proposition 3.3.1]{SW}.
\end{defn}

When $S$ is the spectrum of a $p$-adically complete $\Z_p$-algebra $R$, which is our main case of interest, we will more often view a $p$-divisible group as an fpqc sheaf on $\mathrm{Nilp_{R}^\text{op}}$, the opposite category of $R$-algebras on which $p$ is nilpotent, sending $A\in \mathrm{Nilp_{R}^\text{op}}$ to $\varprojlim_i\varinjlim_n \G[p^n](A/p^i)$. 

Denote by $e_\G$ the zero section of $\G$.

\begin{defn}
Let $\hat{\G}$ be the fpqc sheaf on $\mathrm{Nilp_{R}^\text{op}}$:
\[A\mapsto \varinjlim_k\{x\in \G(A)\mid x=e_\G \text{ in } A/I, \text{ for an ideal } I\subset A, \text{ such that } I^{k+1}=0\}.\]
\end{defn}

\begin{prop}
The sheaf $\hat{\G}$ is a formal Lie (group) variety in the sense of \cite[Chapter II, Definition (1.1.4)]{Messing}. It is represented by an affine formal scheme over $\Spf R$. There is some integer $d\geq 0$, such that it is Zariski locally isomorphic to 
\[\Spf(R[[X_1,...X_d]]).\]
\end{prop}
\begin{proof}
This is \cite[Chapter II, Theorem 3.3.18]{Messing}, cf. \cite[Lemma 3.1.2]{SW}.
\end{proof}

We call $d$ the dimension of the $p$-divisible group $\G$ relative to $\Spf R$.

\begin{defn}
The fpqc sheaf of $\CO_S$-modules ${Lie}\,\G:= {Lie}\, \hat{\G}$
is the dual of the (Zariski) locally free $\CO_S$-module $\omega_{\G}:=e_{\hat{\G}}^*\Omega^1_{\hat{\G}/S}$ of rank $d$. It is called the Lie algebra of $\G$. We use straight letters $\operatorname{Lie}\G$ to denote its global sections. This is a finite projective $R$-module.
\end{defn}

When $\G$ is connected, $\G=\hat{\G}$ and is hence (pro-)representable by a formal scheme. In general it is not representable, but one can nevertheless define its adic generic fiber $\G^\text{ad}_\eta$ by \cite[Proposition 2.2.2]{SW}.

\subsection{Classification over \texorpdfstring{$\CO_C$}{}}

Let $C/\Q_p$ be a complete algebraically closed non-archimedean field with ring of integers $\CO_C$. Scholze and Weinstein have classified $p$-divisible groups over $\CO_C$ in terms of the Hodge-Tate filtration on their Tate modules \cite[Theorem B]{SW}.

Let $\G$ be a $p$-divisible group over $\CO_C$. Recall the Hodge-Tate exact sequence (due to Fargues) as in \cite[Theorem 12.1.1]{Berkeley}.

\begin{theorem}\label{BTclassification}
There is a natural short exact sequence:
\[0\rightarrow \operatorname{Lie}\G\otimes_{\CO_C}C(1)\xrightarrow{\alpha_{G^*}^*(1)}T_p\G (\CO_C)\otimes_{\Z_p}C\xrightarrow{\alpha_G}(\operatorname{Lie}\G^\vee)^\ast\otimes_{\CO_C}C\rightarrow 0.\] 
\end{theorem}  

Here to define $\alpha_G$, we view a section $f$ of $T_p\G$ as a homomorphism $\Q_p/\Z_p \rightarrow \G$. Then the Lie algebra functor applied to its dual $f^\vee: \G^\vee\rightarrow \mu_{p^\infty}$ gives $\operatorname{Lie}(f^\vee): \operatorname{Lie}\G^\vee\rightarrow \operatorname{Lie}\mu_{p^\infty}$. By picking a coordinate of $\Gm$, say $t$, the $\CO_C$-linear dual $(\operatorname{Lie}\mu_{p^\infty})^*$ is naturally trivialized and is isomorphic to $\CO_C \frac{dt}{t}$. Hence $\alpha_G$ is defined as $f\mapsto (\operatorname{Lie} f^\vee)^*(\frac{dt}{t})$.

Let $\{(T,W)\}$ be the category of pairs consisting of a finite free $\Z_p$-module $T$ and $W\subset T\otimes_{\Z_p}C(-1)$ is a sub-$C$-vector space. A morphism between two such pairs is a pair of morphisms between the $\Z_p$-modules and the sub-vector spaces, compatible with each other. The dual of $(T,W)$ is the pair $(T^*(1),W^\perp)$, with $*$ being the usual vector space dual, $(1)$ the Tate twist and $\perp$ the orthogonal complement (with respect to the natural pairing between $T\otimes_{\Z_p}C$ and $T^\ast\otimes_{\Z_p}C$). Then we have

\begin{theorem}{\cite[Theorem B, 5.2.1]{SW}}
The category of $p$-divisible groups over $\CO_C$ is equivalent to the above category $\{(T,W)\}$ via: 
\[\Psi: \G \mapsto (T_p\G(\CO_C), \operatorname{Lie}\G\otimes_{\CO_C}C),\] 
where $\operatorname{Lie}\G\otimes_{\CO_C}C$ is viewed as a subspace of $T_p\G(\CO_C)\otimes_{\Z_p}C(-1)$ via the Hodge-Tate filtration $\alpha^*_{G^*}$. This equivalence is compatible with duality.
\end{theorem}

\subsection{Complements}
We record below some descent properties of $p$-divisible groups.

\begin{lemma}\label{milnor}
    Given a cartesian diagram of rings
    \[\begin{tikzcd}
        R\ar[r]\ar[d] & R_2 \ar[d]\\
        R_1\ar[r, two heads]     & R_3
    \end{tikzcd}\]
    such that $R_1\rightarrow R_3$ is surjective (a Milnor square), the corresponding diagram of categories of finite projective modules over these rings is 2-cartesian, i.e. the category of finite projective modules over $R$ is equivalent to that of ``gluing triples” 
    \[(M_1,M_2,\alpha: M_1\otimes_{R_1}R_3 \xrightarrow{\sim} M_2\otimes_{R_1}R_3),\]
    where $M_i$ is a finite projective module over $R_i$ for $i=1,2$, and $\alpha$ is an isomorphism between their base changes.
\end{lemma}
\begin{proof}
    Given a finite projective module over $R$, one can construct a gluing triple by base changing to $R_i$, $i=1,2,3$ and the isomorphism $\alpha$ is the identity. Conversely, given a gluing triple, one can get an $R$-module by taking the kernel of the difference map
    \[M_1\oplus M_2\xrightarrow{\alpha - id} M_2\otimes_{R_1} R_3.\]
    That this gives the desired equivalence follows from \cite[Theorem 2.1-2.3]{Milnor}. 
\end{proof}

\begin{example}
    Let $C$ be a complete algebraically closed non-archimedean field with ring of integers $\CO_C$ and $C^+$ is a bounded valuation subring of $\CO_C$. Let $k$ be its residue field, $\overline{C^+}$ be the image of $C^+$ in $k$ and $\varpi\in C^+$ be a pseudo-uniformizer of $C$. Taking $R=C^+$, $R_1=\CO_C$, $R_2=\overline{C^+}$ and $R_3=k$ gives a Milnor square. Similarly, taking $R=C^+/\varpi\cdot \CO_C$, $R_1=\CO_C/\varpi$, $R_2=\overline{C^+}$ and $R_3=k$ gives a Milnor square. These examples will be used in the proof of Proposition \ref{cartesian}.
\end{example}

\begin{prop}\label{BTgluing}
    Let $A$ be a ring. Denote the category of $p$-divisible groups on $\Spec(A)$ by $BT(A)$ (``BT" stands for Barsotti-Tate). Given a Milnor square as in Lemma \ref{milnor}, we have $BT(R)$ is the 2-cartesian product of $BT(R_1)$ and $BT(R_2)$ over $BT(R_3)$.
\end{prop}
\begin{proof}
    We have a functor 
    \[BT(R)\rightarrow BT(R_1)\times_{BT(R_3)}BT(R_2)\] 
    by base changes. Given two $p$-divisible groups $\G,\G'$. We have
    \[\mathrm{Hom}_R(\G,\G')=\varprojlim_n \mathrm{Hom}_R(\G[p^n],\G'[p^n]).\]
    For each $n$, $\mathrm{Hom}_R(\G[p^n],\G'[p^n])$ is given by maps between the $R$-modules $\CO(G'[p^n])$ and $\CO(G[p^n])$, respecting the Hopf algebra structures on both sides. Since the Hopf algebra structures are given by morphisms of $R$ modules, we conclude by Lemma \ref{milnor} that giving such a map is equivalent to giving a pair of maps on the restrictions of $G[p^n]$ and $G'[p^n]$ to $R_1$ and $R_2$, identical on $R_3$. Passing to the inverse limit, this shows full-faithfulness. Essential surjectivity follows from a similar reasoning. Namely given a gluing triple of $p$-divisible groups, restricting to $p^n$-torsion points for each $n$, we can first recover the ring of functions $\CO(\G[p^n])$ as an $R$-module by Lemma \ref{milnor} and then endow it with a Hopf algebra structure. The full-faithfulness ensures that this will define a $p$-divisible group which restricts to the correct thing.
\end{proof}

\begin{lemma}\label{lemma: ModuleProduct}
    For $R=\prod_iV_i$ being a product of valuation rings (or fields) and $n$ be an integer, the category of rank $n$ projective modules over $R$ is equivalent to the collection of those over each $V_i$. 
\end{lemma}
\begin{proof}
    We have a functor from rank $n$ projective modules on $R$ to those on each $V_i$ by base changes. Conversely, given a collection of rank $n$ projective modules $M_i$ over each $V_i$, we can take the product $M:=\prod_i M_i$. Since each $M_i$ is necessarily free and we can pick a basis $e_{i1},\dots, e_{in}$ of it, the product $M$ is also free and trivialized by $(e_{i1})_i, \dots, (e_{in})_i$. This gives a functor in the opposite direction. Clearly they are inverses to each other. 
\end{proof}

\begin{cor}\label{BTproduct}
    The category of $p$-divisible groups of a fixed height over $R$ is equivalent to the collection of those over each $V_i$. 
\end{cor}
\begin{proof}
    Using Lemma~\ref{lemma: ModuleProduct}, one can again reason by first truncating the $p$-divisible groups and then saying that the Hopf algebra structure on the rings of functions of each $p^n$-torsion subgroup is defined using maps between the underlying (finite projective) modules.
\end{proof}

\subsection{Dieudonn\'e modules} 
We recall the prismatic Dieudonn\'e theory of $p$-divisible groups following \cite{AL}. The results we need over certain semiperfect and perfectoid rings, are already known to \cite{Lau} and \cite[Appendix to Lecture 17]{Berkeley}.

For a $p$-divisible group $G$ over a quasi-syntomic ring $R$, Ansch\"utz and Le Bras defined its prismatic Dieudonn\'e crystal as a sheaf on the small quasi-syntomic site of $R$, see \cite[Section 4]{AL}. If $R$ is quasi-regular semiperfectoid (qrsp), giving this crystal is the same as giving the evaluation of its associated sheaf on the prismatic site of $R$ at the initial prism $(\prism_R,I)$, which is called the \textit{prismatic Dieudonn\'e module} of $G$. It is a finite locally free $\prism_R$ module and is equipped with an endomorphism $\varphi_M:\varphi^\ast M\rightarrow M$, admissible in the sense of \cite[Proposition 4.9]{AL} (where $\varphi$ is the Frobenius on $\prism_R$). The so-defined prismatic Dieudonn\'e module is contravariant in $G$. To be consistent with \cite{Berkeley}, we use the \textbf{covariant} prismatic Dieudonn\'e module and denote it by $M_{\prism}(G)$. This is obtained by applying $\mathrm{Hom}_{\prism_R}(-,\prism_R)$ to the contravariant one. The following theorem will be crucial to us.

\begin{theorem}[{Cf. \cite[Proposition 4.73, 4.12]{AL}}]\label{BTDieudonne}
    Let $R$ be a qrsp ring. The covariant prismatic Dieudonn\'e module functor sending $G$ to $M_{\prism}(G)$ is an equivalence between the category of $p$-divisible groups over $R$ and that of admissible Dieudonn\'e modules over $R$.
\end{theorem}

\begin{example}{(\cite[Example 3.20]{AL})}
    An integral perfectoid ring, or a $p$-complete bounded $p^\infty$-torsion quotient of a perfectoid ring by a finite regular sequence, is qrsp. In particular, for a perfectoid Tate ring $(R,R^+)$ with $\varpi\in R^+$ a pseudo-uniformizer of $R$, both $R^+$ and $R^+/\varpi$ are qrsp.
\end{example}

\begin{example} 
    Let $R$ be an integral perfectoid ring, then
    \[(\prism_R,I)=(W(R^\flat),\mathrm{ker}(\theta\circ \varphi_R^{-1})),\]
    where $\theta$ is Fontaine's theta map. In this case an admissible prismatic Dieudonn\'e module is the same as a minuscule Breuil-Kisin-Fargues module with a leg at $V(\mathrm{ker}(\theta\circ \varphi_R^{-1}))$ (see Definition \ref{BKF}), and the construction of the covariant Dieudonn\'e module agrees with the construction in \cite[Theorem 17.5.2]{Berkeley}, see \cite[Proposition 4.48]{AL}.
\end{example}

\begin{example}\label{example-crystallineDieu}
    Let $R$ be qrsp and $pR=0$ (e.g. the integral subring in a Tate perfectoid ring modulo a pseudo-uniformizer), then 
    \[(\prism_R,I)=(A_{\cris}(R),(p)).\]
    In this case the covariant Dieudonn\'e module agrees with the (naive dual of the contravariant) crystalline Dieudonn\'e module of Berthelot, Breen and Messing, \cite[Lemma 4.45]{AL}. 
\end{example}
\hfill

\section{Serre-Tate theory}
	As the main input of the fiber product description of the Shimura variety, we need the following theorems regarding the deformation of abelian schemes, originally due to Serre-Tate, Messing and Drinfeld. 

\begin{theorem}[{\cite[Theorem 2.4.1]{CS19}}]\label{ST0}
Let $S'\twoheadrightarrow S$ be a surjection of rings in which p is nilpotent, with kernel $I\subset S'$ being a nilpotent ideal.
	\begin{enumerate}
	\item The functor $\G_{S'}\mapsto \G_{S'}\times_{S'}S$ from $p$-divisible groups up to isogeny over $S'$ to $p$-divisible groups up to isogeny over S is an equivalence of categories.
	\item The functor $A_{S'}\mapsto A_{S'}\times_{S'}S$ from abelian schemes up to $p$-power isogeny over $S'$ to abelian schemes up to $p$-power isogeny over S is an equivalence of categories.
 	\end{enumerate}
\end{theorem}

\begin{theorem}[{\cite[Theorem 2.4.2]{CS19}}(Serre-Tate)]\label{ST1}
Let $S'\twoheadrightarrow S$ be a surjection of rings in which p is nilpotent, with kernel $I\subset S'$ being a nilpotent ideal. The functor 
\[A_{S'}\mapsto (A_S, A_{S'}[p^\infty],id)\]
is an equivalence of categories between the category of abelian schemes over $S'$ and the category of triples consisting of an abelian scheme $A_S$ over $S$, a $p$-divisible group $\G_{S'}$ over $S'$ and an isomorphism $\rho:A_S[p^\infty]\rightarrow \G_{S'}\times_{S'}S$. 
\end{theorem}

\hfill

\section{Shimura varieties}
    We introduce the PEL-type Shimura data and their associated Shimura varieties this paper concerns, following Kottwitz \cite[Section 5]{Kottwitz}, cf. \cite{Lan13}.

\subsubsection{Global PEL setup}

\begin{defn}
    A global PEL-datum is a tuple 
    \[(B,\ast,V,(\cdot,\cdot),h),\]
    where
    \begin{itemize}
        \item $B$ is a finite-dimensional semisimple $\Q$-algebra.
        \item $\ast$ is a positive involution on $B$, i.e. over $\R$, $\ast$ satisfies that $\mathrm{tr}_{B_\R/\R}(xx^\ast)>0$ for all $0\neq x\in B_\R$.
        \item $V$ is a finite left $B$-module.
        \item $(\cdot,\cdot)$ is a nondegenerate $\Q$-valued alternating form on $V$ such that $(bv,w)=(v,b^\ast w)$ for all $v,w\in V$ and $b\in B$. In particular, the induced involution on $\mathrm{End}(V)$ that sends an endomorphism to its adjoint with respect to $(\cdot,\cdot)$ extends $\ast$ on $B\subset \mathrm{End}(V)$. 

        Let $G/\Q$ be the algebraic group determined by the functor:
        \[R \mapsto \{x \in \mathrm{End}_{B\otimes R}(V\otimes R)\mid xx^*\in R^\times\}.\]
        \item $h: \mathbb{S}\rightarrow G_\R$ is a homomorphism, such that $h(\bar{z})=h(z)^\ast$ for any $z\in \C$, the symmetric real-valued bilinear form $(v,h(i)w)$ on $V_\R$ is positive-definite, and the induced Hodge structure on $V_\R$ is of type $(-1,0)$, $(0,-1)$.\footnote{This means $V_\C$ decomposes into a direct sum of two subspaces where the $h(z)$ action is by $z$ and $\bar{z}$ respectively, for all $z\in \C$. Here we follow Deligne's convention in \cite{Deligne79} that we think of the Hodge structure corresponding to $h$ as coming from the de Rham homology of an abelian variety, hence having negative weights.}
    \end{itemize}
\end{defn}

Let $X$ be the $G(\R)$-conjugacy class of $h$. Then the pair $(G,X)$ is a Shimura datum. Let $V_\C\cong V_1\oplus V_0$ be the $B_\C$-module decomposition induced by $h$ such that $h(z)$ acts on $V_1$ (resp. $V_0$) by $z$ (resp. $\bar{z}$).\footnote{Note that in \cite{Deligne79}, the subspaces $V_1$ and $V_0$ would be denoted $V^{-1,0}$ and $V^{0,-1}$ respectively, which is meant to indicate that in the cohomological convention of the Hodge decomposition, these are the subspaces of type $(-1,0)$, $(0,-1)$. Our justification for the notation here is that we use a subscript to indicate that we think of $V$ as the homology, while the number $1$ and $0$ are the weights of the Hodge cocharacter attached to $h$, in the usual representation theoretic sense.} Let $E_0$ be the field of definition of the complex representation $V_1$ of $B$, i.e.
\[E_0=\Q[\{\mathrm{tr}(b\mid V_1)\}_{b\in B}].\]
Then the reflex field $E(G,X)=E_0$.

If $B$ is simple, its center $F$ is a field and $F^+:=F^{\ast=id}$ is a totally real subfield. Let $G_1/\Q$ be the closed subgroup of $G$ defined by
\[R\mapsto\{x\in \mathrm{End}_{B\otimes R}(V \otimes R)\mid xx^{\ast} = id\}.\]
Then it is the restriction of scalar of some group $G_0/F^+$ from $F^+$ to $\Q$. According to the type of the extension $F/F^+$ and $G_0$, the PEL-datum falls into three families, cases $A$, $C$ and $D$, where respectively $F/F^+$ is a complex quadratic extension, $G_0$ is an inner form of the quasi-split unitary group over $F^+$ (of type $A_{n-1}$); $F=F^+$ is totally real, $G_0$ is a symplectic group in $2n$ variables; and $F=F^+$, $G_0$ is an orthogonal group of $2n$ variables.  Here $n=\tfrac{[F:F^+]}{2}(\mathrm{dim}_{F}\mathrm{End}_{B}(V))^{\frac{1}{2}}$ is forced to be even by the existence of $h$. 

In general the semisimple $\Q$-algebra $B$ decomposes into a product of simple algebras. According to \cite[Lemma 1.2.1.11]{Lan13}, the involution $\ast$ leaves stable each simple factor. Hence the symplectic $B$-module $(V,(\cdot,\cdot))$ decomposes accordingly. Up to similitude factors, $G$ is the product of groups as $G_1$ above. 

From now on, we will make the following additional assumptions on the PEL-datum and on the prime $p$:
\begin{assumption}\label{assumption}\hfill
\begin{enumerate}[leftmargin=0.7cm]
    \item (Type AC) In the decomposition of $B$ into simple factors, no factor of type $D$ appears.
    \item $B_{\Q_p}$ is a product of matrix algebras over unramified extensions of $\Q_p$. There exists a $\ast$-invariant $\Z_{(p)}$-order $\CO_B \subset B$, whose $p$-adic completion is a maximal $\Z_p$-order of $B_{\Q_p}$, and a $\Z_{(p)}$-lattice $\Lambda_0 \subset V$, stable under the $\CO_B$-action and self-dual with respect to $(\cdot, \cdot)$.
\end{enumerate}
\end{assumption}

\begin{remark}\label{Hasse}
    The Hasse principle holds for groups of type $C$. For type $A$, $G_0$ above is the inner form of a quasi-split unitary group over $F^+$, determined by the quadratic extension $F/F^+$. The Hasse principle holds if the Hermitian space giving rise to the quasi-split unitary group has even dimension over $F$; otherwise it can fail but this failure comes from the failure of the Hasse principle for the center of $G$, see \cite[Section 7]{Kottwitz}. But for type $D$, the Hasse principle fails in a more essential way. The reason we put assumption $(1)$ is to ensure that the moduli problem we will consider below will be a union of copies of Shimura varieties given by the PEL-datum. Involving type $D$ factors destroys this feature due to failure of the Hasse principle, see \cite[Example A.7.2]{Lan15}. Our assumption $(2)$ on the prime $p$ ensures a smooth integral structure at $p$. In particular, hyperspecial subgroups exist, or equivalently the group $G_{\Q_p}$ is quasi-split and splits over an unramified extension of $\Q_p$, see below.
\end{remark}

Fix $\CO_B$ and a self-dual $\CO_B$-lattice $\Lambda_0 \subset V$ as in Assumption \ref{assumption}(2). These determine a connected reductive group $G_{\Z_{(p)}}$ over $\Z_{(p)}$ with generic fiber $G$ as 
\[G_{\Z_{(p)}}(R)=\{x\in \mathrm{End}_{\CO_B\otimes_{\Z_{(p)}} R}(\Lambda_0\otimes_{\Z_{(p)}} R)\mid xx^*\in R^\times\}.\]


\subsubsection{Moduli interpretation}
Let $(B,\ast,V,(\cdot,\cdot),h)$ be a global PEL-datum satisfying Assumption \ref{assumption}, $(\CO_B,\ast, \Lambda_0, (\cdot,\cdot),h)$ its integral model at $p$, and $G_{\Z_{(p)}}$ as above. Let $\Lambda$ be the $p$-adic completion of $\Lambda_0$. We fix the hyperspecial subgroup $K_p=G_{\Z_{(p)}}(\Z_p) \subset G(\Q_p)$. Let $K^p \subset G(\mathbb{A}_f^p)$ be a compact open subgroup and $K=K_pK^p$. We can define a moduli stack of polarized abelian varieties with endomorphism by $\CO_B$ at level $K$, over the localization of $\CO_{E_0}$ at some prime above $p$. By Remark \ref{Hasse}, its generic fiber is a finite disjoint union of copies of the Shimura variety attached to the given PEL-datum. The number of copies agrees with the cardinality of the set of locally trivial elements in $H^1(G,\Q)$ see\cite[Section 8]{Kottwitz}. We ignore this difference below and call the moduli space the Shimura variety. 

\begin{defn}\label{abelianVar}
	Let $S$ be a scheme over $\CO_{E_0}\otimes_\Z \Z_{(p)}$, where $\Z_{(p)}$ is the localization of $\Z$ at $p$. An abelian scheme with $G$-structure at level $K$ over $S$ is a quadruple $\AV=(A,\iota, \lambda, \bar{\eta})$ where:
	\begin{itemize}[leftmargin=0.8cm]
 	\item $A$ is an abelian scheme of dimension $g=\frac{1}{2}\text{dim}_\Q V$ over $S$;
	\item $\iota:\CO_B\rightarrow \mathrm{End}(A)\otimes \Z_{(p)}$ is an $\CO_B$-action, satisfying the Kottwitz condition that 
\[\mathrm{det}_{\CO_S}(\iota(b)\mid Lie(A))=\mathrm{det}(b\mid V_1),\]
for all $b\in \CO_B$, where $V_\C \cong V_1 \oplus V_0$ is the decomposition such that $h(z)$ acts on $V_1$ (resp. $V_0$) by $z$ (resp. $\bar{z}$);
\footnote{In other words, the determinant of any element $b \in \CO_B$ acting on the Lie algebra $Lie(A)$ (as a free $\CO_S$-module) of $A$ agrees with the determinant of it acting on $V_1$. This makes sense as the decomposition $V_\C \cong V_1 \oplus V_0$ is defined over $E_0$ and the determinant of $b$ lies in $\CO_{E_0}\otimes \Z_{(p)}$. As remarked by \cite[Section 5]{Kottwitz}, for a point $s\in S$ with residue field $K/E_0$, this condition ensures that $\mathrm{Lie}(A_s)\cong V_{1,K}$ as $B_K$-modules, where $V_K \cong V_{1,K} \oplus V_{0,K}$ is a $K$-vector space decomposition whose base change to $\C$ is the above. In this way, $V_\C$ can be identified with the Betti homology $H_1(A_\C,\C)$ preserving the Hodge structures.}
	\item $\lambda: A\rightarrow A^\vee$ is a prime-to-p quasi-isogeny, symmetric with respect to the double duality $A\cong A^{\vee\vee}$, such that for some natural number $n$, $n\lambda$ is induced by an ample line bundle on $A$ (hence pointwise a polarization) and whose Rosati involution on $\mathrm{End}(A)\otimes \Z_{(p)}$ is compatible with $\ast$ on $\CO_B$ via $\iota$;
	\item $\bar{\eta}$ is a $K^p$-orbit of a chosen trivialization $\eta$ of the locally constant pro-\'etale\footnote{In the sense of \cite{BS15}.} sheaf $\underline{H_1}(A,\Af^p)$ on $S$, under the action of $\underline{G(\Af^p)}$. Namely, the sheaf
\[\underline{\mathrm{Isom}}_{G}(\underline{H_1}(A,\Af^p), \underline{V_{\Af^p}}),\]
whose sections are $B\otimes_\Q \Af^p$-module isomorphisms that preserve $(\cdot,\cdot)$ up to a scalar in $\underline{\Af^{p,\times}}$, is a $\underline{G(\Af^p)}$-torsor on $S_{\text{pro-\'et}}$. Choose one section $\eta$ of it on a trivializing cover $\tilde{S}\rightarrow S$ and look at the $\underline{G(\Af^p)}(\tilde{S})$-action on it. Then $\bar{\eta}$ is its orbit under the subgroup $\underline{K^p}(\tilde{S})$. We further require that $\bar{\eta}$ is invariant under the action of the covering group of $\tilde{S}\rightarrow S$. 
	\end{itemize}
\end{defn}

\begin{defn}
	Let $S_K^\mathrm{pre}$ be the presheaf of groupoids on the big \'etale site of schemes over $\CO_{E_0}\otimes_\Z \Z_{(p)}$, whose value on $S$ is the groupoid of abelian schemes over $S$ with $G$-structure at level $K$, and an isomorphism between $(A,\iota, \lambda, \bar{\eta})$ and $(A',\iota', \lambda', \bar{\eta'})$ is a prime-to-$p$ quasi-isogeny $f: A\rightarrow A'$, such that $f^\vee \circ \lambda' \circ f = c\lambda$, for some $c\in \underline{\Z_{(p)}^\times}(S)$, where $f^\vee: A'^\vee\rightarrow A^\vee$ is the dual quasi-isogeny, $f$ commutes with the action of $\CO_B$ on $A$ and $A'$ via $\iota, \iota'$, and $\bar{\eta}=\bar{\eta'}\circ f_\ast$. 
\end{defn}

This moduli problem is a Deligne-Mumford stack. For small enough $K^p$, it is representable by a smooth quasi-projective scheme, see \cite[Section 5]{Kottwitz}. We will always work in such situations. Let $E$ be the completion of $E_0$ at some prime $\mathfrak{p}$ above $p$ and $\CO_E$ its ring of integers. We base change the moduli functor to $\CO_E$ and denote the representing scheme by $S_K$. Its $p$-adic completion is denoted by $\mathscr{S}_{K}$. This is a formal scheme over $\Spf \CO_E$.


\begin{remark}\label{G-p-div}
    Let $S$ be an $\CO_E$-scheme. We call a tuple $(\G,\iota, \lambda)$ a $p$-divisible group with $G$-structure over $S$, where
    \begin{itemize}[leftmargin=0.8cm]
        \item $\G$ is a $p$-divisible group over $S$,
        \item $\iota: \CO_B\otimes \Z_p\rightarrow \mathrm{End}(\G)$ is
        a $\Z_p$-linear map satisfying the Kottwitz condition \[\mathrm{det}_{\CO_S}(\iota(b)\mid \mathrm{Lie}(\G))=\mathrm{det}(b\mid V_1\otimes_{E_0} E),\]
        \item $\lambda: \G\rightarrow \G^\vee$ is a polarization, satisfying for any $b\in \CO_B\otimes \Z_p$, $\lambda^{-1}\iota(b)^\vee \lambda=\iota(b^\ast)$. 
    \end{itemize} 
    An isomorphism (resp. quasi-isogeny) between $(\G,\iota, \lambda)$ and $(\G',\iota', \lambda')$ is an $\CO_B\otimes \Z_p$-linear isomorphism (resp. quasi-isogeny) $f:\G \rightarrow \G'$ such that $f^\vee\circ \lambda'\circ f= c\lambda$ for some $c\in \underline{\Z}_p^\times(S)$ (resp. $\underline{\Q}_p^\times(S)$).
    
    Taking the $p$-divisible group of an abelian scheme defines a functor from abelian schemes with $G$-structure up to (prime-to-$p$) quasi-isogenies to $p$-divisible groups with $G$-structure up to (isomorphisms) quasi-isogenies.
\end{remark}



\begin{example}\label{unitary}(Non-compact unitary Shimura varieties appeared in \cite[Section 2.1]{CS19}.)
Let $F$ be a CM field with totally real subfield $F^+ \subset F$ and $n\geq 1$ be an integer. Then we can take $B=F$, $*=$CM conjugation, $V=F^{2n}$, with alternating form
\[(\cdot,\cdot): V\times V \rightarrow \Q\]
\[((x_i),(y_i)) \mapsto \mathrm{tr}_{F/\Q}(\Sigma_{i=1}^{n}(x_i\bar{y}_{2n+1-i}-x_{2n+1-i}\bar{y}_i)).\]
The reductive group $G$ is a unitary similitude group and $G_\R\cong \mathrm{GU}(n,n)^{[F^+:\Q]}$.
\[X= \prod_{\tau: F^+\rightarrow \R}X_{\tau,+}\sqcup \prod_{\tau: F^+\rightarrow \R}X_{\tau,-},\]
where $X_{\tau,+}$ (resp. $X_{\tau,-}$) is the space of positive (negative) definite $n$-dimensional subspaces of $V\otimes_{F^+}\R \cong \C^{2n}$, each being isomorphic to the Hermitian upper (lower) half-space 
\[\mathcal{H}_{n,n}=\{A\in\mathrm{Herm}_n(\C)\otimes_\R \C: \mathrm{Im}(A)>0\}\]
\[(\mathcal{H}_{n,n}^{-}=\{A\in\mathrm{Herm}_n(\C)\otimes_\R \C: \mathrm{Im}(A)<0\}),\]
where $\mathrm{Herm}_n(\C)$ is the set of $n$-by-$n$ Hermitian matrices (see \cite[Section 3.2.5]{Lan16}), and

\[h=\prod_{\tau: F^+\rightarrow\R} h_\tau: \mathbb{S}\rightarrow G_{\R}, z\mapsto 	
(\mathrm{diag}\{z\cdot I_n, \overline{z}\cdot I_n\}_\tau)_{\tau:F^+\rightarrow \R}.\]

This is a moduli problem of type $A$, and the corresponding (unitary) Shimura variety is non-compact, since the group $G$ is quasi-split and has rationally defined parabolic subgroups. When $n=1$, $G_\R=\mathrm{GU}(1,1)$, the Shimura variety is one dimensional and we call it a unitary Shimura curve attached to the imaginary quadratic field $F$.
\end{example}

\subsubsection{Minimal compactifications}
Let $K=K_pK^p$ with $K_p$ hyperspecial as before, the smooth quasi-projective scheme $S_K$ over $\CO_E$ has a good minimal compactification, whose properties we summarize below. For more details, see \cite[Section 7.2.4]{Lan13}, \cite[Proposition 2.1.2]{LS}, cf. \cite[Theorem 2.5.8]{CS19}. 

\begin{theorem}
There exists a flat, projective, normal scheme $S_K^\ast$ over $\Spec(\CO_E)$, together with a set-theoretic partition into locally closed subschemes
\[S_K^\ast=\coprod_Z S_{K,Z},\]
where the (finite) index set is endowed with a partial order such that the incidence relations among strata are determined by this partial order. There is a unique dense open stratum that is isomorphic to $S_K$. If the level $K$ is principal, i.e. it is the kernel of the reduction by $N$ map on $G_{\Z_p}(\Z_p)$, for some integer $N$ coprime to $p$, then each $S_{K,Z}$ is a PEL-type Shimura variety. 
\end{theorem}
\begin{remark}\label{MinLevel}
    Over $E$ we can add levels at $p$ and the same statements hold. In this case, the scheme $S_{K,E}$ is the solution to the moduli problem of abelian schemes with $G$-structures at level $K$, where the level $\bar{\eta}$ is a $K$-orbit of trivializations of $\underline{H}_1(A,\Af)$ under the action of $\underline{G(\Af)}$.
\end{remark}

\begin{remark}
 As in the literature, we will refer to elements in the index set as \textit{cusp labels at level $K$}. In general, a cusp label is a tuple of the shape $(Z,(X,Y,\phi, \varphi_{-2},\varphi_0))$ consisting of the following data. see \cite[Definition 5.2.7.1, 5.4.1.3]{Lan13} 
    \begin{enumerate}
         \item $Z$ is a $\CO_B\otimes_{\Z} \hat{\Z}$-stable split two step filtration \[0=Z_{-3}\subset Z_{-2}\subset Z_{-1}\subset Z_0=\Lambda_0\otimes_{\Z}\hat{\Z},\]
         such that each graded piece is isomorphic to $M\otimes_{\Z_{(p)}}\hat{\Z}$ for some finitely generated $\CO_B$-torsionfree $\CO_B$-module (an $\CO_B$-lattice) $M$, and that $Z_{-2}$ and $Z_{-1}$ are annihilators of each other under the pairing $(\cdot,\cdot)$ induced from $\Lambda_0$;
         \item $X$ and $Y$ are $\CO_B$-lattices of the same $\CO_B$-multi-rank\footnote{Assume $B\cong \prod_i B_i$ is a decomposition of $B$ into simple $\Q$-algebras, then each finite $B$-module $M$ decomposes as $\prod_i M_i^{m_i}$, with $M_i$ being the unique simple left $B_i$-module. Then the vector $(m_i)$ is called the $B$-multi-rank of $M$. And for an $\CO_B$-lattice $M$, its $\CO_B$-multi-rank is the $B$-multi-rank of $M\otimes \Q$, see\cite[Definition 1.2.1.21]{Lan13}.} and $\phi: Y\rightarrow X$ is an $\CO_B$-linear injection;
         \item Denote $\Z_{(p)}\otimes_{\Z}\hat{\Z}$ by $R$, then $\varphi_{-2}: \Gr_{-2}^Z\cong \Hom_R(X\otimes_{\Z}\hat{\Z},R(1))$ and $\varphi_0: \Gr_{0}^Z\cong Y\otimes_{\Z}\hat{\Z}$ are isomorphisms such that the induced pairing 
         \[(\cdot,\cdot)_{20}: \Gr_{-2}^Z \times \Gr_{0}^Z\rightarrow R(1)\]
         is the pullback under $(\varphi_{-2},\varphi_0)$ of the pairing:\footnotesize
         \[\Hom_R(X\otimes_{\Z}\hat{\Z},R(1))\times (Y\otimes_{\Z}\hat{\Z}) \xrightarrow{id\times \phi} \Hom_R(X\otimes_{\Z}\hat{\Z},R(1))\times (X\otimes_{\Z}\hat{\Z})\xrightarrow{ev} R(1),\]
         \normalsize where the last map is the evaluation pairing.
    \end{enumerate}
There is an action of the group $K_pG(\Af^p)$ on $\Lambda_0\otimes_{\Z}\hat{\Z}$, inducing an action on the set of cusp labels. A cusp label at level $K$ is a $K$-orbit of cusp labels under this action.
\end{remark}
 \begin{remark}
    Given $Z$, a cusp label at level $K$, assume $\Gr_{-1}^Z\cong M\otimes_{\Z_{(p)}}\hat{\Z}$ for some $\CO_B$-lattice $M$. Then the stratum $S_{K,Z}$ is attached to the integral PEL Shimura datum $(\CO_B,*,M, (\cdot,\cdot)_{11}, h_{-1})$ (see \cite[Proposition 5.1.2.2]{Lan13} for the definition of $h_{-1}$). For an abelian variety corresponding to a $C$-point of $S_{K,E}$ for some complete algebraically closed non-archimedean field $C$, it has semistable reduction over the ring of integers $\CO_C\subset C$ and hence an attached Raynaud extension
    \[0\rightarrow T\rightarrow E\rightarrow B\rightarrow 0\]
    i.e. an extension of a smaller dimensional polarized abelian scheme $B$ by a torus $T$, both equipped with $\CO_B$-endomorphism. Then $S_{K,Z}$ is a parameter space for such $B$'s. In fact $X$ is obtained from the character group\footnote{In \cite{CS19} page 23, $X$ is said to be the cocharacter group instead of the character group and this is a slight inconsistency with the explanations there.} of $T$ (tensored up to $\Z_{(p)}$), $Y$ from that of the dual Raynaud extension, and the filtered pieces of $Z$ are obtained by taking the Tate module of $T$ and that of $E$. For details, see \cite[Section 3.3, 4.2]{Lan13}, cf. \cite[Section 2.5.1]{CS19} in the principally polarized case.
 \end{remark}

\subsubsection{Shimura variety as v-sheaves}
Let $K$, $S_K/\CO_E$ and $\mathscr{S}_{K}$ over $\Spf(\CO_E)$ be as before. We define below variants of the Shimura variety as v-sheaves for later use.

\begin{defn}
    The \textit{adic Shimura variety} at level $K$ is the diamond $\Shi_K$ over $\Spd E$ attached to $S_{K,E}$ by the big diamond functor, see Definition \ref{BigDiamond}, namely $\Shi_K= S^\Diamond_{K,E}$.
\end{defn}

By analytifying the universal abelian scheme over $S_{K,E}$ and passing to diamonds, we obtain a proper map of diamonds $\pi: \AV^\diamond\rightarrow \Shi_{K}$.

Equip $\Z_p$ with its usual $p$-adic topology. Let $\underline{\Z}_p$ be the v-sheaf on $\mathcal{A}^\diamond$ attached to the profinite space $\Z_p$. The sheaf of $\underline{\Z}_p$-modules on $\Shi_K$
\[T_p\mathcal{A}:=\underline{\mathrm{Hom}}_{\underline{\Z}_p}(R^1\pi_\ast\underline{\Z}_p,\underline{\Z}_p)\]
is called the Tate module of the universal object. 

\begin{defn}
     The \textit{Shimura variety with infinite level at $p$} is the diamond $\Shi_{K^p}$ over $\Shi_K$ of $\CO_B$-linear trivializations of $T_p\AV$, which preserve the alternating paring $(\cdot,\cdot)$ up to a constant in $\underline{\Z}_p^\times$, namely
    \[\Shi_{K^p}= \underline{\mathrm{Isom}}_G(T_p\mathcal{A}, \underline{\Lambda})\rightarrow \Shi_K.\]
\end{defn}
\begin{remark}
    Alternatively this is the limit in the category of diamond of $\Shi_K$'s for $K_p$ running over compact open subgroups of $G(\Q_p)$. It is in fact representable by a perfectoid space by the work of \cite{Sch15}.
\end{remark}

\begin{defn}[Cf. {\cite[Definition 5.19, Section 5.2]{MiedaImai}}]
    The \textit{good reduction locus} of the adic Shimura variety at level $K$ is the diamond $\Shi_K^\circ$ attached to the adic generic fiber of the formal scheme $\mathscr{S}_{K}$, i.e. 
    \[(\mathscr{S}_K^\text{ad}\times_{\Spa(\CO_E,\CO_E)}\Spa(E))^\diamond.\]
    This is a spatial diamond over $\Spd E$. 
\end{defn}

The diamond $\mathcal{S}_{K}^\circ$ still has a moduli interpretation in the following sense:
\begin{lemma}\label{PresheafGRL}
The diamond good reduction locus $\mathcal{S}_{K}^\circ$ is the sheafification with respect to the analytic topology of the presheaf on $\Perf$
\[S=\Spa(R,R^+)\mapsto \{(S^\sharp, \Spf R^{\sharp+}\rightarrow\mathscr{S}_K)\},\]
where $S^\sharp=\Spa(R^\sharp,R^{\sharp+})$ is an untilt of $S$ over $E$.
\end{lemma}
\begin{proof}
   Combine \cite[Theorem 10.1.5]{Berkeley} and \cite[Proposition 2.2.2]{SW}.
\end{proof}

\begin{defn}
     The \textit{good reduction locus with infinite level at $p$} is the diamond $\Shi_{K^p}^\circ$ over $\Spd E$, obtained by pulling back $\Shi_K^{\circ}$ to $\Shi_{K^p}$.
\end{defn}

\begin{defn}
     The \textit{(v-sheaf) integral model} of the Shimura variety at level $K$ is the v-sheaf $\SK$ attached to the formal scheme $\mathscr{S}_K$, or alternatively this is the small diamond functor applied to the $\CO_E$-scheme $S_K$, see Example \ref{vsheafFormal}, \ref{BigDiamond} and \cite[Remark 2.11]{AGLR}.
\end{defn}

As explained in Remark \ref{MinLevel}, over $E$ we have minimal compactifications $S_{K_pK^p,E}^\ast$ for Shimura varieties with deeper levels at $p$.
\begin{defn}
The \textit{minimal compactification with infinite level at $p$} is the diamond $\Shi_{K^p}^\ast$ over $\Spd E$
\[\Shi_{K^p}^\ast:= \varprojlim_{K_p}S_{K_pK^p,E}^{\ast, \diamond},\]
where the limit is taken over compact open subgroups $K_p\subset G(\Q_p)$.   
\end{defn}

\hfill

\section{\texorpdfstring{$\Bdr^+$}{}-affine Grassmannian and Hodge-Tate period map}
     \subsection{\texorpdfstring{$\Bdr^+$}{}-affine Grassmannian}

Let $G$ be a reductive group over $\Q_p$. Recall the following definition from \cite[Section 19.1]{Berkeley}.

\begin{defn}
    The $\Bdr^+$-affine Grassmannian $\Gr_G$ is the v-sheaf on $\Perf/\Spd \Q_p$ sending $S=\Spa(R,R^+)$ to the set of isomorphism classes of pairs $(\mathcal{F}, \alpha)$, where $\mathcal{F}$ is a $G$-torsor over $\Bdr^+(R^\sharp)$, $\alpha$ is a trivialization of $\mathcal{F}$ over $\Bdr(R^\sharp)$, and $S^\sharp= \Spa(R^\sharp, R^{\sharp+})$ is the untilt of $S$ over $\Spa \Q_p$ determined by the map $S\rightarrow \Spd\Q_p$. Alternatively, this is the \'etale sheafification\footnote{According to \cite{Kestutis}, analytic sheafification suffices.} of the presheaf sending $S\rightarrow \Spd \Q_p$ to the set $G(\Bdr(R^\sharp))/G(\Bdr^+(R^\sharp))$.
\end{defn}

Over an algebraically closed non-archimedean extension $C/\Q_p$, fixing a split torus and a Borel $T\subset B\subset G_C$, we have the Cartan decomposition 
\[G(\Bdr(C))=\coprod_{\mu \in X^+_*(T)} G(\Bdr^+(C))\cdot \xi^\mu\cdot G(\Bdr^+(C)).\]
where $ X^+_*(T)$ is the set of dominant cocharacters of $T$. This defines a Bruhat stratification on $\Gr_{G,C}$, the base change of $\Gr_G$ to $\Spd C$:

\begin{defn}\label{SchubertCell}
Let $\mu\in X^+_*(T)$. Then $\Gr_{G,C,\mu}$ (respectively $\Gr_{G,C,\leq \mu}$) is the subfunctor of $\Gr_{G,C}$ sending $S/\Spd C$ to the set of maps from $S$ to $\Gr_{G,C}$ for which any geometric point $\Spa(\tilde{C},\tilde{C}^+)$ of $S$ maps to a point in the coset 
\[G(\Bdr^+(\tilde{C}))\cdot \xi^\mu\cdot G(\Bdr^+(\tilde{C}))\] (respectively in the union of cosets labeled by some $\lambda\leq \mu$ in the Bruhat order on $X^+_\ast(T)$). If the $G(C)$-conjugacy class of $\mu$ is defined over some field $E$ with $C/E/\Q_p$, then so is $\Gr_{G,\mu}$. 
\end{defn}

\subsection{Hodge-Tate period map in the PEL setup}

Let $(B,\ast, V, (\cdot,\cdot),h)$ be a global PEL-datum satisfying Assumption \ref{assumption}, and $G/\Q$, $\nu_h$, $E_0$, $E$ determined by it as in section 5. Fix an isomorphism $\C\cong \overline{\Q}_p$ over $E_0$, where $\overline{\Q}_p$ is an algebraic closure of $\Q_p$ containing $E$. Choose a maximal torus and a Borel $T\subset B\subset G_{\overline{\Q}_p}$. Let $\mu$ be a dominant cocharacter representing the $G(\overline{\Q}_p)$-conjugacy class of $\nu_h^{-1}$. 

We consider the $\Bdr^+$-affine Grassmannian $\Gr_G$ attached to $G_{\Q_p}$. The conjugacy class $[\mu]$ determines a Schubert cell $\Gr_{G,\mu} \subset \Gr_{G,E}$. Since $\mu$ is minuscule it equals $\Gr_{G,\leq\mu}$ and is proper by \cite[Proposition 19.2.3]{Berkeley}. From now on, we drop the subscript $\Q_p$ from $G_{\Q_p}$ when it is clear from the context that the situation is local at $p$.   
\begin{remark}
\label{minusculeSchubert}
Let $C/E$ be a complete algebraically closed nonarchimedean field. Using Theorem \ref{BTclassification} and a Bialynicki-Birula isomorphism \cite[Proposition 19.4.2]{Berkeley}, we can interpret $\Spd C$-valued points of $\Gr_{G,\mu}$ as parametrizing $p$-divisible groups over $\CO_C$ with additional structures:

Let $\Fl_{G,\mu}$ be the analytification of the partial flag variety $G/P_\mu$ over $E$, with $P_\mu$ being the maximal parabolic subgroup of $G$ such that for any $g\in P_\mu$, the limit 
\[\varinjlim_{t\rightarrow 0}\mu(t)^{-1}g\mu(t)\] 
exists. The Bialynicki-Birula isomorphism identifies $\Gr_{G,\mu}$ with $\Fl_{G,\mu}^\diamond$. Giving a $\Spa C$-point of $\Fl_{G,\mu}$ is equivalent to giving a $B$-equivariant filtration $W\subset V\otimes_\Q C$ by a maximal isotropic subspace with respect to the pairing $(\cdot,\cdot)$.

According to Theorem \ref{BTclassification}, this filtration, together with the self-dual lattice $\Lambda\subset V_{\Q_p}$, defines a $p$-divisible group $\G$ with trivialized Tate module $T_p\G \cong \Lambda$. This $\G$ is equipped with a polarization $\G\rightarrow \G^*$ coming from $(\cdot,\cdot): (\Lambda,W)\rightarrow (\Lambda^*(1),W^\perp), (t,w)\mapsto ((\cdot,t), (\cdot,w))$, an $\CO_B$-endomorphism coming from the $\CO_B$-module structures on $(\Lambda,W)$, and an infinite level structure coming from the trivialization $T_p\G \cong \Lambda$. 
\end{remark}

Let $K^p$ be a compact open subgroup of $G(\Af^p)$ and $\Shi_{K^p}^\circ/\Spd E$ the good reduction locus of the diamond infinite $p$-level Shimura variety. The Hodge-Tate period map of \cite{Sch15} and \cite[Theorem 2.1.3]{CS17} restricted to $\Shi_{K^p}^\circ$ can be rephrased as below.

\begin{theorem}\label{HT}
There exists a $G(\Q_p)$-equivariant Hodge-Tate period map of diamonds over $\Spd E$
\[\pi_{HT}^\circ: \Shi_{K^p}^\circ\rightarrow \Gr_{G,\mu}.\] 
It is also equivariant with respect to the natural $G(\Af^p)$-action on the inverse system $\{\Shi_{K^p}^\circ\}_{K^p}$ and the trivial action on the target. 
\end{theorem}

\begin{proof}
View $\Shi_{K^p}^\circ$ as the analytic sheafification of the presheaf $\Shi_{K^p}^{\circ,\text{pre}}$:
\[S=\Spa(R,R^+)\mapsto 
\{(S^\sharp, \mathfrak{A},\beta)\},\]
where $S^\sharp=\Spa(R^\sharp, R^{\sharp +})$ is an untilt of $S$ over $\Spa E$, $\mathfrak{A}$ is a formal abelian scheme with $G$-structures over $\Spf R^{\sharp+}$, and $\beta\in \underline{\mathrm{Isom}}_G(T_p\mathcal{A},\underline{\Lambda})(S)$ is a trivialization of the Tate module of the generic fiber $\AV$ of $\mathfrak{A}^\diamond$.

Given an $S$-point $(S^\sharp, \mathfrak{A},\beta)$ of $\Shi_{K^p}^{\circ,\text{pre}}$, write $T$ for $T_p\AV(S)$. This is a finite projective module over the ring $C^0(S,\Z_p)$ of continuous $\Z_p$-valued maps on $S$. Consider the prismatic Dieudonn\'e module
\[(M:=M(\mathfrak{A}[p^\infty]), \varphi_M).\]

By strong \'etale comparison of \cite[Theorem 1.20]{GR}, or in fact its base-change to $\Bdr(R^\sharp)$, we have a natural comparison isomorphism
\[c: T\otimes_{C^0(S,\Z_p)}\Bdr(R^\sharp)\cong M\otimes_{W(R^+)}\Bdr(R^{\sharp}),\]
compatible with the $G$-structures.

Let $\mathcal{F}$ be the \'etale sheaf of symplectic similitude $\CO_B$-linear trivializations on $X:=\Spec(\Bdr^+(R^\sharp))$
\[\underline{\mathrm{Isom}}_G (M\otimes_{W(R^+)}\CO_X, \Lambda\otimes_{\Z_p}\CO_X).\]
This is a $G_{\Q_p}$-torsor by \cite[Theorem 21.6.4, Remark 21.6.5]{Berkeley}, and it is trivialized over $\Spec(\Bdr(R^\sharp))$ by the section $\alpha:=(\beta\otimes id)\circ c^{-1}$. The pair $(\mathcal{F},\alpha)$ defines an $S$-valued point of $\Gr_G$. This induces a map of sheaves
\[\pi_{HT}^\circ: \Shi_{K^p}^\circ\rightarrow \Gr_G.\]

To see that the image lies in $\Gr_{G,\mu}$, we can post-compose with the closed immersion $\Gr_G\hookrightarrow \Gr_{\GL(\Lambda)}$ and assume $S=s:=\Spa(C,C^+)$ is a point. The image is determined by the relative position of the $\Bdr^+(C^\sharp)$-lattices $M_0:=M\otimes_{W(C^+)}\Bdr^+(C^\sharp)$ and $M:=T\otimes_{\Z_p}\Bdr^+(C^\sharp)$, where the latter is trivialized by $\beta$. Since $(M,\varphi_M)$ comes from a $p$-divisible group, we have $M\subset M_0\subset \xi^{-1}M$ and the image of $M_0$ in $\xi^{-1}M/M\cong T\otimes_{\Z_p} C^\sharp(-1)$ agrees with the Lie algebra $\mathrm{Lie}(\mathfrak{A}[p^\infty])\otimes_{\CO_{C^\sharp}}C^\sharp$, cf. \cite[Section 14.8]{Berkeley}. Hence the image of $\pi_{HT}^\circ$ will lie in a single minuscule Schubert cell  of $\Gr_{\GL(\Lambda)}$, labeled by some dominant cocharacter $\lambda$. We need to determine $\lambda$.

For this, we may choose $s$ such that $s^\sharp$ lies over $\Spa \overline{\mathbb{Q}}_p$ and assume the fiber $\AV_s$ algebraizes\footnote{One can use the techniques developed in \cite{Conrad} to prove algebraization, but we do not pursue it here.} to some abelian variety $A_s$, which is already defined over a finite extension of $\mathbb{Q}_p$. Then using the isomorphism $\C\cong \overline{\Q}_p$, we can go through a chain of comparison theorems: between $p$-adic \'etale and de Rham homologies \cite[Theorem 1.6]{Sch13}, analytic and algebraic de Rham homologies \cite[Corollary 32.2.2]{AB} (first $p$-adic analytic with algebraic, then algebraic with complex analytic), and de Rham and Betti homologies \footnote{Also, use invariance of \'etale cohomology under algebraically closed field extensions, and a rigid GAGA theorem \cite[Example 3.2.6]{Conrad} to identify the Hodge cohomologies.}, to get an isomorphism
\[T\otimes_{\mathbb{Z}_p} C^\sharp\cong H_1(A_{s,\C},\C)\otimes_{\overline{\mathbb{Q}}_p}C^\sharp\]
that identifies the subspace $\operatorname{Lie} \mathcal{A}_s(1)\hookrightarrow T\otimes_{\mathbb{Z}_p} C^\sharp$ with the quotient $H_1(A_{s,\C},\C)\otimes_{\overline{\mathbb{Q}}_p}C^\sharp\twoheadrightarrow \operatorname{Lie} A_{s,\C}\otimes_{\overline{\mathbb{Q}}_p}C^\sharp$. This shows that $\lambda$ is oppposite to the Hodge cocharacter determined by $A_{s,\C}$, so $[\lambda]=[\nu_{h^{-1}}]=[\mu]$.

To check Hecke equivariance, by qcqsness of both $\Shi_{K^p}^\circ$ and $\Gr_{G,\mu}$ and properness of the latter, we may assume $S=s:=\Spa C$ is a rank one geometric point. For $g\in G(\Q_p)$, there exists some $N\gg 0$, such that $p^N\Lambda\subset g\Lambda \subset p^{-N}\Lambda$. Denote by $K$ the image of $g\Lambda$ in the quotient 
\[\overline{\beta}: p^{-N}\Lambda/p^N\Lambda\xrightarrow{\sim} \AV_s[p^{2N}].\] 
Then $g$-action sends $(\AV_s, \beta: T_p\AV_s\cong \Lambda)$ to 
\[(\AV_s':=\AV_s/K, T_p\AV_s'\cong g\Lambda).\]
This agrees with the $g$-action on $\Gr_{G,\mu}$ which sends a point $(\Lambda, W)$ as in Remark \ref{minusculeSchubert} to $(g\Lambda, W)$. Away from $p$, the Hecke action conjugates $\Shi_{K^p}^\circ$ to $\Shi^\circ_{g^{-1}K^pg}$ for some $g\in G(\Af)$. Let $K'$ be $K^p\cap g^{-1}K^pg$, then the composition $\Shi^\circ_{K'}\rightarrow \Shi^\circ_{K^p}\xrightarrow{\pi_{HT}^\circ} \Gr_G$ is the Hodge-Tate period map on $\Shi^\circ_{K'}$, similarly for  $\Shi^\circ_{g^{-1}K^pg}$.
\end{proof}
\begin{remark}
    This construction agrees with the construction in \cite{CS17} by compatibility of the strong \'etale comparison of \cite{GR} with the \'etale-de Rham comparison of \cite[Proposition 2.2.3]{CS17}. We leave out the details.
\end{remark}

\begin{remark}\label{rem: SignHT}
    Note that the flag variety in the target of the Hodge-Tate map is labeled by $\mu$ while the Hodge cocharacter of the Shimura variety is $\mu^{-1}$. In particular, this flag variety will have a different sign from the one in the classical Borel embedding. This switch of signs is due to the fact that the Hodge filtration on de Rham cohomologies takes the opposite order to the Hodge-Tate filtration on \'etale cohomologies, under $p$-adic \'etale-de Rham comparison. This sign switch is noted in \cite{CS17}, where they deliberately changed the signs in the Cartan decomposition of the affine Grassmannians to match the two, see \cite[Section 3.4, Footnote 14]{CS17}. We choose not to do this to reflect this difference. When we consider shtukas in Section 11, there will be another flip of the signs, since we want the boundedness of a shtuka to be capturing the Hodge-filtration on de Rham cohomology, see Remark~\ref{rem: SignShtuka}.
\end{remark}
\hfill

\section{Stack of \texorpdfstring{$G$}{}-bundles on the Fargues-Fontaine curve}
     Here we recall the relative Fargues-Fontaine curve and the classifying stack of $G$-torsors $\Bun_G$ following \cite[Section 3.2]{CS17} and \cite[Section II.1]{FS}. For $R$ an $\F_p$-algebra, $[\cdot]:R\rightarrow W(R)$ denotes the Teichm\"uller lift, and $\varphi:=\varphi_R$ is the Frobenius on $W(R)$ lifting the Frobenius on $R$. We use $V(\cdot)$ to denote the vanishing locus of a function on a topological space.

\subsection{Fargues-Fontaine curve and vector bundles}
\begin{defn}[The Fargues-Fontaine curve]
For $S=\Spa(R,R^+)\in \Perf$ an affinoid perfectoid with a pseudo-uniformizer $\varpi\in R^+$, we use the following versions of the relative Fargues-Fontaine curve over $S$:
	\begin{itemize}[leftmargin=0.5cm]
	\item (adic space) Denote by $Y_S$ the adic space $\Spa(W(R^+))\backslash V(p\cdot[\varpi])$, then the adic Fargues-Fontaine curve is the quotient $X_S:=Y_S/\varphi^\Z$;

	\item (scheme) The line bundle $\CO_{Y_S}$ with the linearization $\pi^{-1}\varphi:\CO_{Y_S}\cong \CO_{Y_S}$ descends to an ample line bundle $\CO_{X_S}(1)$ on $X_S$. Define $\CO_{X_S}(n):=\CO_{X_S}(1)^{\otimes n}$ and $P:=\bigoplus_{n\geq 0} H^0(X_S,\CO_{X_S}(n))$, then $X_S^\text{alg}:=\mathrm{Proj}(P)$, defines the algebraic Fargues-Fontaine curve, with a natural morphism of locally ringed spaces:
\[X_S\rightarrow X_S^\text{alg}.\]
	\end{itemize}
This globalizes to a construction of a relative Fargues-Fontaine curve $X_S$ (and hence $X_S^\text{alg}$) for a general $S \in \Perf$.
\end{defn}

\begin{remark}\label{graded}

When $S=\Spa(R,R^\circ)$, we can alternatively construct $X_S$ using the crystalline period ring  \[B_{\text{cris}}^+(R^\circ):=A_\text{cris}(R^\circ)[1/p],\] then by reducing to the case $S$ is a geometric point, and argue as in \cite[Proposition 10.15]{AnsFF}, cf. \cite[Example 5.2.9]{FF}, we have
\[P\cong \bigoplus_{d\geq 0}(B_{\text{cris}}^+(R^\circ))^{\varphi=\pi^d}.\]
See also \cite[Section 6.2]{Far16}. Later we will use this to attach vector bundles on the Fargues-Fontaine curve to $p$-divisible groups via their crystalline Dieudonn\'e modules. 
\end{remark}

\begin{remark}\label{Y}
We denote by $\mathcal{Y}(S)$ the punctured spectrum 
\[\Spa(W(R^+),W(R^+))\backslash\{[\varpi]=p=0\}.\]  
 For $I=[a,b]$ an interval in $[0,\infty]$ with $a, b\in \Q\cup \{\infty\}$, denote by $\mathcal{Y}_{I}(S)$ the open subspace of $\mathcal{Y}(S)$ where 
    \[|p|^b\leq |[\varpi]|\leq |p|^a.\] 

 In this notation, the $Y_S$ above is $\mathcal{Y}_{(0,\infty)}(S)$ and the space 
 \[\Spa(W(R^+), W(R^+))\backslash V([\varpi])\]
 is $\mathcal{Y}_{[0,\infty)}(S)$.\footnote{In \cite[Lecture 11]{Berkeley}, $Y_S$ is alternatively denoted $S\dot{\times}\Spa \Q_p$. Similarly $\Spa(W(R^+))\backslash V([\varpi])$ is denoted $S\dot{\times}\Spa \Z_p$.} It is proven in \cite[Proposition 11.2.1]{Berkeley} that $\mathcal{Y}_{[0,\infty)}(S)$ is an adic space. It is covered by rational subsets of the form $\{|p|\leq|[\varpi^{\frac{1}{p^n}}]|\}, n=1,2,\dots$. Each is represented by an affinoid sousperfectoid space\footnote{i.e. locally the adic spectrum of a complete Tate $\Q_p$-algebra $R$ that admits a split injection of topological $R$-modules into a perfectoid Tate ring} $\Spa(R_n, R_n^+)$, where $R_n^+$ is the $[\varpi]$-adic completion of $W(R^+)[p/[\varpi^{\frac{1}{p^n}}]]$ and $R_n$ is $R_n^+$ inverting $[\varpi]$. As $R_n$ has a presentation 
    \[\{\Sigma_{i\geq 0}[r_i]\left(\frac{p}{[\varpi^{1/p^n}]}\right)^i|r_i\in R, r_i\rightarrow 0\},\]
which depends only on $R$ and not on $R^+$, the category of vector bundles over $\Spa(R_n,R_n^+)$ and hence that over $\mathcal{Y}_{[0,\infty)}(S)$ is independent of the choice of $R^+$ in $R$, see also \cite[Proposition 2.1.1]{PR}. In particular, the category of vector bundles over $Y_S$, or $X_S$, is independent of the choice of $R^+$ in $R$.
\end{remark}

A GAGA type of result holds, relating the adic and algebraic curve: 

\begin{theorem}[{GAGA,\cite[Theorem 8.7.7]{KL1}, \cite[Proposition II.2.7]{FS}}]
Pullback along $X_S\rightarrow X_S^\text{alg}$ induces an equivalence of categories between vector bundles on $X_S$ and $X_S^\text{alg}$.
\end{theorem}

Cartier divisors on $X_S$ classify Frobenius orbits of untilts of $S$. More precisely, fix any untilt $S^\sharp$ over $F$ of $S$. It is locally of the form $\Spa(R^\sharp,R^{\sharp+})$. Each kernel of the surjections $W(R^+)\cong W(R^{\sharp\flat+})\rightarrow R^{\sharp+}$ is a principal ideal generated by an element of the form $p-a[\varpi]$ for some $ a\in W(R^+)$. The induced maps 
\[\Spa(R^\sharp, R^{\sharp+})\rightarrow Y_{\Spa(R,R^+)}\] 
glue and define a closed Cartier divisor $S^\sharp \hookrightarrow Y_S$, which maps to a closed Cartier divisor $S^\sharp \hookrightarrow X_S$. It is cut out by a global section of $\CO_{X_S}(1)$. Hence by GAGA there is a corresponding global section of $\CO(1)$ on $X_S^\text{alg}$, which cuts out a closed Cartier divisor $S^{\sharp,\text{alg}} \hookrightarrow X_S^\text{alg}$.

For $S$ being affinoid, the algebraic curve $X_S^\text{alg}$ is covered by two principal affine charts $X_S^\text{alg}\backslash V(f_i)$, $i=1,2$ for any two linearly independent $f_1, f_2\in H^0(X_S^\text{alg},\CO(1))$. In particular if $S^\sharp$ is cut out by $\xi \in H^0(X_S^\text{alg},\CO(1))$, choose $t\in H^0(X_S^\text{alg},\CO(1))$ linearly independent to $\xi$, then $S^{\sharp,\mathrm{alg}}\hookrightarrow X_S^\text{alg}$ is defined by 
\[(P[1/t])_0 \twoheadrightarrow R^\sharp,\] 
where $(\cdot)_0$ means taking degree zero part of the graded ring. The completion of $(P[\frac{1}{t}])_0$ along $\xi$ is $\Bdr^+(R^\sharp)$. This combined with the Beauville-Laszlo lemma, leads to an interpretation of the $\Bdr^+$-affine Grassmaniann $\Gr_{\GL_n}$ as parametrizing modifications of the trivial rank $n$ bundle on $X_S$. 

Namely, for any $S=\Spa(R,R^+)\in \Perf$, an S-point of the $\Bdr^+$-affine Grassmannian for $\GL_n/\Q_p$ amounts to a triple of an untilt $S^\sharp=\Spa(R^\sharp, R^{\sharp+})$ over $\Q_p$, a rank $n$ vector bundle $\mathcal{F}$ over $\Spec(\Bdr^+(R^\sharp))$, and a trivialization of $\mathcal{F}$ over $\Spec(\Bdr(R^\sharp))$. Via Beauville-Laszlo, this triple defines a new bundle on $X_S^\text{alg}$ by glueing the trivial rank $n$ bundle on $X_S^\text{alg}\backslash\Spec(R^\sharp)$ and $\mathcal{F}$, along the trivialization. This corresponds to a rank $n$ vector bundle on the adic curve $X_S$ by GAGA.

\subsection{The stack of \texorpdfstring{$G$}{}-bundles}

Let $G/\Q_p$ be a reductive group and $X$ be a scheme or a sousperfectoid space over $\Q_p$. Denote by $\mathrm{Rep}_{\Q_p}(G)$ the exact symmetric monoidal category of finite dimensional algebraic $\Q_p$-representations of $G$ and by $\Bun(X)$ that of vector bundles on $X$. 

\begin{def/prop}[{\cite[Theorem 19.5.1, 19.5.2]{Berkeley}}]\label{G-bundle}
A $G$-bundle on $X$ is an exact tensor functor
\[\mathrm{Rep}_{\Q_p}(G) \rightarrow \Bun(X).\]
Equivalently this is an \'etale sheaf on $X$ with an action of $G$ that is \'etale locally isomorphic to $G$.
\end{def/prop} 

The relative Fargues-Fontaine curve $X_S$ for a perfectoid space $S\in \Perf$ is sousperfectoid by \cite[Proof of Proposition 11.2.1]{Berkeley}. Hence one can talk about $G$-torsors on $X_S$ as above. Moreover, by post-composing with the exact tensor equivalence $\Bun(X_S^\text{alg})\cong \Bun(X_S)$, GAGA extends to an equivalence between categories of $G$-torsors on $X_S$ and $X_S^\text{alg}$.

\begin{def/prop}[{\cite[Proposition II.2.1, Definition III.1.2, Proposition III.13]{FS}}]
The pre-stack on $\Perf_{\F}$ sending a perfectoid space $S\in \Perf_{\F}$ to the groupoid of $G$-torsors on $X_S$ is a small v-stack, denoted by $\Bun_G$. 
\end{def/prop}

Using Tannakian formalism, the interpretation of the $\Bdr^+$-affine Grassmannian $\Gr_{\GL_n}$ as parametrizing modifications of the trivial rank $n$ vector bundle generalizes to any other reductive group $G/\Q_p$:

For $S=\Spa(R,R^+)\in \Perf_{\F}$ with an untilt $S^\sharp$ over $F$, viewed as a closed Cartier divisor on $X_S$, the equivalence between $\Bun(X_S^\text{alg})$ and the 2-fiber product 
\[\Bun(X_S^\text{alg}\backslash S^{\sharp,\mathrm{alg}})\times_{\Bun(\Spec(\Bdr(R^\sharp)))}\Bun(\Spf(\Bdr^+(R^\sharp)))\]
is exact and symmetric monoidal. 

Therefore given an $S$-valued point of $\Gr_G$ over $\Spd \Q_p$, i.e. a pair $(\mathcal{F},\alpha)$, where $\mathcal{F}$ is a $G$-torsor over $\Spec(\Bdr^+(R^\sharp))$ and $\alpha$ is a trivialization of it over $\Spec(\Bdr(R^\sharp))$, one can glue the trivial $G$-torsor on $X_S^\text{alg}\backslash S^{\sharp,\mathrm{alg}}$ with $\mathcal{F}$ via $\alpha$ to get new $G$-torsor on $X_S^\text{alg}$. This defines the ``Beauville-Laszlo uniformization" morphism between small v-stacks, see \cite[Section III.3]{FS}:
\[BL: \Gr_G \rightarrow \Bun_G.\]

\begin{remark}\label{rem}
In the above interpretation of $\Gr_G$ as a moduli space of modifications of $G$-torsors, the initial $G$-torsor to modify can be any $G$-torsor, not necessarily the trivial one.   
\end{remark}

\begin{prop}[{\cite[Proposition III.3.1]{FS}}]\label{BLSurj}
The ``Beauville-Laszlo" morphism is surjective as a map of pro-\'etale-stacks.
\end{prop}

\subsection{Stratification}
Let $G$ over $\Q_p$ be a reductive group. Following \cite{FS}, \cite{Ans16}, we review the Newton (or Harder-Narasimhan) stratification of $\Bun_G$, labeled by the Kottwitz set $B(G)$, which is first studied in the setup of isocrystals by Kottwitz \cite{Kottwitz85}, cf. \cite{RR}.

Fix an algebraically closed field $k$ of $\F_p$. Let $L$ be the fraction field of $W(k)$ and $\sigma$ be the Frobenius on $L$. Fix an algebraic closure $\bar{\Q_p}$ of $\Q_p$ containing $L$.

\begin{defn}[{\cite[Definition 5.2]{Ans16}}]
    The Kottwitz category $\mathcal{B}(G)$ is the groupoid whose objects are elements in $G(L)$ and the set of isomorphisms between $b,b'\in G(L)$ is 
    \[\{c\in G(L)\mid cb\sigma(c)^{-1}=b'\}.\]
    Composition of morphisms is defined by multiplication in $G(L)$. The \textit{Kottwitz set} $B(G)$ is the set of isomorphism classes of objects in this category. This is in bijection to the set of $\sigma$-conjugacy classes in $G(L)$. 
\end{defn}
\begin{remark}
    According to Kottwitz \cite{Kottwitz85}, $B(G)$ is invariant under passing to algebraically closed extensions of $k$.
\end{remark}

For each perfectoid space $S$ over $k$, the pullback $\tilde{\mathscr{E}}_1$ of the trivial $G$-bundle $\mathscr{E}_1$ on $X_S$ to $Y_S$ is equipped with a natural descent datum 
\[\alpha: \varphi_S^\ast\tilde{\mathscr{E}}_1\xrightarrow{\sim}\tilde{\mathscr{E}}_1.\] 
Twist $\alpha$ with the automorphism $b\in G(L)\subset \mathrm{Aut}(\tilde{\mathscr{E}}_1)$. The descent datum $(\mathscr{E}_1, b^{-1}\alpha \varphi^\ast b)$ gives rise to a new $G$-bundle $\mathscr{E}_b$ on $X_S$. This assignment 
\[b\mapsto \mathscr{E}_b\in \Bun_G(S)\]
is functorial with respect to pullback along maps $S'\rightarrow S$ over $\Spd k$. In this way we obtain a functor
\[\mathcal{B}(G)\rightarrow \Bun_G(\Spd k),\]
where the target category is understood by v-descent of $G$-torsors on the Fargues-Fontaine curve, namely, by taking any v-cover of $\Spd k$ by a perfectoid space $S$ and considering the category of $G$-torsors on $X_S$ with descent data. We have the following theorem of Ansch\"utz:

\begin{theorem}[{\cite[Theorem 5.3]{Ans16}}]\label{AbsFF}
    The functor 
    \[\mathcal{B}(G)\rightarrow \Bun_G(\Spd k)\] is an equivalence of categories.
\end{theorem}

The Kottwitz set $B(G)$ can be endowed with a partial order recording the degeneration relations of $G$-isocrystals in families. Equip it with the topology defined by the opposite of this partial order. Then a result of Viehmann shows that the above equivalence is compatible with the topologies on the set of objects on both sides. 

More precisely, fix $T \subset B\subset G_{\bar{F}}$, where $T$ is a maximal torus and $B$ is a Borel. Let $X_\ast(T)$ be the cocharacter lattice of $T$. It has an action by the Weyl group $W$ and the absolute Galois group $\Gamma$ of $F$. We write $X_\ast(T)^\Gamma$ for the Galois invariants. Denote by $\pi_1(G)$ the algebraic fundamental group of $G$, i.e. the quotient of $X_\ast(T)$ by the lattice generated by the coroots. It is also equipped with a $\Gamma$-action and we write $\pi_1(G)_\Gamma$ for the $\Gamma$-coinvariants. Kottwitz defined the \textit{Newton} and the \textit{Kottwitz map}
\[\nu_G: B(G)\rightarrow (X_\ast(T)_\Q/W)^\Gamma\]
\[\kappa_G: B(G)\rightarrow \pi_1(G)_\Gamma\]
satisfying certain characterizing properties, see \cite[Theorem 1.8, 1.15]{RR}. In particular
\[\nu_G\times\kappa_G: B(G)\rightarrow (X_\ast(T)_\Q/W)^\Gamma\times \pi_1(G)_\Gamma\]
is injective. Using this, for $[b]$ and $[b']\in B(G)$, we say $[b]\leq [b']$ if $\kappa_G([b])=\kappa_G([b'])$ and $\nu_G([b])\leq\nu_G([b'])$ in the Bruhat order, i.e. choosing a dominant cocharacter to represent the $W$-orbit of each, then the difference $\nu_G([b'])-\nu_G([b])$ is a sum of positive coroots with non-negative coefficients. Equip $B(G)$ with the topology such that $\{[b]\}\in \overline{\{[b']\}}$ if and only if $[b]\geq [b']$. We have 

\begin{theorem}[{\cite[Theorem 1.1]{Viehmann}}]\label{TopBunG}
The equivalence in Theorem~\ref{AbsFF} induces a homeomorphism 
\[B(G)\cong |\Bun_G|.\]    
\end{theorem}

Now we can define locally closed substacks of $\Bun_{G, k}$, the base-change of $\Bun_{G}$ to $\Spd k$.

\begin{def/prop}[{\cite[Theorem III.0.2(v), III.5.3]{FS}}]\label{BunGStrata}
    For any $[b]\in B(G)$, define the substack $\Bun_G^b$ of $\Bun_{G, k}$ to be
    \[\Bun_G\times_{|\Bun_G|}\{[b]\}.\]
    It can be identified with the classifying stack of $\widetilde{G}_b$-torsors, for the v-sheaf of groups
    \[\widetilde{G}_b: S\mapsto \mathrm{Aut}_{X_S}(\mathscr{E}_b).\]
\end{def/prop}

\begin{remark}\label{basic}
If the element $[b] \in B(G)$ is \textit{basic}, i.e. maximal under generalization, then the group $\tilde{G}_b$ agrees with the v-sheaf attached to the locally profinite group $G_b(F)$, where $G_b$ is an inner form of $G$ defined by
\[G_b(R)=\{g\in G(L\otimes_F R)\mid g=b\sigma(g)b^{-1}\},\]
for any $\Q_p$-algebra $R$.  
\end{remark}

Let $\mu$ be a dominant cocharacter, we can describe the image of the Beauville-Laszlo map on $\Gr_{G,\mu}$ under the identification $|\Bun_G|\cong B(G)$.

Note that $\mu$ defines an element $\bar{\mu}\in (X_\ast(T)_\Q/W)^\Gamma$ by averaging its Galois conjugates, i.e. 
\[\bar{\mu}:=\frac{1}{[E':\Q_p]}\sum_{\gamma\in \mathrm{Gal}(E'/\Q_p)}\gamma(\mu)\]
for a large enough Galois extension $E'/\Q_p$ over which $\mu$ is defined. Also, let $\mu^\flat$ be the image of $\mu$ in $\pi_1(G)_\Gamma$.

\begin{defn}
The subset $B(G,\mu)\subset B(G)$ of $\mu$-admissible elements is
\[\{[b]\in B(G)\mid \nu_G([b])\leq \bar{\mu}, \kappa([b])=\mu^\flat\}.\]
\end{defn}

\begin{prop}\label{BLImage}
    The map of topological spaces
    \[|\Gr_{G,\mu}|\xrightarrow{|BL|} |\Bun_G|\rightarrow B(G)\]
    has image $B(G,\mu)$.
\end{prop}
\begin{proof}
    This is \cite[Proposition 3.5.3]{CS17}, combined with \cite[Proposition A.9]{Rap}, see \cite[Remark 3.5.8]{CS17}, except that our convention on the Cartan decomposition on $\Gr_G$ differs from theirs by a minus sign, which eliminates the minus sign on $\mu^{-1}$ from their statement.
\end{proof}

\hfill

\section{A PEL-type Igusa stack and the rational conjecture}
     We return to the global PEL setup as in Section 5. Fix the level subgroup $K^p \subset G(\Af^p)$. Let $K_p=G_{\Z_{(p)}}(\Z_p)$ and $K=K_pK^p$. The adic Shimura variety is defined over $E/\Q_p$ with residue field $\F_q$ and $S_K$ over $\CO_E$ is the schematic Shimura variety at level $K$. Let $\Bun_G$ and $\Gr_G$ be those for the group $G_{\Q_p}$. We construct the PEL type Igusa stack at level $K^p$ over $\Spd \F_q$ and discuss part $(1)$ of Conjecture \ref{conjecture} on the good reduction locus. The word ``rational" in the title is in contrast to the integral model in Section 11.

\subsection{Construction of the Igusa stack}
\begin{defn}\label{DefIgusa}
    Equip the slice category $\Perf/\Spd \F_q$ with the v-topology. Let $\Igs:=\Igs_{K^p}^\circ$ be the stackification of $\Igs^\mathrm{pre}$, the category fibered in groupoids over $\Perf/\Spd \F_q$ determined by: 
    \[T=\Spa(R,R^+) \mapsto \Igs^\mathrm{pre}(T),\] 
    where objects in $\Igs^\mathrm{pre}(T)$ are quadruples $(A_0,\iota,\lambda,\bar{\eta})$ of abelian schemes with $G$-structure at level $K$ over $R^+/\varpi$, or $R^+/\varpi$-points of $S_K$ (where $\varpi$ is any pseudo-uniformizer of $R^+$). Isomorphisms between two objects $\AV_0 = (A_0, \iota, \lambda, \bar{\eta}), \AV'_0=(A'_0, \iota', \lambda',\bar{\eta'})$ are quasi-isogenies preserving the $G$-structures, i.e.
	\begin{align*}
  	&\mathrm{Hom}_{\Igs^\mathrm{pre}(T)}(\AV_0, \AV'_0)= \\
	&\left\lbrace \rho \in (\mathrm{Hom}_{R^+/\varpi}(A_0, A'_0)\otimes \Q)^\times \;\middle|\;
  \begin{tabular}{@{}l@{}}
	$\rho \circ \iota(b)=\iota'(b)\circ \rho$, for any $b\in \CO_B$\\
	$\rho^\vee\circ \lambda'\circ \rho=c\lambda$, for \\some $c\in \underline{\Q^\times}(\Spec(R^+/\varpi))$\\
	$\bar{\eta}=\bar{\eta'}\circ \rho_*$
   \end{tabular}
  \right\rbrace.\\
	\end{align*}
\end{defn}

\begin{remark} 
For a different choice of pseudo-uniformizer $\varpi'\in R^+$ (without loss of generality $\varpi \in (\varpi')$), the base change along $R^+/\varpi \twoheadrightarrow R^+/\varpi'$ induces an equivalence between $\Igs^\mathrm{pre}_{\varpi}(T)$ and $\Igs^\mathrm{pre}_{\varpi'}(T)$ by Serre-Tate lifting, see Theorem~\ref{ST0}. So the functoriality of $\Igs^\mathrm{pre}$ is ensured by composing with this equivalence, even if a map $\Spa(R_1,R_1^+) \rightarrow \Spa(R_2,R_2^+)$ does not necessarily preserve the choice of pseudo-uniformizers.
\end{remark}

Using the moduli interpretation of the good reduction locus $\Shi_{K}^\circ$, we get immediately the following:
\begin{prop}
Sending an isomorphism class of abelian schemes with $G$-structure to its reduction (modulo a pseudo-uniformizer on the base) defines a map of v-stacks $\mathrm{red}: \Shi_K^\circ \rightarrow \Igs$.
\end{prop}
\begin{proof}[Construction]	
View $\Shi_{K}^\circ$ as the sheafification of the presheaf on $\Perf_{\F_q}$
\[\Shi_K^{\circ,\mathrm{pre}}: S=\Spa(R,R^+)\mapsto \{(S^\sharp, \mathscr{S}_K(\Spf R^{\sharp+}))\},\]
where $S^\sharp=\Spa(R^\sharp,R^{\sharp+})$ is an untilt of $S$ over $E$. Choose a pseudo-uniformizer $\varpi\in R^+$. An $S$-point of $\Shi_K^{\circ,\mathrm{pre}}$ gives a formal abelian scheme $\mathfrak{A}/\Spf R^{\sharp+}$ with $G$-structure at level $K$. As $R^+/\varpi \cong R^{\sharp+\flat}/\varpi \cong R^{\sharp+}/\varpi^\sharp$,  $\AV_0:=\mathfrak{A}\times_{R^{\sharp+}} R^{\sharp+}/\varpi^\sharp$ is an object of $\Igs^\mathrm{pre}(S)$. Sending $(S^\sharp, \mathfrak{A})$ to $\AV_0$ defines a map $\Shi_K^{\circ,\mathrm{pre}} \rightarrow \Igs^\mathrm{pre}$, inducing the desired map. 
\end{proof}

\begin{remark}\label{IntRed}
By allowing the untilt $S^\sharp$ to lie over $\Spa \CO_E$, the construction above extends to a map on the v-sheaf integral model
\[\mathrm{red}: \SK\rightarrow \Igs.\]
\end{remark}

Next, we would like to construct a map $\Igs\to \Bun_G$. Before that we need a lemma.

Let $S=\Spa(R,R^+)$ in $\Perf$ be an affinoid perfectoid and $S^\sharp=\Spa(R^\sharp, R^{\sharp+})$ be any untilt of $S$ over $\Spa E$. Assume $\mathfrak{A}$ is a formal abelian scheme (without $G$-structures) over $\Spf R^{\sharp+}$ with
\[A_0:=\mathfrak{A} \times_{R^{\sharp+}}R^+/\varpi.\] 
Note that we have a priori two ways of attaching a vector bundle on the relative Fargues-Fontain curve $X_S$ to $\mathfrak{A}$. Namely, the rational crystalline Dieudonn\'e module 
\[M[1/p]:=M_\mathrm\mathrm(A_0[p^\infty]\times_{R^+/\varpi} R^\circ/\varpi)[1/p]\] 
of $A_0[p^\infty]\times_{R^+/\varpi} R^\circ/\varpi$ is a finite projective $B_\cris^+(R^\circ/\varpi)$-module. Then the graded module 
\[\bigoplus_{d\geq 0}(M[1/p])^{\varphi=p^{d+1}}\] 
defines a vector bundle $\mathscr{E}(A_0)$ with $G$-structure on the relative Fargues-Fontain curve $X_S$, using the description of the algebraic curve in Remark \ref{graded}, as well as the GAGA theorem for the curve.

On the other hand, if we denote the prismatic Dieudonn\'e module of the $p$-divisible group of $\mathfrak{A}$ by $(\tilde{M},\varphi_{\tilde{M}})$, then the restriction of $(\tilde{M},\varphi_{\tilde{M}})$ to $\mathcal{Y}_{[r, \infty)}(S)$ for $r\gg 0$ descends to a vector bundle $\mathscr{E}$ on $X_S$. We have the following relation between the two constructions.

\begin{lemma}\label{commute}
There is a natural (with respect to isomorphism $\mathfrak{A}\xrightarrow{\sim} \mathfrak{A}'$ and base change in $S$) isomorphism $\mathscr{E}\simeq \mathscr{E}(A_0)$.
\end{lemma}
\begin{proof}
     Let $\AV$ be the adic generic fiber of $\mathfrak{A}$ and $T_p\AV$ be its Tate module, considered as a $\underline{\Z}_p$-local system on $S$. For any $f\in T_p\mathcal{A}(S)$, by viewing it of a homomorphism $\Q_p/\Z_p\rightarrow \mathfrak{A}[p^\infty]$ as $p$-divisible groups over $\Spf R^{\sharp+}$, one gets a map between the covariant Dieudonn\'e modules $W(R^+)\cong M_{\prism}(\Q_p/\Z_p)\rightarrow \tilde{M}$. This gives rise to a natural Frobenius equivariant evaluation map 
    \[T_p\mathcal{A}(S)\otimes W(R^+) \rightarrow \tilde{M},\]
    where the tensor is over the global sections of $\underline{\Z}_p$ on $S$. Base change to $A_\mathrm{cris}(R^+)$, one has a similar map. Upon identifying $\tilde{M}\otimes_{W(R^+)}A_\mathrm{cris}(R^+)$ with the crystalline Dieudonn\'e module of $A_0$ (see \cite[Theorem 17.5.2]{Berkeley}), we see that both $\mathscr{E}$ and $\mathscr{E}(\AV_0)$ are naturally the modification of $T_p\mathcal{A}\otimes \CO_{X_S}$ by the $\Bdr^+(R^\sharp)$-lattice $\tilde{M}\otimes_{W(R^+)} \Bdr^+(R^\sharp)$ along the de Rham comparison isomorphism
    \[T_p\mathcal{A}(S)\otimes \Bdr(R^\sharp) \cong \tilde{M}\otimes _{W(R^+)} \Bdr(R^\sharp).\]

    Hence one has an identification $\mathscr{E}\cong \mathscr{E}(\mathcal{A}_0)$ depending only on the isomorphism $\mathfrak{A}\times_{R^{\sharp+}}R^+/\varpi\cong A_0$ and is functorial with respect to isomorphisms between lifts and base change in $S$.
\end{proof}
\begin{remark}
    From the proof we see that $\mathscr{E}(\mathcal{A}_0)$ sits in the short exact sequence
    \[0\rightarrow T_p\mathcal{A}\otimes_{\underline{\Z}_p} \CO_{X_S}\rightarrow \mathscr{E}(\mathcal{A}_0)\rightarrow i_\ast Lie(\mathcal{A})\rightarrow 0,\]
    where $i:S^\sharp \rightarrow X_S$ is the closed immersion of the Cartier divisor $S^\sharp$.
\end{remark}

\begin{prop}\label{prop: MapIgstoBunG}
    There is a morphism of v-stacks
	\[\overline{\pi}_{HT}^\circ: \Igs\rightarrow \Bun_G.\]
 (The notation is justified by Theorem~\ref{cartesian}.)
\end{prop}

\begin{proof}[Construction]
It suffices to construct the map on $\Igs^\mathrm{pre}$. As $\Bun_G$ is a v-stack, this will necessarily factor through the v-stackification and give the desired map.

For $S=\Spa(R,R^+) \in \Perf$ with a chosen pseudo-uniformizer $\varpi \in R^+$, $\Igs^\mathrm{pre}(S)$ is the groupoid of quadruples $\AV_0:=(A_0, \iota, \lambda, \bar{\eta})$, where $A_0$ is an abelian scheme over $R^+/\varpi$; $\iota$ is a morphism $\CO_B \rightarrow \mathrm{End}(A_0)\otimes \Z_{(p)}$; $\lambda$ is a polarization $A_0\rightarrow A_0^\vee$ whose degree is prime to $p$ and $\bar{\eta}$ is a $K^p$-level structure, with morphisms being quasi-isogenies compatible with the $G$-structures. The $p$-divisible group of $\AV_0$ is defined up to isogeny, and is equipped with $\CO_B$-endomorphism and polarization induced by those on $A_0$. Its rational crystalline Dieudonn\'e module $M[1/p]:=M_\mathrm{cris}(\AV_0[p^\infty])[1/p]$ is a finite projective $B_\cris^+(R^+/\varpi)$-module. This is equipped with a $G$-structure, i.e. an $B\otimes_\Q B_\cris^+(R^+/\varpi)$-module structure and a symplectic pairing, by full-faithfulness of the crystalline Dieudonn\'e module functor \cite[Theorem A]{SW}. Now base-change to $B_\cris^+(R^\circ/\varpi)$, and the graded module $\bigoplus_{d\geq 0}(M[1/p])^{\varphi=p^{d+1}}$ defines a vector bundle $\mathscr{E}(\AV_0)$ with $G$-structure on $X_S$.

Consider the sheaf on $X_S$ of trivializations of $\mathscr{E}(\AV_0)$ as a symplectic similitude $\CO_B\otimes \CO_{X_S}$-module 
\[\underline{\mathrm{Isom}}_{G}(\mathscr{E}(\AV_0),\Lambda\otimes_{\Z_p} \CO_{X_S}).\] 
We claim that this is a $G$-torsor and hence an object in $\Bun_G(S)$. To prove this, choose an arbitrary untilt $S^\sharp=\Spa(R^{\sharp},R^{\sharp+})$ over $E$. Using the formal smoothness of $\mathscr{S}_K$, a formal abelian scheme with $G$-structure $\mathfrak{A}$ over $R^{\sharp+}$ lifting $\AV_0$ compatible with the $G$-structures exists. We can now apply the description of $\mathscr{E}(\mathcal{A}_0)$ in Lemma~\ref{commute}. In particular, since $T_p\mathcal{A}\otimes_{\underline{\Z}_p} \CO_{X_S}$ and the completion of $\mathscr{E}(A_0)$ at $S^\sharp$ are \'etale locally on $X_S$ isomorphic to $\Lambda\otimes_{\Z_p} \CO_{X_S}$ and its completion at $S^\sharp$ as symplectic similitude $\CO_B\otimes \CO_{X_S}$-modules, the sheaf 
\[\underline{\mathrm{Isom}}_{G}(\mathscr{E}(\AV_0),\Lambda\otimes_{\Z_p} \CO_{X_S})\]
is an \'etale $G$-torsor on $X_S$. Indeed, for the statement about the completion at $S^\sharp$, note that $i^\ast\mathscr{E}(\mathcal{A}_0)$ is the Lie algebra of the universal extension of $\mathfrak{A}[p^\infty]$ with $p$ inverted. By \cite[Section 3.23 c]{RZ}, \cite[Remark 21.6.5]{Berkeley}, the Kottwitz condition ensures that $i^\ast\mathscr{E}(\AV_0)$ is \'etale locally on $\Spec(R^\sharp)$ isomorphic to $\Lambda \otimes_{\Z_p} \CO$ as polarized $\CO_B\otimes \CO$-modules. This property lifts to $\Bdr^+(R^\sharp)$, which is complete along $(\xi)$.

On the other hand, by full-faithfulness of the Dieudonn\'e module functor, a quasi-isogeny $\AV_0 \rightarrow \AV_0'$ over $R^+/\varpi$ preserving the $G$-structures, induces an isomorphism $M[1/p]\cong M'[1/p]$ compatible with the $G$-structures, and hence an isomorphism of the attached $G$-bundles $\mathscr{E}(\AV_0)\cong \mathscr{E}(\AV'_0)$. This is a morphism in $\Bun_G(S)$. Everything is functorial in $S$ and hence the above defines a morphism $\Igs^\mathrm{pre}\rightarrow \Bun_G$.
\end{proof}

\begin{remark}
From now on we no longer distinguish the vector bundle $\mathscr{E}(\AV_0)$ with $G$-structure and its attached $G$-torsor. To get back $\mathscr{E}(\AV_0)$ from its attached $G$-torsor, one can take pushout along the standard representation of $G_{\Q_p}$ on $V_{\Q_p}$. 
\end{remark}

We can describe the image of the above map under the homeomorphism $|\Bun_G|\cong B(G)$ from Theorem \ref{TopBunG}. 

\begin{prop}\label{IgsImage}
    The image of 
    \[|\Igs|\rightarrow |\Bun_G|\cong B(G)\] 
    is the subset $B(G,\mu)$ of $\mu$-admissible elements.
\end{prop}

\begin{proof}
    This follows from Proposition \ref{BLImage} and commutativity of the diagram in Theorem \ref{cartesian} below. Note that $\pi_{HT}^\circ$ is surjective on topological spaces: this holds for the Hodge-Tate map $\pi_{HT}^\ast$ on the minimal compactification \cite[Theorem IV.1.1(i)]{Sch15} but the good reduction locus of any fiber of $\pi_{HT}^\ast$ is non-empty, since its canonical compactification contains an Igusa variety, see Theorem~\ref{HTfiber} in the next section. 
\end{proof}

\subsection{A fiber product formula}
We first record two lemmas.
\begin{lemma}
\label{1}
Let $V$ be a valuation ring with fraction field $K$. Let $A$ be an abelian scheme over $V$ and $A_K$ its generic fiber. If $j: G_K\rightarrow A_K$ is a finite sub-group scheme, then there exists a finite sub-group scheme $G$ of $A$, flat over $V$, whose generic fiber agrees with $G_K$.
\end{lemma}

\begin{proof} 
The map $j$ is quasi-compact. Let $Z$ be its schematic image, i.e. defined by the quasi-coherent ideal $\mathcal{I}:=\mathrm{ker}(\mathcal{O}_A\rightarrow j_*\mathcal{O}_{G_K})$. We write $\CO_Z$ for the quotient $\CO_A/\mathcal{I}$, considered as an $\CO_A$-module. We have 
\[\CO_A\twoheadrightarrow \CO_Z \hookrightarrow j_*\CO_{G_K}.\]
Since $j_*\CO_{G_K}$ is torsionfree, so is the submodule $\CO_Z$, which implies that it is flat over $V$ (as $V$ is a valuation ring). It suffices to show that $Z$ can be endowed with a group scheme structure, or equivalently it is equipped with morphisms $m_Z:Z\times_{\mathrm{Spec}(V)} Z\rightarrow Z$ (multiplication), $i_Z: Z\rightarrow Z$ (inverse) and $e_Z: \mathrm{Spec}(V)\rightarrow Z$ (identity section), satisfying the group axioms.

For this, note that $\CO_Z(Z)$ is finitely generated over a valuation ring, so it is a projective $V$-module.\footnote{Cf. the last sentence of \cite{Couchot} and \cite[Definition 2, Proposition 4(iii)]{Hiremath}. Note that Proposition 4(iii) is easy to prove by choosing $M$ to be $C$ in Definition 2 in \textit{loc. cit.}.}
Hence the surjection $\CO_A\rightarrow \CO_Z$ splits as a $V$-module homomorphism. Define $m_Z^*$ to be the composition 
\[\CO_Z\rightarrow \CO_A\xrightarrow{m_A^*} m_* \CO_{A\times A}\twoheadrightarrow m_*\CO_{Z\times Z}.\]
This map, a priori a $V$-module homomorphism, is in fact a $V$-algebra homomorphism, due to the commutativity of the following diagram:
	\[
	\begin{tikzcd} 
	\CO_A\ar[r, shift right] \ar[d,"m^*"]&\CO_Z \ar[r,hook] \ar[l,shift right]\ar[d,dashed, "m_Z^*"]& j_* \CO_{G_K}\ar[d, "m^*"]\\
	\CO_{A\times A}\ar[r]   &\CO_{Z\times Z} \ar[r,hook] & (j\times j)_* \CO_{G_K\times G_K}.
	\end{tikzcd}
	\]
Namely, when postcomposed with the injection $\CO_{Z\times Z} \hookrightarrow (j\times j)_* \CO_{G_K\times G_K}$, $m_Z^*$ agrees with $\CO_Z\hookrightarrow j_*\CO_{G_K}\xrightarrow{m^*} (j\times j)_* \CO_{G_K\times G_K}$, which is a $V$-algebra homomorphism. This defines the multiplication morphpism.

The inverse map on $A$ preserves $G_K$ and hence also its schematic image. Hence we can restrict the inverse map on $A$ to $Z$ to get $i_Z$. The identity section on $G_K$ extends to a section $e_Z:\mathrm{Spec}(V)\rightarrow Z$ by properness of $Z$ over $V$. By uniqueness, it is the same as the identity section of $A$.

Now since $m_Z$, $i_Z$, $e_Z$ are the restrictions of the corresponding morphisms on $A$, they satisfy the desired group axioms. This finishes the proof of the lemma.
\end{proof}

\begin{lemma}\label{modifyAb}
    Let $S$ be an $\CO_E$-scheme, $\AV=(A/S,\iota_A, \lambda_A, \bar{\eta})$ an abelian scheme with $G$-structure at level $K$, i.e. an $S$-point of the Shimura variety $S_K$, and $\mathcal{H}=(\mathcal{H}/S, \iota_H, \lambda_H)$ a $p$-divisible group with $G$-structure. Assume 
    \[\rho: \AV[p^\infty]\rightarrow \mathcal{H}\]
    is an $\CO_B$-linear isogeny preserving the polarization up to a scalar in $\underline{\Q}_p^\times(S)$. Then $A':=A/\mathrm{ker}\rho$ can be uniquely (up to isomorphism) promoted to $(A',\iota'_A, \lambda'_A,\bar{\eta}')$, an abelian scheme with $G$-structure at level $K$, such that the induced map 
    \[\rho': A'[p^\infty]\rightarrow \mathcal{H}\]
    is an isomorphism of $p$-divisible groups with $G$-structure and that the quotient map $\pi:A\rightarrow A'$ is a $G$-isogeny preserving the $K^p$-levels. 
\end{lemma}
\begin{proof}
    Clearly $A'$ inherits a $K^p$-level structure $\bar{\eta}$ from $\AV$. The condition that $\pi$ preserves the $G$- and level structures forces $\bar{\eta}'$ to be $\bar{\eta}$, $\iota'_A$ to be $\pi\iota_A \pi^{-1}$ and $\lambda'_A$ to be $(\pi^\vee)^{-1}(d\cdot \lambda_A)\pi^{-1}$ for some $d\in \underline{\Q}^\times(S)$. Also, for each $b\in \CO_B$, $\iota'_A(b):=\pi\iota_A(b) \pi^{-1}$ is indeed prime-to-$p$, because on the $p$-divisible groups $\iota'_A(b)[p^\infty]=\iota_H(b)$ is an isomorphism. To fix $d$, assume $c\in \underline{\Q}_p^\times (S)$ is a constant such that
    \[c\cdot \lambda_A[p^\infty]=\rho^\vee\circ \lambda_H\circ \rho.\]
    Then $\rho'$ preserves the polarization up to $c\cdot d^{-1}\in \underline{\Q}_p^\times (S)$. The condition that $\rho'$ is an isomorphism of $p$-divisible groups with $G$-structure requires this section to be in $\underline{\Z}_p^\times (S)$.
    Let $v_p: \underline{\Q}_p \rightarrow \underline{\Z}$ be the $p$-adic valuation. We must then require $v_p(d)=v_p(c)$ as sections of $\underline{\Z}(S)$. This fixes $d$ up to a unit in $\underline{\Z}_{(p)}(S)$, showing uniqueness. It also ensures that $\lambda'_A$ is prime-to-$p$, as the induced polarization on $H$ is principal. So $(A', \iota_A', \lambda_A',\bar{\eta}')$ is an abelian scheme with $G$-structure at level $K$.
\end{proof}
\begin{remark}
    If $\rho$ is only a quasi-isogeny, assume that $p^N\rho$ is an actual isogeny for some $N\gg 0$. Then 
    \[\AV':= (A/\mathrm{ker}(p^N\rho), \pi\iota_A\pi^{-1}, (\pi^\vee)^{-1}(c^{-1}\cdot \lambda_A)\pi^{-1},\bar{\eta})\]
    is an abelian scheme with $G$-structure at level $K$ by the same reasoning as above. Note that even though $p^N$, and hence $p^N\rho$, itself might not preserve the $G$-structures, the conjugation $p^{-N}\rho p^N$ does. Note that this construction is implicit in \cite[Section 6.13]{RZ}.
\end{remark}

Let $\Shi_{K^p}^\circ$ over $\Spd E$ be the good reduction locus with infinite level at $p$ as before. It maps to $\Igs$ by composing the projection $\Shi_{K^p}^\circ\rightarrow \Shi_K^\circ$ with the reduction map $\Shi_K^\circ\rightarrow \Igs$. Let $\pi^\circ_{HT}$ be the Hodge-Tate map on it as in section 6.2. The following is the main result of this section.

\begin{theorem}
\label{cartesian}
The diagram of small v-stacks on $\Perf_{\F_q}$
    \[
    \begin{tikzcd}
    \Shi_{K^p}^\circ \ar[r,"\pi_{HT}^\circ"] \ar[d,"\mathrm{red}"]& \Gr_{G, \mu} \ar[d,"BL"]\\
    \Igs \ar[r,"\overline{\pi}_{HT}^\circ"] 					&   \Bun_G 
    \end{tikzcd}
    \]
is 2-cartesian.
\end{theorem}
	
\begin{proof}
We observe that the stackification procedure in defining $\Igs$ will cause no trouble: since stackification commutes with 2-fiber product \cite[Tag 04Y1]{stacks-project}, if we can show $\Shi_{K^p}^\circ$ is the v-stackification of $\Igs^\mathrm{pre}\times_{\Bun_G}\Gr_{G,\mu}$, we also have $\Shi_{K^p}^\circ \cong \Igs\times_{\Bun_G}\Gr_{G,\mu}$. Hence it suffices to show 
\[\Spre\cong
\Igs^\mathrm{pre}\times_{\Bun_G}\Gr_{G,\mu}\]
on a basis of the v-topology, see Lemma \ref{PresheafGRL} for the definition of $\Spre$.

We quickly check 2-commutativity. For $S=\Spa(R,R^+)$ in $\Perf_{\F_q}$, an $S$-point of $\Spre$ is a tuple consisting of an untilt $S^\sharp=\Spa(R^\sharp, R^{\sharp+})$, a formal abelian scheme $\mathfrak{A}$ over $R^{\sharp+}$, with endomorphism $\iota$, polarization $\lambda$, $K^p$-level structure $\bar{\eta}$ and a $G$-trivialization $\alpha$ of the Tate module of its generic fiber. Write $A_0$ for the reduction $\mathfrak{A}\times_{R^{\sharp+}}R^{+}/\varpi$. Lemma \ref{commute} and the construction of the Hodge-Tate map show that along both ways in the diagram, the tuple is sent to a point of $\Bun_G$ that is naturally (with respect to automorphisms of $\mathfrak{A}$ and base change in $S$) identified with the vector bundle with $G$-structure glued from $T_p\AV\otimes \CO_{X_S}$ and 
\[M_\mathrm{cris}(A_0[p^\infty])\otimes_{\mathrm{A}_\mathrm{cris}(R^+)}\Bdr^+(R^\sharp)\] 
along de Rham comparison isomorphism. Hence the diagram is 2-commutative.

Note also that $\Igs^\mathrm{pre}\times_{\Bun_G}\Gr_{G,\mu}(S)$ is discrete and hence the fiber product is 0-truncated: indeed, this groupoid is equivalent to the groupoid whose objects are tuples 
\[(\AV_0, y\in \Gr_{G,\mu}(S), \phi: \mathscr{E}(\AV_0)\cong \mathscr{E}(y)),\]
where $\AV_0$ over $R^+/\varpi$ is an abelian scheme with $G$-structure up to quasi-isogenies, $y$ is an $S$-point of $\Gr_{G,\mu}$, and $\phi$ is an isomorphism of their attached $G$-bundles on $X_S$. An automorphism is a self $G$-quasi-isogeny
$f:\AV_0\rightarrow \AV_0$ in $\Igs^\mathrm{pre}(S)$ such that $\mathscr{E}(f)=\mathrm{id}$. But this means $f$ is the identity on $\AV_0[p^\infty]$ and hence is prime-to-$p$, i.e. an automorphism in $S_K(R^+/\varpi)$. By representability of $S_K$, $f$ must be the identity.

The desired result now follows from Proposition \ref{Bijective} below.
\end{proof}

Assume $S$ is a product of points, i.e $R^+=\prod_{i\in I}C_i^+$, where each $C_i$ is complete algebraically closed with a pseudo-uniformizer $\varpi_i$, $\varpi=(\varpi_i)$ and $R=R^+[1/\varpi]$. We denote by $k_i$ the residue field of $C_i$ and $\bar{C_i^+}$ the image of $C_i^+$ in $k_i$. Later whenever we put a subscript $i$ to a morphism on $S$, we mean its restriction along $\Spa(C_i,C_i^+)\rightarrow S$.
\begin{prop}\label{Bijective}
    Let $S$ be as above, the map 
    \[F: \Spre(S) \rightarrow (\Igs^\mathrm{pre} \times_{\Bun_G}\Gr_{G,\mu})(S)\]
    is a natural (in $S$) bijection.
\end{prop}

\begin{proof}
Assume we are given an $S$-point of the fiber product, i.e. a tuple 
\[(\AV_0, y\in \Gr_{G,\mu}(S), \phi: \mathscr{E}(\AV_0)\cong \mathscr{E}(y))\]
as above. The projection of $y$ to $\Spd E$ selects an untilt $S^\sharp=\Spa(R^\sharp, R^{\sharp+})$. Out of this datum, we would like to construct a formal abelian scheme $\mathfrak{A}$ with $G$-structure at level $K$ over $\Spf(R^{\sharp+})$ and a trivialization of its Tate module. The idea is to apply Serre-Tate theory as in Theorem~\ref{ST1}, which requires us to obtain a $p$-divisible group $\G/R^{\sharp+}$ from $(y,\phi)$, lifting $\mathcal{A}_0[p^\infty]$ up to isomorphism (up to modifying $\AV_0$ in its isogeny class).  

We first handle the case $S=\Spa(C,C^+)$ is a geometric point. By Remark~\ref{minusculeSchubert} and full-faithfulness of $\mathscr{E}(\cdot)$ on $p$-divisible groups over $\CO_C/\varpi$ up to isogeny \cite[Theorem 5.1.4(ii)]{SW}, in this case the tuple simplifies to
\[(\AV_0, (S^\sharp, \mathcal{H}, \alpha), \rho: \AV_0[p^\infty] \times_{C^+/\varpi} \mathcal{O}_C/\varpi \dashrightarrow \mathcal{H} \times_{\mathcal{O}_{C^\sharp}} \mathcal{O}_C/\varpi),\]
where $S^\sharp$$=\Spa(C^\sharp, C^{\sharp+})$ is an untilt of $S$ over $\Spa E$, $\mathcal{H}$ is a $p$-divisible group with $G$-structure over $\CO_{C^\sharp}$, $\alpha: T_p\mathcal{H}\cong \underline{\Lambda}$ is a sympletic similitude $\CO_B$-linear trivialization, and $\rho$ is a quasi-isogeny preserving the $G$-structures.

We may assume $p^n\cdot\rho$ is an actual isogeny, for some integer $n$. Apply Lemma \ref{1} to the valuation ring $\overline{C^+}$ with fraction field $k$, the abelian scheme $\mathcal{B}:=\mathcal{A}_0\times_{C^+/\varpi}\overline{C^+}$ and the finite group scheme $\mathrm{ker}(p^n\cdot \rho_k)$, where $\rho_k$ is the base change of $\rho$ to $k$. We see the closure of $\mathrm{ker}(p^n\cdot \rho_k)$ in $\mathcal{B}$ is a finite flat subgroup scheme. We can then take the quotient of $\mathcal{B}$ by it to obtain a new abelian scheme $\mathcal{B}'$, equipped with $G$-structure by Lemma \ref{modifyAb},\footnote{Note that although we are not exactly in the situation of the lemma, since we have the quasi-isogeny on $p$-divisible groups only over $k$ not over $\overline{C^+}$, but this suffices to fix the constant $d$.} whose $p$-divisible group agrees with $\mathcal{H}$ when base-changed to $k$. 

Now using Lemma \ref{BTgluing}, we can glue $\mathcal{B}'[p^\infty]$ and $\mathcal{H}$ to get a $p$-divisible group $\mathcal{G}$ over $C^{\sharp+}$, with a trivialization $\alpha: T_p\G=T_p\mathcal{H}\cong \Lambda$. Also, applying the full-faithfulness part of Lemma~\ref{BTgluing} to $p^n\rho$ and $\mathcal{B}[p^\infty]\rightarrow \mathcal{B}'[p^\infty]$, we get an isogeny compatible with the $G$-structures 
\[\mathcal{A}_0[p^\infty]\times_{C^+/\varpi}(C^+/\varpi\cdot\CO_C)\rightarrow \mathcal{G}\times_{C^{\sharp+}}(C^+/\varpi\cdot\CO_C).\]
By Theorem \ref{ST0}, this lifts to a $G$-isogeny
\[\mathcal{A}_0[p^\infty]\rightarrow \mathcal{G}\times_{C^{\sharp+}}C^+/\varpi.\]

For a general product of points $S$, by restricting to each $s_i =\Spa(C_i,C_i^+)$, the above construction gives a $p$-divisible group $\G_i/C_i^{\sharp+}$ with $G$-structure, trivialization $\alpha_i: T_p\G_i\cong \Lambda$, and a $G$-quasi-isogeny 
\[\rho_i: \AV_0[p^\infty]\times_{R^+/\varpi}C_i^+/\varpi_i \dashrightarrow \G_i\times_{C_i^{\sharp+}}{C_i^+/\varpi_i},\]
such that $p^{n_i}\rho_i$ is an actual isogeny for some integer $n_i$. By Lemma \ref{BTproduct}, we can take the product of $\G_i$'s to get a $p$-divisible group $\G/R^{\sharp+}$ with $G$-structure. There is a unique trivialization $\alpha: T_p\G\cong \underline{\Lambda}$ that restricts to the $\alpha_i$'s, by properness of the diagonal of $[\ast/\underline{K}_p]\rightarrow \ast:= \Spd \F_q$, see Proposition~\ref{qcqs}. In order to get a $G$-quasi-isogeny
\[\AV_0[p^\infty]\dashrightarrow \G_0:=\G\times_{R^{\sharp+}}R^+/\varpi,\]
which restricts to the $\rho_i$'s, we have to show that the $n_i$'s can be commonly bounded by a large enough integer $N$. 

Note that $(\G,\alpha)$ defines an $S$-point of $\Gr_{G,\mu}$ by tensoring the prismatic Dieudonn\'e module $M_{\prism}(\G)$ up to $\Bdr^+(R^\sharp)$ and using
\[T_p\G\otimes_{\Z_p}\Bdr^+(R^\sharp)\xrightarrow{\overset{\alpha\otimes id}{\sim}} \Lambda\otimes_{\Z_p}\Bdr^+(R^\sharp).\]
Since this $S$-point agrees with the original point $y$ on each $s_i$, by properness of $\Gr_{G,\mu}$ it is so on the whole $S$. In particular, the vector bundle $\mathscr{E}(\G):=\mathscr{E}(\G_0)$ attached to $\G$ is isomorphic to $\mathscr{E}(y)$ and the isomorphism $\phi$ can be viewed as an element in $\Hom(\mathscr{E}(\AV_0[p^\infty]), \mathscr{E}(\G))$.

On the other hand, there is a commutative diagram
	\[
	\begin{tikzcd}
	\Hom(\AV_0[p^\infty], \G_0)\ar[r,"\mathscr{E}(\cdot)"]\ar[d,hook]
	& \Hom(\mathscr{E}(\AV_0[p^\infty]), \mathscr{E}(\G)) \ar[d,hook]\\
	\prod_i(\Hom(\AV_{0,i}[p^\infty], \G_{0,i})[1/p]) \ar[r,"{\mathscr{E}(\cdot)}"]
	& \prod_i \Hom(\mathscr{E}(\AV_{0,i}[p^\infty]), \mathscr{E}(\G_{i})).
	\end{tikzcd}
	\]
If we set $(M,\varphi)$ and $(M',\varphi')$ to be the crystalline Dieudonn\'e module of $\AV_0$ and $\G_0$, then using the classification of $p$-divisible groups over qrsp rings, cf. Example \ref{example-crystallineDieu}, 
the top row can be identified with 
\[(M^\vee\otimes M')^{\varphi^\vee\otimes \varphi'=1} \rightarrow (M^\vee\otimes M'[1/p]\otimes B^+_\cris(R^\circ/\varpi))^{\varphi^\vee\otimes \varphi'=1}.\]
It is then clear that for $N\gg 0$, $p^N\phi$, and hence $p^N(\phi_i)_{i\in I}$, lies in the image of $\mathscr{E}(\cdot)$, so this $N$ serves as the desired upper bound. 

To conclude, by taking product we get a quasi-isogeny 
\[\rho=p^{-N}(\prod_ip^{N}\rho_i):\AV_0[p^\infty]\dashrightarrow \G_0,\] 
that is compatible with $G$-structures and maps to $\phi$ by $\mathscr{E}(\cdot)$. Using Lemma \ref{modifyAb}, we can modify $\AV_0$ with $\rho$ inside its isogeny class to get a new abelian scheme satisfying the condition of Theorem \ref{ST1}. More precisely, there exists a unique (up to isomorphisms) abelian scheme $\AV'_0$ with $G$-structure over $R^+/\varpi$, together with a morphism $\AV_0\rightarrow \AV'_0$ in $\Igs^\mathrm{pre}(S)$, that identifies $\AV'_0[p^\infty]$ $\CO_B$-linearly with $\G_0$, preserving the polarization up to a scalar in $\underline{\Z}_p^\times(S)$.

Now upon modifying the polarization on $\G$ by a section in $\underline{\Z}_p^\times(S)$, which does not change its isomorphism class, $\rho'$ preserves the polarization on the nose and therefore by Theorem \ref{ST1}, the triple $(\AV'_0, \G, \rho')$ gives rise to a formal abelian scheme $\mathfrak{A}$ over $R^{\sharp+}$ with $G$-structure. Its $p$-divisible group agrees with $\G$ and its reduction to $R^+/\varpi$ is isomorphic to $\AV_0$ in $\Igs(S)$. Lifting the $K^p$-level structure $\bar{\eta_0}$ of $\AV'_0$ to $\mathfrak{A}$ is automatic since $(R^{\sharp+},(\varpi))$ is a Henselian pair. Also, $\alpha: T_p\mathfrak{A}_\eta=T_p\G\cong \underline{\Lambda}$ gives the level structure at $p$. This defines an $S$-point of $\Shi^{\circ,\mathrm{pre}}_{K^p}$.

To conclude the proof, we are left to check the assignment 
\[(x,y,\phi)\mapsto(\mathfrak{A}, \alpha)\] defines an inverse of $F$. This is direct and is left to the reader.
\end{proof}

As a result, we get cartesian diagrams at any level. Especially, to compare with the integral model at hyperspecial level in Section 11, we have:
\begin{cor}\label{LevelK}
    Write $\Gr_{G,\mu,K_p}$ for the quotient stack $\left[\Gr_{G, \mu}/K_p\right]$. The following diagram at level $K_p=G_{\Z_p}(\Z_p)$ is 2-cartesian
    \[
    \begin{tikzcd}
    \Shi_K^\circ \ar[r,"{\pi_{HT,K_p}^\circ}"] \ar[d,"\mathrm{red}"]& \Gr_{G,\mu,K_p}\ar[d,"BL"]\\
    \Igs \ar[r,"\overline{\pi}_{HT}^\circ"] 					&   \Bun_G.
    \end{tikzcd}
    \]
    
\end{cor}

Theorem \ref{cartesian} implies directly some geometric properties of $\Igs$.

\begin{cor}
The v-stack $\Igs$ is small.
\end{cor}

\begin{proof}
The Beauville-Laszlo map $\Gr_G\rightarrow \Bun_G$ is a surjective map of v-stacks (even as pro-\'etale stacks), see Proposition \ref{BLSurj}. Hence $\Shi_{K^p}^\circ\rightarrow \Igs$ is surjective. As remarked in \cite{CS19} before Proposition 2.6.4, $\Shi_{K^p}^\circ$ is representable by a perfectoid space. Since $\Shi_{K^p}^\circ\times_{\Igs}\Shi_{K^p}^\circ$ is the fiber product of $\Shi_{K^p}^\circ$ and $\Gr_{G,\mu}\times_{\Bun_G}\Gr_{G,\mu}$ over $\Gr_{G,\mu}$, which is a small v-sheaf, this shows $\Igs$ is small.
\end{proof}

\begin{cor}\label{compactifiable}
The map $\bar{\pi}_{HT}^\circ$ is compactifiable in the sense of \cite[Definition 22.2]{Sch18}.
\end{cor}
\begin{proof}
Combine Proposition~\ref{cartesian}, \cite[Proposition 22.3(iii)]{Sch18} and the compatifiability of $\pi_{HT}^\circ$.
\end{proof}

\begin{cor}\label{nice}
The map $\bar{\pi}_{HT}^\circ$ is separated, representable in spatial diamonds and of finite dim. trg., as defined in \cite[Definition 21.7]{Sch18}.
\end{cor}
\begin{proof}
Representability can be checked after basechange to $\Gr_{G,\mu}$, \cite[Proposition 13.4, 10.11]{Sch18}, where $\pi_{HT}^\circ$ is representable in spatial diamonds. For the claim on finite transcendental dimension, it suffices to check on geometric points of $\Bun_G$ in the image of $\Igs$. Lifting to $\Gr_{G,\mu}$, each fiber of $\pi_{HT}^\circ$ has finite dim.trg., since up to a canonical compactification, it agrees with an Igusa variety, see Theorem \ref{HTfiber}, the diamond attached to some perfect scheme with finite dimension over the base field.
\end{proof}

\begin{cor}
The small v-stack $\Igs$ is an Artin v-stack in the sense of \cite[Chapter IV]{FS}.
\end{cor}
\begin{proof}
Combine \cite[Proposition IV.1.8(iii)]{FS}, \cite[Example IV.1.19]{FS}, and Corollary~\ref{nice}.
\end{proof}

\subsection{Sheaf theoretic implications}
The geometric properties has the following sheaf theoretic implication. Here we use the derived category $D_\text{\rm{\'et}}(\cdot)$ of a small v-stack in the sense of \cite[Definition 14.13]{Sch18}. We fix a ring of coefficients $\Lambda$ such that $n\Lambda=0$ for some $n$ prime to $p$.

\begin{prop}
\label{PBC}
There is a natural base change equivalence
\[BL^\ast R\bar{\pi}^\circ_{HT,\ast}\cong R\pi^\circ_{HT,\ast}\mathrm{red}^\ast\]
of functors $D_\text{\rm{\'et}}(\Igs,\Lambda)\rightarrow D_\text{\rm{\'et}}(\Gr_{G,\mu},\Lambda)$. In particular, the complex $R\pi^\circ_{HT,\ast}\Lambda$ on $\Gr_{G,\mu}$ descends to $R\bar{\pi}^\circ_{HT,\ast}\Lambda$ on $\Bun_G$. The similar statement for $R\bar{\pi}^\circ_{HT,!}$ holds true.
\end{prop}
\begin{proof}
The statement for usual pushforward follows from Corollary \ref{nice} and the qcqs base change of \cite[Proposition 17.6]{Sch18} (see \cite[Section 21]{Sch18} for a discussion on bounding the cohomological dimension of a map by its dim. trg..). The statement for the pushforward with proper supports follows from Corollary \ref{nice}, Corollary \ref{compactifiable} and \cite[22.8]{Sch18}.
\end{proof}

Recall the notion of universal locally acyclic (ULA) sheaves from \cite{FS}. We denote by $D^\mathrm{ULA}(\Bun_G,\Lambda)$ the full subcategory of $D_\text{\rm {\'et}}(\Bun_G,\Lambda)$ consisting of objects that are universal locally acyclic for the structure map $\Bun_G\rightarrow \ast$. By \cite[Theorem V.7.1]{FS}, $A\in D_\text{\rm {\'et}}(\Bun_G,\Lambda)$ lies in this subcategory means that for any $b\in B(G)$, its restriction along 
\[i_b: \Bun_G^b\hookrightarrow \Bun_G\]
is ``admissible". More precisely, let $D(G_b(\Q_p),\Lambda)$ be the derived category of smooth representations of $G_b(\Q_p)$ with $\Lambda$-coefficient (see Remark \ref{basic} for the definition of $G_b$). Then by \cite[Proposition V.2.2]{FS}, there is an equivalence 
$D_\text{\rm {\'et}}(\Bun_G^b,\Lambda)\cong D(G_b(\Q_p),\Lambda)$. Under this equivalence, $i_b^\ast A$ is identified with a complex of $G_b(\Q_p)$-representations, such that for any open pro-$p$ subgroup $K$ of $G_b(\Q_p)$, the $K$-invariants of this complex is a perfect complex of $\Lambda$-modules. We denote the full subcategory of such complexes in $D(G_b(\Q_p),\Lambda)$ by $D^\mathrm{adm}(G_b(\Q_p),\Lambda)$.

We prove below that the complex $R\overline{\pi}^\circ_{HT,\ast}\Lambda$ is ULA with respect to $\Bun_G\rightarrow \ast$. As we will refer to the description of the Newton stratification on the Igusa stack in Proposition~\ref{stratum}, the reader may skim Section 9.4 for the relevant statements. We denote the restriction of $\overline{\pi}^\circ_{HT}$ to $\Igs^b:=\Igs^{\circ,b}_{K^p}$ by $\overline{\pi}^{\circ,b}_{HT}$, and base change everything to $\Spd \overline{\F}_p$.

\begin{prop}\label{ULA}
The complex $R\overline{\pi}^\circ_{HT,\ast}\Lambda$ lies in $D^\mathrm{ULA}(\Bun_G,\Lambda)$.
\end{prop}
\begin{proof}
It suffices to check $i_b^\ast R\overline{\pi}^\circ_{HT,\ast}\Lambda \cong R\overline{\pi}^{\circ,b}_{HT,\ast}\Lambda$ lies in $D^\mathrm{adm}(G_b(\Q_p),\Lambda)$ under the equivalence 
    \[D_\text{\rm {\'et}}(\Bun_G^b,\Lambda)= D_\text{\rm {\'et}}([\ast/\widetilde{G}_b],\Lambda) \cong D_\text{\rm {\'et}}([\ast/G_b(\Q_p)],\Lambda) \cong D(G_b(\Q_p),\Lambda),\]
where the second equivalence is induced by pullback along the section 
\[\pi_\mathrm{unip}: [\ast/G_b(\Q_p)]\rightarrow [\ast/\widetilde{G}_b].\]
By Proposition~\ref{stratum}, up to canonical compactification, $\Igs^b\cong [\mathrm{Ig}^{b,\diamond}/\widetilde{G}_b]$. Then by qcqs base change (along $\pi_\mathrm{unip}$) and the fact that canonical compactification does not change the usual \'etale cohomology, see \cite[Lemma 4.4.2]{CS17}, $\pi_\mathrm{unip}^\ast R\overline{\pi}^{\circ,b}_{HT,\ast}\Lambda$ is identified with
\[R\Gamma(\mathrm{Ig}^{b,\diamond},\Lambda)^\mathrm{sm}:=\varinjlim_{K} R\Gamma(\mathrm{Ig}^{b,\diamond},\Lambda)^K\in D(G_b(\Q_p),\Lambda),\]
where $K$ runs over all open pro-$p$ subgroups of $G_b(\Q_p)$, and $G_b(\Q_p)$ acts on $\mathrm{Ig}^{b,\diamond}$ via the embedding $G_b(\Q_p)\hookrightarrow \widetilde{G}_b$. Now for any open pro-$p$ subgroup $K$ of $G_b(\Q_p)$, we have
\[(R\Gamma(\mathrm{Ig}^{b,\diamond},\Lambda)^\mathrm{sm})^K\cong R\Gamma(\mathrm{Ig}_K^{b,\diamond},\Lambda),\]
where $\mathrm{Ig}_K^{b,\diamond}$ is the diamond attached to the quotient $[\mathrm{Ig}^b/K]$, which is representable by a finite \'etale cover of the perfect central leaf $\mathscr{C}^{\mathbb{X}_b}_\mathrm{perf}$. Now use \cite[Proposition 27.2]{Sch18} and the fact that $\mathscr{C}^{\mathbb{X}_b}_\mathrm{perf}$ is the perfection of a separated scheme of finite type over an algebraically closed field ($\Lambda$ is torsion coprime to $p$), we see that $R\Gamma(\mathrm{Ig}_K^{b,\diamond},\Lambda)$ is a perfect complex of $\Lambda$-modules as desired.
\end{proof}

Proposition \ref{ULA} combined with \cite[Proposition 4.17]{HL23} implies a finer structure of $R\overline{\pi}^\circ_{HT,\ast}\Lambda$. Here we take $\Lambda$ to be $\overline{\F}_\ell$ with $\ell\neq p$. We let 
\[\phi: W_{\Q_p}\rightarrow {}^LG(\overline{\F}_\ell)\]
be a semi-simple $L$-parameter and write 
\[(-)_\phi: D_\text{\rm {\'et}}(\Bun_G,\Lambda) \rightarrow
D_\text{\rm {\'et}}(\Bun_G,\Lambda)\]
for the idempotent localization functor as in Definition A.1 of \textit{loc. cit.}. Note that for any $A\in D_\text{\rm {\'et}}(\Bun_G,\Lambda)$, any Schur irreducible subquotient of $A_\phi$ has Fargues-Scholze $L$-parameter equal to $\phi$. We let $B(G)_\mathrm{un}:=\mathrm{Im}(B(T)\rightarrow B(G))$ be the unramified elements of $B(G)$, where $T\subset G$ is a maximal torus contained in a rational Borel. This is also the set of $b\in B(G)$ whose $\sigma$-centralizer $G_b$ is quasi-split and is a Levi subgroup of $G$, see \cite[Lemma 2.12]{Linus22}, \cite[Section 4.2.1, Corollary 4.2.12]{XZ17}, and hence independent of the choice of $T$.

\begin{prop}
Assuming \cite[Assumption 4.4]{HL23}, if $\phi$ is induced from a generic toral parameter, as defined in \cite[Condition/Definition 1.4]{Linus22}, we have 
\[(R\overline{\pi}^\circ_{HT,\ast}\Lambda)_\phi\cong \bigoplus_{b\in B(G)_\mathrm{un}}(R\overline{\pi}^{\circ,b}_{HT,\ast}\Lambda)_\phi.\]
\end{prop}
\begin{proof}
    Combine Proposition \ref{ULA} and \cite[Proposition 4.17]{HL23}.
\end{proof}

\begin{remark}
\cite[Assumption 4.4]{HL23} is an assumption on the properties of the Fargues-Scholze local Langlands correspondence and is verified for the cases listed in Table (22) of \textit{loc. cit.}.
\end{remark}
\hfill
\hfill

\section{Minimal compactification}
This section aims to construct a minimal compactification of the Igusa stack $\Igs:=\Igs_{K^p}^\circ$ under the assumption that the boundary of the minimal compactification of the Shimura variety has codimension at least two. We first list some notation that is used only in this section.

\begin{notation}\hfill
\begin{itemize}
\item For the pair of structure sheaves $(\CO,\CO^+)$ on the v-site of a small v-stack, we abbreviate $(\CO(S),\CO^+(S))$ as $(\CO,\CO^+)(S)$.  
\item A superscript $()^a$ stands for ``almost", see Section 9.1.3 below.
\item For an adic space (resp. diamond, scheme, formal scheme) $X$, we use $X_\tau$, $\tau=\text{Zar, an, \'et, or } v$, to denote the Zariski, analytic, \'etale or v-site of $X$.
\end{itemize}
\end{notation}

\subsection{Basic constructions}
We define the affinization of small v-stacks, and then review the canonical compactification of maps between v-stacks, as well as some almost mathematics as needed.
\subsubsection{Structure sheaf on small v-stacks}
\begin{defn}[{\cite[Theorem 8.7]{Sch18}}]
Let $X$ be a small v-stack on $\Perf$. Then the sheaves
$\CO_{X}$, $\CO_{X}^+$
are defined to be the unique sheaves that restrict to $\CO_{Y}, \CO_{Y}^+$, for all $Y$ in $\Perf$ with a map $Y\rightarrow X$. 
\end{defn}

\begin{defn}
For any small v-stack $X$, we define a presheaf 
\[X_0:= (S\mapsto \Hom((\CO_{X}, \CO_{X}^+)(X), (\CO_{S},\CO^+_{S})(S))).\]
Here for $\Hom$ we take continuous ring homomorphisms from 
$\CO_{X}(X)$ to $\CO_{S}(S)$ that maps $\CO^+_{X}(X)$ into $\CO_{S}^+(S)$, where the ring $\CO_{X}(X)$ is computed through covers by perfectoid spaces and it is equipped with a limit topology from its expression as an equalizer. 
\end{defn}
\begin{remark}\label{untilted}
     We caution the reader that for a general small v-stack, the global sections $(\CO_{X}, \CO_{X}^+)(X)$ will not have the preferable properties like being affinoid perfectoid, or Tate etc. 
\end{remark}

\begin{lemma}\label{affinization}
    There is a map of presheaves (of groupoids) $X\rightarrow X_0$ such that for any affinoid perfectoid $Y$ in $\Perf$, any map $X\rightarrow Y$ factors uniquely through $X\rightarrow X_0$. We call $X\rightarrow X_0$ the \textbf{affinization} of $X$. 
\end{lemma}
\begin{proof}
    For any $S\in \Perf$ mapping to $X$, by taking the induced map on the global sections of the structure sheaves, we obtain a map $S\rightarrow X_0$. This defines the desired map $X\rightarrow X_0$. On the other hand, any map $X\rightarrow Y$, for $Y=\Spa(R,R^+)$ in $\Perf$, induces a homomorphism 
    \[(R,R^+)
    \rightarrow (\CO_{X},\CO^+_{X})(X),\]
    thus giving a map $X_0\rightarrow Y$ and the composition $X\rightarrow X_0 \rightarrow Y$ is the given map.
\end{proof}

\begin{lemma}
    The presheaf $X_0$ is a v-sheaf.
\end{lemma}
\begin{proof}
    Let $T\to S$ be a v-cover of affinoid perfectoid spaces and $T^\bullet:=T^{\times_S^\bullet}$ be the Cech nerve. We need to show 
    \[X_0(S)=\lim_{\Delta} X_0(T^\bullet).\]
    But each term in the Cech nerve is affinoid perfectoid, so by Lemma~\ref{affinization}, this is equivalent to showing
    \[X(S)=\lim_{\Delta} X(T^\bullet),\]
    which holds since $X$ itself is a sheaf (of groupoids).
\end{proof}

Taking affinization defines a functor from the category of small v-stacks to that of v-sheaves.

\subsubsection{Canonical compactification of maps of v-stacks}

\begin{defn}[{\cite[Defintion 18.6]{Sch18}}]
    Let $f:X\rightarrow Y$ be a separated map of v-stacks. The functor on totally disconnected perfectoid spaces sending $\Spa(R,R^+)$ to 
    \[X(R,R^\circ)\times_{Y(R,R^\circ)}Y(R,R^+)\]
    extends to a v-stack $\overline{X}^{/Y}$. There is a factorization of $f$ as 
    \[X\rightarrow \overline{X}^{/Y}\xrightarrow{\overline{f}^{/Y}} Y\]
     with $\overline{f}^{/Y}$ being partially proper. For any partially proper map $Z\rightarrow Y$ of v-stacks, any map $X\rightarrow Z$ factors uniquely through $X\rightarrow \overline{X}^{/Y}$.
\end{defn}
\begin{prop}[{\cite[Proposition 18.7, Corollary 18.8]{Sch18}}]
    The construction $f\mapsto \overline{f}^{/Y}$ is functorial in $f$.
\end{prop}

\subsubsection{The almost setup}
Let $R$ be a perfectoid Tate ring with subring of power-bounded elements $R^\circ$. We denote the category of $R^\circ$-modules by $R^\circ$-mod.
\begin{defn}[{\cite[Definition 3.21, 3.23]{Sch18}}]
    An $R^\circ$-module $M$ is almost zero if $\varpi M=0$ for all pseudo-uniformizers $\varpi$. Such modules form a thick Serre subcategory of $R^\circ$-mod. The category $R^{\circ a}$-mod of almost $R^\circ$-modules is the quotient of $R^\circ$-mod by the subcategory of almost zero modules. 
\end{defn}
\begin{remark}
    Similarly one can define almost $R^+$-modules for any ring of integral elements $R^+\subset R^\circ$ and the forgetful functor from $R^{\circ a}$-mod to $R^{+ a}$-mod is an equivalence.
\end{remark}
\begin{prop}[{\cite[Theorem 3.24]{Sch18}}]
    Let $(R,R^+)$ be a perfectoid Tate Huber pair and let $X$ be $\Spa(R,R^+)$. Then the $R^+$-module $H^i(X,\CO^+_X)$ is almost zero for $i>0$ and $H^0(X,\CO_X^+)=R^+$.
\end{prop}

\subsection{Igusa varieties}
    To construct the minimal compactification of the Igusa stack, we need to know some properties of the fibers of the Hodge-Tate period map. Since they are identified with Igusa varieties up to canonical compactifications (cf. Theorem \ref{HTfiber} below), we first recall relevant facts about PEL Igusa varieties, based on  \cite{CS17}, \cite{CS19}.

\subsubsection{Igusa varieties}
Let $E/\Q_p$, $\F_q$, $S_K/\CO_E$ with $K_p$ hyperspecial and $\mu$ be as in section 5. Fix an algebraically closed field $k$ containing $\F_q$. Write $S_{K,k}$ for the base change of $S_K$ to $k$. Consider the Kottwitz set $B(G)$ for $G_{\Q_p}$. Using Proposition \ref{IgsImage} and \cite[Theorem 5.1.4]{SW}, the isogeny classes of $p$-divisible groups over $k$ with the corresponding $G$-structure are in bijection to the set $B(G,\mu)$ of $\mu$-admissible elements. For any $[b]\in B(G,\mu)$, we fix a $p$-divisible group $\mathbb{X}_b$ representing the corresponding isogeny class. Write $\tilde{\mathbb{X}}_b$ for its universal cover.

\begin{defn}
    We let $\underline{\mathrm{Aut}}_G(\mathbb{X}_b)$ be the fpqc sheaf of groups on $k$-$\mathrm{alg^{op}}$, the opposite category of $k$-algebras
    \[R\mapsto \mathrm{Aut}_G(\mathbb{X}_b\times_kR)\]
    i.e. $\CO_B$-linear automorphisms of $\mathbb{X}_b\times_kR$ that preserve the polarization up to an element in $\underline{\mathrm{Aut}}(\mu_{p^\infty})(R)=\underline{\Z_p^\times}(R)$.
    Here $\underline{\Z_p^\times}$ is considered as a sheaf on $k$-$\mathrm{alg^{op}}$ by taking continuous maps from $\Spec(R)$ to the profinite group $\Z_p^\times$.
\end{defn}

\begin{def/prop} [{\cite[Definition 2.3.1]{CS19}, \cite{Oort}}]
    The central leaf $\mathscr{C}^{\mathbb{X}_b}\subset S_{K,k}$ is the smooth locally closed subscheme, where the fibers of the $p$-divisible group of the universal abelian scheme at all geometric points are isomorphic to $\mathbb{X}_b$.
\end{def/prop}

\begin{def/prop}\label{IgTorsor} [{\cite[Proposition 2.2.6]{CS19}, cf. \cite[Corollary 2.3.2]{CS19}}]
    The perfect Igusa variety $\mathrm{Ig}^{b}$ is the $\underline{\mathrm{Aut}}_G(\mathbb{X}_b)$-torsor over $\mathscr{C}^{\mathbb{X}_b}$ parametrizing isomorphisms $A[p^\infty]\xrightarrow{\sim} \mathbb{X}_b$, compatible with the $G$-structures. It is a perfect $k$-scheme and is (up to isomorphism) independent of the choice of $\mathbb{X}_b$ in its isogeny class. Moreover, $\mathrm{Ig}^{b}$ is a torsor under the profinite group $\Gamma_{\mathbb{X}_b}:= \mathrm{Aut}_G(\mathbb{X}_b)$ over the perfection $\mathscr{C}^{\mathbb{X}_b}_{\text{perf}}$ of $\mathscr{C}^{\mathbb{X}_b}$.
\end{def/prop}
Note that our notation is different from \cite{CS19}, where they use Fraktur letters for the perfect Igusa varieties to distinguish them with the (pro-)Igusa varieties.

It is manifest from the definition that $\underline{\mathrm{Aut}}_G(\mathbb{X}_b)$ acts on the Igusa variety $\mathrm{Ig}^{b}$. In fact $\mathrm{Ig}^{b}$ has an alternative moduli interpretation, which allows the action of a larger formal group scheme, namely the group of self-quasi-isogenies of $\mathbb{X}_b$, or equivalently, the automorphism group of the universal cover $\tilde{\mathbb{X}}_b$.

\begin{defn}[{Cf. \cite[Definition 4.2.9]{CS17}}]\label{def:Aut}
    Let $\underline{\mathrm{Aut}}_G(\tilde{\mathbb{X}}_b)$ be the fpqc sheaf on the opposite category of $k$-algebras
    \[R\rightarrow \mathrm{Aut}_G(\tilde{\mathbb{X}}_b\times_kR)\]
    where the latter is the group of $B$-linear automorphisms of $\tilde{\mathbb{X}}_b\times_kR$ that preserve the polarization up to an element in $\underline{\mathrm{Aut}}(\tilde{\mu}_{p^\infty})(R)=\underline{\Q_p^\times}(R)$.
    Here $\underline{\Q_p^\times}$ is considered as a sheaf by taking continuous maps from $\Spec(R)$ to the locally profinite group $\Q_p^\times$. 
\end{defn}
This is the special fiber of the formal group scheme considered in \cite[Definition 4.2.9]{CS17}. It follows from \cite[Lemma 4.2.10]{CS17} that it is representable by a formal group scheme over $k$.

\begin{prop}[{\cite[Lemma 4.3.4, Corollary 4.3.5]{CS17}, \cite[Lemma 4.2.2]{CT}}]\label{prop:ActionofAut}
For a $k$-algebra $R$, $\mathrm{Ig}^{b}(R)$ can be identified with the set of isomorphism classes of pairs $(A,\rho)$, where $A\in S_K(R)$ is an abelian scheme with $G$-structures, considered up to isogeny and $\rho$ is a quasi-isogeny 
\[\rho: A[p^\infty]\rightarrow \mathbb{X}_b\times_k R\]
respecting the $G$-structures. In particular,  $\underline{\mathrm{Aut}}_G(\tilde{\mathbb{X}}_b)$ acts on $\mathrm{Ig}^{b}$. 
\end{prop}

\subsubsection{Partial minimal compactifications}
Using the theory of well-positioned subsets \cite{Boxer} and \cite{LS}, one can construct well-behaved partial minimal compactification of the central leaf $\mathscr{C}^{\mathbb{X}_b}$ inside the minimal compactification $S_{K,k}^\ast$ of $S_{K,k}$.

\begin{defn}
    Let $Y$ be the complement of $\mathscr{C}^{\mathbb{X}_b}$ in its closure in $S_{K,k}$. Then the \textit{partial minimal compactification} $\mathscr{C}^{\mathbb{X}_b,\ast}$ of $\mathscr{C}^{\mathbb{X}_b}$ is $\overline{\mathscr{C}^{\mathbb{X}_b}}\backslash \overline{Y}$, i.e. the closure of $\mathscr{C}^{\mathbb{X}_b}$ in $S_{K,k}^\ast$, deleting the closure of $Y$. 
\end{defn}

\begin{defn}
    The partial minimal compactification $\mathrm{Ig}^{b,*}$ is the normalization of $\mathscr{C}^{\mathbb{X}_b,\ast}$ in $\mathrm{Ig}^b$.
\end{defn}

\begin{prop}[{\cite[Proposition 3.2.4]{Mafalda}}]\label{DegPDiv}
Consider the set-theoretic partition of $S^\ast_K=\coprod_Z S_Z$ according to cusp labels at level $K$ as in Section 5. For simplicity, we assume $K$ is principal. Then for a cusp label $Z=(Z,(X,Y,\phi,\varphi_{-2},\varphi_0))$ representing a cusp label at level $K$, the intersection
\[\mathscr{C}_Z^\natural:=\mathscr{C}^{\mathbb{X}_b,\ast}\times_{S^\ast_{K,k}} S_{Z,k}\]
is a central leaf for the Shimura variety $S_{Z,k}$. It is non-empty if and only if the $p$-divisible group with $G$-structure $\mathbb{X}_b$ admits a decomposition
\[\mathbb{X}_b\cong \Hom(X,\mu_{p^\infty})\oplus \mathbb{X}_Z\oplus Y\otimes (\Q_p/\Z_p).\]
In this case, $\mathscr{C}_Z^\natural$ is attached to the $p$-divisible group $\mathbb{X}_Z$.
\end{prop}
\begin{proof}
    Combine Theorem 2.3.2 and Proposition 3.4.2 of \cite{LS}. cf. \cite[Proposition 3.1.4]{CS19} in the principally polarized case.
\end{proof}

The following result is crucial to our construction of the minimal compactification of $\Igs$.
\begin{prop}[{\cite[Theorem 3.3.3, Proposition 3.3.5]{Mafalda}}]\label{IgAffine}
    The partial minimal compactifications $\mathrm{Ig}^{b,\ast}$ and $\mathscr{C}^{\mathbb{X}_b,\ast}$ are affine.
\end{prop}
\begin{proof}
The proof in \cite[Theorem 3.3.2, Proposition 3.3.4]{CS19} works verbatim, as the results they cited hold in the generality of Section 5. 
\end{proof}

We record one more property of the partially minimally compactified central leaves.
\begin{lemma}
    The partial minimal compactification $\mathscr{C}^{\mathbb{X}_b,\ast}$ is a normal scheme.
\end{lemma}
\begin{proof}
    The central leaf $\mathscr{C}^{\mathbb{X}_b}$ possesses a partial toroidal compactification $\mathscr{C}^{\mathbb{X}_b,\mathrm{tor}}$, see \cite[Definition 2.2.2]{Mafalda}. One can show that the latter is smooth and in particular normal, using \cite[Proposition 2.3.13]{LS}, see \cite[Lemma 3.1.5]{CS19}. Since $\mathscr{C}^{\mathbb{X}_b,\ast}$ is affine and $\mathscr{C}^{\mathbb{X}_b,\mathrm{tor}}\to \mathscr{C}^{\mathbb{X}_b,\ast}$ has connected fibers, we have that 
    \[\mathscr{C}^{\mathbb{X}_b,\ast}=\Spec (\CO(\mathscr{C}^{\mathbb{X}_b,\mathrm{tor}})).\]
    Now apply \cite[Tag 0358]{stacks-project} connected component-wise to $\mathscr{C}^{\mathbb{X}_b,\mathrm{tor}}$. Since it has finitely many connected components, we deduce that $\mathscr{C}^{\mathbb{X}_b,\ast}$ is normal as desired.
\end{proof}

\subsubsection{Dimension}
Let $G/\Q$ be the algebraic group defined by the global PEL-datum as before. The dimension of a central leaf labeled by $[b]\in B(G_{\Q_p})$ is computed by Hamacher and it agrees with the formal dimension of  $\underline{\mathrm{Aut}}_G(\tilde{\mathbb{X}}_b)$ (by which we mean the number of formal variables in the formula in \cite[Proposition 4.2.11]{CS17}), i.e. we have:

\begin{lemma}\label{dimformula}
    The dimension of the central leaf $\mathscr{C}^{\mathbb{X}_b}$ and the formal dimension of $\underline{\mathrm{Aut}}_G(\widetilde{\mathbb{X}}_b)$ are the same. Both equal $\langle 2\rho, \nu_G(b)\rangle$, where $\rho$ is the half sum of the (absolute) positive roots of $G_{\Q_p}$ and $\nu_G(b)$ is the Newton point of $b$.
\end{lemma}
\begin{proof}
    Combine \cite[Proposition 4.2.11]{CS17} and \cite[Corollary 7.8(2)]{Hamacher}.
\end{proof}

This leads to the following proposition.
\begin{prop}\label{codim}
    If the boundary of the minimal compactification $S^\ast_K$ of the Shimura variety has codimension at least two, then so does the boundary of the partial minimal compactification of any central leaf $\mathscr{C}^{\mathbb{X}_b,\ast}\subset S^\ast_{K,k}$.
\end{prop}

\begin{proof}
We may assume the $\Q$-algebra with positive involution $(B,\ast)$ in the global PEL-datum defining $S_K$ is simple. Assume $\mathbb{X}_b$ admits a decomposition
\[\mathbb{X}_b\cong \Hom(X,\mu_{p^\infty})\oplus \mathbb{X}_Z\oplus Y\otimes (\Q_p/\Z_p)\]
for some cusp label $Z$, cf. Proposition \ref{DegPDiv}.
We compare the dimension of $\mathscr{C}^{\mathbb{X}_b}$ with that of $\mathscr{C}_Z^\natural$. Using Lemma \ref{dimformula}, it suffices to compare the dimension of $\underline{\mathrm{Aut}}_G(\mathbb{X}_b)$ and that of $\underline{\mathrm{Aut}}_G(\mathbb{X}_Z)$ (passing to the universal cover does not change the formal dimension). It suffices to exclude the case where the difference between their dimensions is one, under our codimension assumption on the boundary of $S^\ast_K$. 

But $\underline{\mathrm{Aut}}_G(\mathbb{X}_b)$ is a closed subgroup of $\underline{\mathrm{Aut}}(\mathbb{X}_b)$, which is of the form 
\[
\begin{pmatrix}
\underline{\mathrm{Aut}(\Hom(X,\mu_{p^\infty}))} &  &  \\
\mathcal{H}_{\mathbb{X}_Z,\Hom(X,\mu_{p^\infty})} & \underline{\mathrm{Aut}(\mathbb{X}_Z)}& \\
\mathcal{H}_{Y\otimes (\Q_p/\Z_p),\Hom(X,\mu_{p^\infty})} 
& \mathcal{H}_{Y\otimes (\Q_p/\Z_p),\mathbb{X}_Z} & \underline{\mathrm{Aut}(Y\otimes \Q_p)}.\\ 
\end{pmatrix}
\]
Here we use $\mathcal{H}_{\mathbb{X}_Z,\Hom(X,\mu_{p^\infty})}$ etc. to denote the internal Hom $p$-divisible groups of \cite[Lemma 4.1.6]{CS17}. Hence $\underline{\mathrm{Aut}}_G(\mathbb{X}_b)$ admits a 2-step filtration 
\[U_2\subset U_1\subset U_0 = \underline{\mathrm{Aut}}_G(\mathbb{X}_b),\]
with closed subgroups
\[U_2\subset \mathcal{H}_{Y\otimes (\Q_p/\Z_p),\Hom(X,\mu_{p^\infty})},\] 
\[U_1/U_2\subset \mathcal{H}_{\mathbb{X}_Z,\Hom(X,\mu_{p^\infty})}\times \mathcal{H}_{Y\otimes (\Q_p/\Z_p),\mathbb{X}_Z}.\] 

Decompose $(B,\ast)_{\Q_p}$ into simple factors, which fall into three possible cases, see \cite[Corollary 4.5]{Hamacher}. A casewise study shows that both $U_2$ and $U_1/U_2$ will be of positive dimensions unless $\mathbb{X}_Z=0$. But in this case the degeneration is towards a 0-dimensional cusp, since the $p$-divisible group of the universal abelian scheme over $\mathscr{C}_Z^\natural$ is fiberwise at geometric points isomorphism to $\mathbb{X}_Z$. Also, when $\mathbb{X}_Z=0$ the central leaf $\mathscr{C}^{\mathbb{X}_b}$ agrees with the $\mu$-ordinary locus of $S_K$, which is dense by \cite{Wedhorn}. This means the Shimura variety itself is 1-dimensional and has a 0-dimensional cusp, which is excluded by our assumption.
\end{proof}

For a scheme $X$, write $X_0:=\Spec(\CO_X(X))$ for its affinization.
\begin{cor}
\label{0=*}
If the boundary of the minimal compactification $S_K^\ast$ of the Shimura variety has codimension at least two. Then $(\mathrm{Ig}^b)_0\cong \mathrm{Ig}^{b,\ast}$.
\end{cor}

\begin{proof}
The statement is true for the corresponding central leaves, i.e. $(\mathscr{C}^b)_0\cong {\mathscr{C}}^{b,\ast}$, because the latter is normal, noetherian with boundary codimension at least two by Proposition \ref{codim}. Hence the algebraic Hartogs' extension lemma \cite[Tag 031T]{stacks-project}
applies, from which we obtain that $(\mathscr{C}^b)_0\cong {\mathscr{C}}^{b,\ast}$. It follows immediately that for their perfections, we also have $(\mathscr{C}^b_{\text{perf}})_0\cong {\mathscr{C}}^{b,\ast}_{\text{perf}}$.

Now for the Igusa varieties, we use that $\mathrm{Ig}^b\xrightarrow{q} \mathscr{C}^b_{\text{perf}}$ is a pro-finite \'etale, Galois cover, under the Galois group $\Gamma_{\mathbb{X}_b}:=\mathrm{Aut}_G(\mathbb{X}_b)$ (Proposition \ref{IgTorsor}). For any normal compact open subgroup $K_b$ of $\Gamma_{\mathbb{X}_b}$, the map 
\[\mathrm{Ig}^b_{K_b}:=\mathrm{Ig}^b/\underline{K_b}\xrightarrow{q_{K_b}} \mathscr{C}^b_{\text{perf}}\] 
is a finite Galois cover under $\Gamma_{\mathbb{X}_b}/K_b$. We define its ``partial minimal compactification" to be the normalization of $\mathscr{C}^{b,\ast}_{\text{perf}}$ in $\mathrm{Ig}^b_{K_b}$, denoted
\[\mathrm{Ig}_{K_b}^{b,\ast}\xrightarrow{\overline{q_{K_b}}} \mathscr{C}^b_{\text{perf}}.\]
We have a commutative diagram
\[
\begin{tikzcd}
\mathrm{Ig}^b 	\ar[r,]\ar[d,"f"]	& \mathrm{Ig}^b_{K_b} 	\ar[r,"q_{K_b}"]	\ar[d,"f_{K_b}"] &\mathscr{C}^b_{\text{perf}}	\ar[d,"f_0"] \\
\mathrm{Ig}^{b,\ast} \ar[r]	& \mathrm{Ig}^{b,\ast}_{K_b}	\ar[r,"\overline{q_{K_b}}"] &\mathscr{C}^{b,\ast}_{\text{perf}},
\end{tikzcd}
\]
and $\mathrm{Ig}^{b,\ast}$ equals $\varprojlim_{K_b}\mathrm{Ig}^{b,\ast}_{K_b}$. We denote its projection to $\mathscr{C}^{b,\ast}_{\text{perf}}$ by $\bar{q}$.

At each $K_b$-level, we have
\[(\overline{q_{K_b}})_\ast f_{K_b,\ast}\CO_{\mathrm{Ig}^b_{K_b}}
\cong f_{0,\ast} q_{K_b,\ast}\CO_{\mathrm{Ig}^b_{K_b}}\]
is a finite $\CO_{\mathscr{C}^{b,\ast}_{\text{perf}}}$-algebra, because $q_{K_b,\ast}\CO_{\mathrm{Ig}^b_{K_b}}$ is a finite \'etale $\CO_{\mathscr{C}^b_{\text{perf}}}$-algebra, while 
$f_{0,\ast}\CO_{\mathscr{C}^b_{\text{perf}}}\cong \CO_{\mathscr{C}^{b,\ast}_{\text{perf}}}.$
Hence this is also the normalization of $\CO_{\mathscr{C}^{b,\ast}_{\text{perf}}}$ in it, and we have 
\[(\overline{q_{K_b}})_\ast f_{K_b,\ast}\CO_{\mathrm{Ig}^b_{K_b}}\cong (\overline{q_{K_b}})_\ast \CO_{\mathrm{Ig}^{b,\ast}_{K_b}}.\]
Therefore we can compute that
\begin{align*}
\Gamma(\mathrm{Ig}^b,\CO_{\mathrm{Ig}^b}) &= \varinjlim_{K_b}\Gamma(\mathscr{C}^b_{\text{perf}},q_{K_b,\ast}\CO_{\mathrm{Ig}_{K_b}^b})\\
& = \varinjlim_{K_b}\Gamma(\mathscr{C}^{b,\ast}_{\text{perf}},f_{0,\ast}q_{K_b,\ast}\CO_{\mathrm{Ig}_{K_b}^b})\\
& = \varinjlim_{K_b}\Gamma(\mathscr{C}^{b,\ast}_{\text{perf}},\overline{q_{K_b,\ast}}f_{K_b,\ast}\CO_{\mathrm{Ig}_{K_b}^b})\\
& = \varinjlim_{K_b}\Gamma(\mathscr{C}^{b,\ast}_{\text{perf}},\overline{q_{K_b,\ast}}\CO_{\mathrm{Ig}_{K_b}^{b,\ast}}) = \Gamma(\mathrm{Ig}^{b,\ast},\CO_{\mathrm{Ig}^{b,\ast}}).
\end{align*}
This is what we want to prove.
\end{proof}

From the proof of Proposition \ref{codim} we see that the boundary of the partial minimal compactification of a central leaf having codimension at least two is a very mild condition. In fact when the PEL datum of type AC is simple, this happens only if the corresponding Shimura variety is a non-compact curve. Below we classify such Shimura varieties with a central leaf whose partial minimal compactification has boundary codimension one. A quick observation is that the condition on dimension and existence of cusps already forces the group $G$ to be quasi-split over $\Q$ with absolute root system of type $A_1$. 

\begin{prop}\label{classification}
    Let $(B,\ast)$ be a simple $\Q$-algebra with positive involution and $(B,\ast,V,(\cdot,\cdot), h)$ be a global PEL-datum satisfying Assumption \ref{assumption}. If the boundary of the partial minimal compactification of a central leaf on the attached Shimura variety (hyperspecial level) has codimension one, then the leaf must be the ordinary locus and the Shimura variety is either the modular curve, or a unitary Shimura curve attached to an imaginary quadratic extension of $\Q$ as in Example \ref{unitary}.
\end{prop}

\begin{proof}
    We see in the proof of Proposition \ref{codim} that for the codimension-one situtation to happen, the $p$-divisible group $\mathbb{X}_b$ is ordinary and admit an $\CO_B$-linear decomposition 
    \[\mathbb{X}_b\cong \Hom(X,\mu_{p^\infty})\oplus Y\otimes (\Q_p/\Z_p).\]
    Moreover, the dimension of $U_2=U_1$ must be one. 

    Since each simple factor of $B$ is stable under the involution $\ast$ \cite[Lemma 1.2.1.11]{Lan13}, $B$ itself is simple. By Wedderburn's structure theorem, such a $B$ is a matrix algebra over some division $\Q$-algebra $D$. Hence under Morita equivalence, we may assume $B=D$.
    
    Let $F$ be the center of $D$, $d$ the degree of $D$ over $F$, and $n$ the dimension of $V$ over $D$. Then $\mathrm{End}_D(V)\cong M_n(D)$ and in particular if we base change to $\C$, we have by definition\footnote{Recall that the involution on $\mathrm{End}_B(V)$ induced by $(\cdot,\cdot)$, extending that on $B$, is still denoted $\ast$. We write $\ast_B$ below for its restriction to $B$ to avoid confusion.}
    \[G_{\C}=\{g\in \prod_{F\hookrightarrow \C}M_{n\cdot d^2}\mid gg^\ast\in \Gm_{,\C}\}.\]
    Up to identifying the similitude factors, this will be a product of $\GL_{n\cdot d^2}\times \Gm$ and $\mathrm{GSp}_{n\cdot d^2}$'s, with the number of each factor depending on the shape of the involution $\ast$. Now the constraint that the root datum is of type $A_1$ requires $n\cdot d^2$ to be two and therefore we have necessarily $n=2$ and $d=1$.
    
    According to whether $\dim_{\Q}F$ is one or two, and whether $\ast_B$ is trivial, we have the following cases:
    \begin{itemize}
        \item Case I: $B\cong F\cong \Q$, $\ast_B$ is trivial, and $V=\Q^2$, equipped with the standard symplectic form. In this case $G=\GL_2$, the Shimura variety is the modular curve;
        \item Case II: $B\cong F$ is a quadratic extension of $\Q$, $\ast_B=\overline{(\cdot)}$ is the nontrivial automorphism of $F$ over $\Q$, and $V=F^2$, equipped with the symplectic form $\mathrm{tr}_{F/\Q}\langle\cdot,\cdot\rangle$, where $\langle x,y \rangle = \bar{x}^\tau\cdot y$ is a skew Hermitian form. For $\ast_B$ to be positive, $F$ must be imaginary quadratic. In this case $G$ is a quasi-split $\Q$-form of $\GL_2\times \Gm$ that splits over $F$. It is a unitary similitude group with signature $(1,1)$ at infinity. The corresponding Shimura variety is a unitary Shimura curve as in Example~\ref{unitary};
        \item Case III: $B\cong F$ is a quadratic extension of $\Q$, $\ast_B$ is trivial, and $V=F^2$ is equipped with the standard symplectic form. For $\ast_B$ to be positive, $F$ must be real quadratic. But in this case, $G$ is a form of $(\GL_2\times \GL_2)/\Gm$, which is excluded by the root system constraint.
    \end{itemize}
\end{proof}

\begin{remark}
    That the group $G/\Q$ comes from a global PEL-datum is crucial in this classification. Otherwise there are examples where the group is quasi-split over $\Q$ with type $A_1$ absolute root system, but does not fall in any of the above cases. For example, one can compute quasi-split outer forms of $G:= \GL_2\times \Gm$ over $\Q$, which amounts to representations of $\mathrm{Gal}(\overline{\Q}/\Q)$ into the outer automorphism group of $G_{\bar{\Q}}$. The latter is in bijection to the automorphism group of the root datum 
    \[\mathcal{R}:=(\Z^3, \{\pm \alpha\}, \Z^3, \{\pm \alpha^\vee\}),\]
    where $\alpha=e_1-e_2$ and $\alpha^\vee=e_1^\vee-e_2^\vee$ for a standard basis $e_i$, $i=1,2,3$. This group is of the form
    \[\biggl\{\begin{bmatrix}
        a & a-1 & b\\
        a-1 & a & b\\
        c   & c & d
    \end{bmatrix}\in \GL_3(\Z)\bigg | \mathrm{det}=\pm 1\biggr\},\]
    which is large due to the high rank of the character lattice.    
    We can also compute its finite order elements. There can be only elements of order $4$ and they are of the form $d=1-2a$, $\mathrm{det}=\pm 1$ (among those order $2$ elements has additionally $c=\frac{2a(1-a)}{b}$). For instance, one can take the automorphism of $\mathcal{R}$ that fixes $\pm \alpha$ and rotates the plane perpendicular to it by $\pi/2$ radians. Using these one can construct quasi-split $\Q$-forms of $\GL_2\times \Gm$ that splits only over a degree four extension of $\Q$, and are hence different from all cases in the above classification. To all of them, there are attached Shimura varieties, since the relevant axioms of Shimura data concern only the infinte place, while over $\R$ the groups are either $\mathrm{GU}(1,1)$ or $\GL_2\times \Gm$.
\end{remark}

\subsubsection{Fibers of the Hodge-Tate map}\label{fiber}
Up to a canonical compactification, fibers of the Hodge-Tate map (resp. its minimal compactification) can be identified with Igusa varieties (resp. their partial minimal compactifications). More precisely, let 
\[x: \Spa(C,\CO_C)\rightarrow \Gr_{G,\mu}\]
be a rank one geometric point with an untilt $\Spa(C^\sharp,\CO_{C^\sharp})$ determined by the structure map to $\Spd E$. It determines a $p$-divisible group $\G_x$ over $\CO_{C^\sharp}$ with trivialized Tate module. Write $k$ for the residue field of $C$. Assume $\G_x\times_{\CO_{C^\sharp}} k$ lies in the isogeny class labeled by $b\in B(G,\mu)$. We have the perfect Igusa variety $\mathrm{Ig}^b$, which admits a canonical lift to $W(k)$ and hence to $\CO_{C^\sharp}$. Notation-wise, we set:
    \begin{itemize}
        \item $\mathrm{Ig}^b$: the perfect scheme over $k$;
        \item $\mathrm{Ig}^b_{\CO_C}$: the lift of $\mathrm{Ig}^b$ to $\Spf \CO_{C^\sharp}$, viewed as a formal scheme;
        \item $\Ig$: the adic generic fiber of $\mathrm{Ig}^b_{\CO_C}$. 
    \end{itemize}
    Here for the latter two spaces we suppress the $\sharp$ symbols from notation. Similarly, we write $\mathrm{Ig}^{b,\ast}$, $\mathrm{Ig}^{b,\ast}_{\CO_C}$, $\mathrm{Ig}^{b,\ast}_C$ for the partial minimal compactifications. Note that since $\mathrm{Ig}^b$ and $\mathrm{Ig}^{b,\ast}$ are perfect, $\Ig$ and $\mathrm{Ig}^{b,*}_C$ are perfectoid spaces. The latter is affinoid perfectoid by Proposition \ref{IgAffine}. Later we do not distinguish $\mathrm{Ig}^{b,(\ast)}_{C}$ and its attached diamond.

\begin{theorem}\label{HTfiber}
Fixing a quasi-isogeny 
\[\G_x\times_{\CO_{C^\sharp}} \CO_{C^\sharp}/p \dashrightarrow \mathbb{X}_b\times_k \CO_{C^\sharp}/p,\]
there are natural (with respect to complete algebraically closed extensions of $C$) open immersions
\[\Ig\hookrightarrow (\pi_{HT}^\circ)^{-1}(x),\,\mathrm{Ig}^{b,*}_C\hookrightarrow (\pi_{HT}^\ast)^{-1}(x)\]
inducing isomorphisms on their canonical compactifications towards $x$.
\end{theorem}
\begin{proof}
    The statement for the fiber on the good reduction locus follows from the arguments towards \cite[Theorem 4.4.4]{CS17}. The statement on the minimal compactification is proven in \cite[Theorem 4.5.1]{CS19} for certain unitary Shimura varieties, but their argument has been extended to our situation by Santos, in her Imperial College London PhD thesis \cite[Theorem 4.3.12]{Mafalda}.
\end{proof}

\subsubsection{Torsion in the first cohomology of the integral structure sheaf}
We record a torsion (almost) vanishing result in the first cohomology of $\CO^+$ on the Igusa variety $(\Ig)_{an}$ for later use, which follows from a general result on the first Witt vector cohomology of perfect schemes. 

\begin{prop}\label{torsionfree}
    Let $X$ be a perfect scheme in characteristic $p$. Denote by $W(X)$ the canonical lift of $X$ to characteristic zero using the $p$-typical Witt vectors. Then the cohomology $H^1(W(X),\CO)$ of the structure sheaf on the Zariski site of $W(X)$ is $p$-torsionfree.
\end{prop}
\begin{proof}
    Consider the short exact sequence 
    \[0\rightarrow W\CO_X\xrightarrow{\cdot p} W\CO_X \rightarrow \CO_X\rightarrow 0\]
    on the perfect scheme $X$, where $W\CO_X$ is the sheaf on $X$ sending an open $U$ to the ring of Witt vectors of $\CO_X(U)$. This is exact by perfectness of $X$. Since the last surjection has a multiplicative section given by the Teichm\"uller lift, this gives surjectivity on global sections 
    \[H^0(X,W\CO_X)=W(H^0(X,\CO_X))\twoheadrightarrow H^0(X,\CO_X).\]
    Take the cohomology long exact sequence and we get
    \[H^1(W(X),\CO)[p]=H^1(X,W\CO_X)[p]=0.\]
\end{proof}
\begin{remark}[Question]
    Assume $X$ comes from taking perfection of a smooth quasi-affine scheme of finite type over a perfect field $k$, do we always have that $H^i(X,WO_X)$, $i>0$ has bounded $p$-torsion? 
\end{remark}

    The example below is provided by Offer Gabber, which shows that torsion can appear in degree $2$ and all higher even degrees. By taking products and replacing $\Gamma$ by $\Z/p^n$ with increasing $n$'s, we see that without the finite type assumption, the question has a negative answer.

\begin{example}
Let $\Gamma$ be the cyclic group of order $p$, acting trivially on $\Z$. Its group cohomologies with integer coefficients are 
\[H^i(\Gamma,\Z)=\Bigg\{\begin{array}{@{}@{}l} \Z,\,\, i=0\\ 0,\,\, i \text{ odd}\\ \Z/p, i\geq 2, \text{even}. \end{array}\]
Let $V$ be a finite dimensional faithful $k$-representation of $\Gamma$. We define $Y'$ to be the spectrum of the symmetric algebra on the dual $V^\ast$ of $V$. It has an induced $\Gamma$-action. Upon replacing $V$ with some tensor power of it, we may assume this action is free on an open $U$ whose complement has codimension $m\geq 2$. Consider the quotient $Y'/\Gamma$. Let $Z$ be the complement of $U/\Gamma$ in $Y'/\Gamma$. We may choose homogenous polynomials $h_1,\cdots,h_{\text{dim} Z}$ in $\mathrm{Sym}^\bullet V^\ast$, such that they form a regular sequence on $Z$ and that the vanishing locus $V(h_1,\cdots, h_{\text{dim}Z})$ is smooth. Now $Z\cap V(h_1,\cdots, h_{\text{dim}Z})$ agrees set-theoretically with the origin. Let $C$ be the algebra $\mathrm{Sym}^\bullet V^\ast/(h_1,\cdots,h_{\text{dim}Z})$, $Y$ be the perfection of its punctured spectrum. Then $\Gamma$ acts freely on $Y$ and the quotient $X:=Y/\Gamma$ is a smooth, quasi-affine perfect scheme of finite type over $k$. We can use a Hochschild-Serre spectral sequence to compute the cohomology of $X$. In particular, we have for any $p\leq m-1$
\[H^p(\Gamma, H^0(Y,W\CO_Y))\hookrightarrow H^p(X,W\CO_X).\]
Since $H^0(Y,\CO_Y)=C_\text{perf}$ has $k$ as a retract, $H^p(X,W\CO_X)$ has $H^p(\Gamma, W(k))$ as a direct summand and hence can be torsion if $p$ is even.
\end{example}

For $x=\Spa(C,\CO_C)\rightarrow \Gr_{G,\mu}$ as in subsection 9.2.4, let $\mathrm{Ig}^b$ be the perfect Igusa variety over $k$, which deforms to the flat formal scheme $\mathfrak{X}:= \mathrm{Ig}^b_{\CO_C}$ over $\Spf \CO_{C^\sharp}$ with (perfectoid) adic generic fiber $\Ig$ as in Section~\ref{fiber}. As a corollary to Proposition~\ref{torsionfree}, we have the following.
\begin{cor}\label{H^1(Ig)}
    The $\varpi$-torsion in $H^1_\text{an}(\Ig,\CO^+)$ is almost zero.
\end{cor}
\begin{proof}  
  Take an affine open cover of $\mathrm{Ig}^b$, lift it using the Witt vector functor and base change to $\Spf \CO_{C^\sharp}$. This gives us an affine open cover $\mathfrak{U}=\{\mathfrak{U}_i\}$ of $\mathfrak{X}$, whose adic generic fiber $\{\mathcal{U}=\mathcal{U}_i\}$ is an open cover of $\Ig$ by affinoid perfectoids. Now by almost acyclicity of $\CO^+$ on each $\mathcal{U}_i$, we can compute using Cech cohomology that
  \[H^1_\text{an}(\Ig,\CO^+)=^a \check{H}^1(\mathcal{U},\CO^+)=\check{H}^1(\mathfrak{U},\CO_{\mathfrak{X}})=H^1_\text{Zar}(\mathfrak{X},\CO_{\mathfrak{X}}).\]
  Apply proposition \ref{torsionfree} to $\mathrm{Ig}^b$ and tensor it up to $\CO_{C^\sharp}$. We see by flat base change that $H^1_\text{Zar}(\mathfrak{X},\CO_{\mathfrak{X}})$ is $p$- and hence $\varpi$-torsionfree. This implies the statement we want.
\end{proof}
\subsection{Compactification of the Igusa stack}

Write $\CO$, $\CO^+$ for the structure sheaves on $\Bun_G$. Let $\Igs:=\Igs^\circ_{K^p}$ be the Igusa stack at level $K^p$ constructed in Section 8, equipped with the $0$-truncated map 
\[\overline{\pi}^\circ_{HT}: \Igs\rightarrow \Bun_G.\]
We now combine the results from Sections 9.1 and 9.2 to construct a minimal compactification of $\Igs$, and extend the cartesian diagram in Theorem \ref{cartesian} to the minimally compactified Shimura variety $\Shib$. Throughout this section, we make the following additional (to Assumption~\ref{assumption}) assumption. 
\begin{assumption}\label{Assumption: codim}
    The boundary in the Shimura variety $\Shib$ has codimension at least two.
\end{assumption}

\subsubsection{Preparation}
We begin with a series of Lemma. 

\begin{lemma}\label{totdisc}
For any strictly totally disconnected perfectoid space $T$, with a pseudo-uniformizer $\varpi\in \CO^+_T(T)$ and any sheaf of $\CO_T^+$-modules $\mathcal{F}$ on $T_\text{an}$, the value of the sheaf quotient $\mathcal{F}/\varpi^n$ on a qcqs open $U\subset T$ agrees with $\mathcal{F}(U)/\varpi^n$, for any $n\in \mathbb{Z}_{>0}$.
\end{lemma}
\begin{proof}
Any qcqs open $U\subset T$ is strictly totally disconnected, so every analytic open cover of it splits and it follows from \cite[Lemma 7.2]{Sch18} that $H^1_\text{an}(U,\mathcal{F})=0$. Combined with the short exact sequence of sheaves on $U_\text{an}$
\[0\rightarrow \mathcal{F}\xrightarrow{\cdot\varpi^n}\mathcal{F}\rightarrow \mathcal{F}/\varpi^n\rightarrow 0,\]
we get 
\[\mathcal{F}(U)/\varpi^n \cong (\mathcal{F}/\varpi^n)(U).\]
\end{proof}
\begin{remark}
    This holds on $T_\text{\'et}$ by the same argument. 
\end{remark}
Let $x=\Spa(C,\CO_C)\to \Gr_{G,\mu}$ be a rank one geometric point. It defines maps $\Ig\rightarrow \Shio$, $\Igb\rightarrow \Shib$ with images lying in the fiber of $\pi_{HT}^\ast$ over $x$ by Theorem~\ref{HTfiber}. To lighten the notation, we denote the fibers by 
\[\mathcal{F}^\circ := \Shio\times_{\Gr_G}x,\]
\[\mathcal{F}^\ast := \Shib\times_{\Gr_G}x.\]
\begin{lemma}\label{lemma: AffinizationFibers}
    The natural map $\mathcal{F}^\circ\to \mathcal{F}^\ast$ identifies $\mathcal{F}^\ast$ with $\overline{(\mathcal{F}^\circ)_0}$, the canonical compactification towards $x$ of the affinization of $\mathcal{F}^\circ$.
\end{lemma}
\begin{proof}
There is a commutative diagram
\[
\begin{tikzcd}
\Ig 	\ar[r,]\ar[d,"j"]	& (\Ig)_0 	\ar[r,"\iota"]	\ar[d,"j_0"] &\mathrm{Ig}^{b,\ast}_C	\ar[d,"\bar{j}"] \\
\mathcal{F}^\circ \ar[r]	& (\mathcal{F}^\circ)_0	\ar[r,"(\ast)"] &\overline{\Igb} \cong \mathcal{F}^\ast,
\end{tikzcd}
\]
where the outer square consists of the natural inclusions in Theorem \ref{HTfiber} and the horizontal maps factor through the middle column as the targets are affinoid.

It suffices to show that $(\ast)$ induces natural bijections on the values of both sides on rank one geometric points. This, together with the fact that $(\mathcal{F}^\circ)_0$ is quasi-compact, separated while $\mathcal{F}^*$ is proper, would imply that $\mathcal{F}^*$ is the canonical compactification of $(\mathcal{F}^\circ)_0$ by \cite[Lemma 4.4.2]{CS19}. Since $\overline{j}$ is a bijection on rank one geometric points, it suffices to prove the following two claims:

\begin{enumerate}
\item The map $j_0$ induces natural bijections when evaluated on rank one geometric points.
\item The map $\iota$ induces an isomorphism on canonical compactifications.
\end{enumerate}

\begin{proof}[Proof of claim (1).]
Since the inclusion $j: \Ig\rightarrow \mathcal{F}^\circ$ induces an isomorphism of their canonical compactifications, we have $j_\ast\CO_{\Ig}\cong \CO_{\mathcal{F}^\circ}$. In particular on global sections 
\[\CO_{\Ig}(\Ig)\cong \CO_{\mathcal{F}^\circ}({\mathcal{F}^\circ}).\]
This proves $(1)$, because points on $(\Ig)_0$ are in canonical bijection to equivalence classes of maps from $(\CO,\CO^+)(\Ig)$ to affinoid perfectoid fields and the subset of rank one points depends only on $\CO(\Ig)$, similarly for $(\mathcal{F}^\circ)_0$.
\end{proof}

\begin{proof}[Proof of claim (2).]
Since both sides are affinoid, it suffices to identify the global sections of their structure sheaves. As the global section of $\CO$ is obtained from that of $\CO^+$ by inverting $p$, it suffices to show $\iota^\ast$ induces an almost isomorphism 
\[\CO^+(\Ig)\cong \CO^+(\Igb).\]

So as in Corollary \ref{H^1(Ig)}, we only have to compare the global sections of the structure sheaves on the formal Igusa varieties $\mathrm{Ig}^b_{\CO_C}$ and $\mathrm{Ig}^{b,\ast}_{\CO_C}$, which via a Cech cohomology computation, reduces to comparing that of their special fibers. The result follows from Corollary~\ref{0=*} and Assumption~\ref{Assumption: codim}.
\end{proof}
\end{proof}

\begin{lemma}\label{lemma: AffinoidPerfectoid}
    For any strictly totally disconnected perfectoid space $T$ with a map to $\Gr_{G,\mu}$, the fiber product $\mathcal{S}_{K^p}^\ast\times_{\Gr_{G,\mu}}T$ is an affinoid perfectoid space.
\end{lemma}
\begin{proof}
    For any connected component $s$ of $T$, the fiber $\Shi^\ast_{K^p}\times_{\Gr_{G,\mu}}s$ is affinoid perfectoid by Theorem~\ref{HTfiber} and perfectoidness of $\Igb$. Now apply \cite[Lemma 11.27]{Sch18} to the spatial diamond $\Shi^\ast_{K^p}\times_{\Gr_{G,\mu}}T$ to conclude that it is an affinoid perfectoid space. 
\end{proof}

For a strictly totally disconnected perfectoid space $T\in \Perf$ with a map $T\rightarrow \Bun_G$, we denote by $\Igs_T$ the fiber product $\Igs\times_{\Bun_G}T$. The following is a key input for our construction of minimal compactification of $\Igs$.

\begin{prop}\label{modvarpi} 
Let $T\in \Perf$ be strictly totally disconnected with a map to $\Bun_G$, which lies in the image of $\overline{\pi}_{HT}^\circ$. The natural map
\[\CO^+(\Igs_{T,\text{v}})\slash \varpi^n=\CO^+(\Igs_{T,\text{an}})\slash \varpi^n \rightarrow (\CO^+\slash \varpi^n)(\Igs_{T,an})\]
is an almost isomorphism, for any pseudo-uniformizer $\varpi\in \CO(T)$ and any $n\in \mathbb{Z}_{>0}$.
\end{prop}
\begin{proof}
Denote the projection $\Igs_{T,\text{an}}\rightarrow T_\text{an}$ by $\pi$. We will show that the natural map 
\[(\pi_\ast\CO^+)/\varpi^n\rightarrow \pi_\ast(\CO^+/\varpi^n)\]
is an almost isomorphism of sheaves on $T_\text{an}$. Once this is done, the lemma follows from taking the global sections of both sides and Lemma \ref{totdisc} applied to $\mathcal{F}=\pi_\ast\CO^+$. Since $T$ is strictly totally disconnected, by pro-\'etale surjectivity of the Beauville-Laszlo map and description of the image of $\overline{\pi}^\circ_{HT}$ in Proposition \ref{IgsImage}, the map $T\rightarrow \Bun_G$ lifts to $\Gr_{G, \mu}$ and we fix such a lift.

We check on stalks, i.e. by pulling back to each connected component $s:$ $\Spa(C,C^+)\rightarrow T$, where $C$ is some complete algebraically closed field with an open and bounded valuation subring $C^+$. We have
\[s^\ast\pi_\ast(\CO^+/\varpi^n) =\varinjlim_{s\in U} (\CO^+/\varpi^n)(\Igs_U)= (\CO^+/\varpi^n)(\Igs_s),\]
where the first equality is the definition, while the second equality follows from the approximation property  of $\CO^+/\varpi^n$ in the sense of \cite[Definition 1.4]{heuer}, see \cite[Proposition 1.6]{heuer}. 

On the other hand, we claim that
\[s^\ast(\pi_\ast\CO^+)/\varpi^n 
=\varinjlim_{s\in U}\CO^+(\Igs_U)/\varpi^n   =\CO^+(\Igs_s)/\varpi^n.\]
This would reduce the proof of the statement to the case where $T=s= \Spa(C,C^+)$ is a geometric point. To prove this claim, note that the first equality follows from the definition of taking stalks and Lemma~\ref{totdisc} above\footnote{To make sure that we can apply Lemma~\ref{totdisc} to each $U$ in the colimit, we may choose all the $U$'s to be quasicompact.} The second equality needs a little more input, since the sheaf $\CO^+$ does not satisfy the approximation property in general. 

However, we can argue as follows: as the fiber $\Shi^\ast_{K^p}\times_{\Gr_G}s$ is Zariski closed in the affinoid perfectoid space $\Shi^\ast_{K^p}\times_{\Gr_G}T$ (Lemma~\ref{lemma: AffinoidPerfectoid}), we have a surjection\footnote{Here we use \cite[Remark 7.5]{BS}, which shows that the notion of Zariski closed and strongly Zariski closed sets of affinoid perfectoid spaces agree.}
\[\CO(\Shi^\ast_{K^p}\times_{\Gr_G}T)\twoheadrightarrow \CO(\Shi^\ast_{K^p}\times_{\Gr_G}s).\]
From Lemma~\ref{lemma: AffinizationFibers}, we see that under Assumption~\ref{Assumption: codim}, we have
\[\CO(\Shi^\ast_{K^p}\times_{\Gr_G}s)\cong \CO(\Shi^\circ_{K^p}\times_{\Gr_G}s)=\CO(\Igs_s).\]
Since the surjection
\[ \CO(\Shi^\ast_{K^p}\times_{\Gr_G}T) \twoheadrightarrow\CO(\Shi^\ast_{K^p}\times_{\Gr_G}s)\cong\CO(\Igs_s)
\]
 factors through the restriction to $\CO(\Igs_T)$, we get surjectivity of 
 \[\CO(\Igs_T)\rightarrow \CO(\Igs_s).\] 
Hence any $f\in \CO^+(\Igs_s)\subset \CO(\Igs_s)$ can be lifted to some $\tilde{f}\in \CO(\Igs_T)$. Now the locus $X:=\{|\tilde{f}|\leq 1\}$ is a rational open of $\Igs_T$ containing $\Igs_s$. If we write $s$ as the intersection of shrinking open and closed quasi-compact open subsets $U\subset T$, 
then taking complements, we have
\[\Igs_T\backslash X\subset \bigcup_U (\Igs_T\backslash\Igs_U)\] 
But the left hand side is quasi-compact, so
we may find some $U$, such that $\Igs_U\subset X$ and $\tilde{f}\in \CO^+(\Igs_U)$, which implies that 
\[\varinjlim_{s\in U}\CO^+(\Igs_U)\to\CO^+(\Igs_s)\] (and hence also the map modulo $\varpi^n$) is surjective. The injectivity is clear: if $f\in \CO^+(\Igs_U)$ for some $U$ is mapped to $0$ in $\CO^+(\Igs_s)/\varpi^n$, then $\Igs_s$ lies in the rational open $\{|f|\leq |\varpi^n|\}\subset \Igs_T$. Using a similar quasi-compactness argument as above, $f$ is $\varpi^n$-divisible in the colimit.

Now if $T=s$ is a geometric point, its lift to $\Gr_{G,\mu}$ determines an Igusa variety $\Ig$. We know by Theorem \ref{HTfiber} that the inclusion $j: \Ig\rightarrow \Igs_s\simeq \mathcal{S}_{K^p}^\circ\times_{\Gr_{G,\mu}}s$ induces an isomorphism of their canonical compactifications towards $s$ and hence $j_*\CO=\CO$ and $\CO^+\rightarrow j_*\CO^+$ is an almost isomorphism. We may therefore replace $\Igs_s$ by $\Ig$ in the above, and what we need to show becomes that the natural map
\[\CO^+(\Ig)/\varpi^n\rightarrow (\CO^+/\varpi^n)((\Ig)_{an})\]
is an almost isomorphism. This map is injective with cokernel measured by the $\varpi^n$-torsion in $H^1_\text{an}(\Ig,\CO^+)$, because of the cohomology long exact sequence attached to the short exact sequence on $(\Ig)_\text{an}$:
\[0\rightarrow \CO^+ \xrightarrow{\cdot \varpi^n}\CO^+\rightarrow \CO^+/\varpi^n\rightarrow 0.\]
In other words, we are only left to show this $\varpi^n$-torsion (almost) vanishes. For this, use Corollary \ref{H^1(Ig)}.
\end{proof}

Let $T\in \Perf$ be a strictly totally disconnected perfectoid space with a map to $\Bun_G$. Fix a pseudo-uniformizer $\varpi\in \CO_T^+(T)$. 
\begin{lemma}\label{IgsPerfectoid}
The v-sheaf $\Igs_T$ is represented by a qcqs perfectoid space. The global sections $(\CO,\CO^+)(\Igs_T)$ form a Huber pair with $\CO(\Igs_T)$ being a perfectoid Tate ring. In particular, the affinization $(\Igs_T)_0$ is represented by an affinoid perfectoid space. 
\end{lemma}
\begin{proof}
    We may assume $T\rightarrow \Bun_G$ lies in the image of $\overline{\pi}^\circ_{HT}$, for otherwise either $\Igs_T$ is empty, or we could replace $T$ by a qcqs open subspace to ensure this. We fix a lift of $T\rightarrow \Bun_G$ to $\Gr_{G, \mu}$. By Theorem \ref{cartesian} we have
    \[\Igs_T\cong \Igs\times_{\Bun_G}\Gr_{G,\mu}\times_{\Gr_{G,\mu}}T\cong \Shi^\circ_{K^p}\times_{\Gr_{G,\mu}}T.\]
    Since $\Shi^\circ_{K^p}\times_{\Gr_{G,\mu}}T$ is a quasi-compact open subspace of $\Shi^\ast_{K^p}\times_{\Gr_{G,\mu}}T$, it is a qcqs perfectoid space by Lemma~\ref{lemma: AffinoidPerfectoid}.

    For the second statement, since $\Igs_T$ is qcqs, we can take a finite (analytic) open cover $\{U_i\}$ of $\Igs_T$ by affinoid perfectoids with affinoid perfectoid intersections $\{U_{ij}\}$. The ring of global sections of the structure sheaf $\CO(\Igs_T)$ has a pseudo-uniformizer $\varpi$, and $\CO^+(\Igs_T)=\CO(\Igs_T)\cap\prod_i\CO^+(U_i)$ is a ring of definition, which is $\varpi$-adically complete and integrally closed. 

    By \cite[Proposition 6.1.6]{Berkeley}, it suffices to show $\CO(\Igs_T)$ is a perfect ring. This can be computed as the equalizer
    \[\CO(\Igs_T)=\mathrm{eq}(\prod_i\CO(U_i)\double \rightarrow 
    \prod_{i,j}\CO(U_{ij})),\]
    which is clearly perfect, since Frobenius is an isomorphism on the both $\CO(U_i)$'s and $\CO(U_{ij})$'s. Hence  $\CO(\Igs_T)$ is perfectoid.
    
\end{proof}

\subsubsection{The main results}
Now we are ready for the construction: For each strictly totally disconnected perfectoid space $T$ over $\Bun_G$, we define $\Igs^\ast_T$ to be the v-sheaf
\[\Igs^\ast_T:=\overline{(T\times_{\Bun_G}\Igs)_0}^{/T},\] 
i.e. the canonical compactification towards $T$ of the affinization of $\Igs_T$.

\begin{prop}\label{construction}
The assignment 
\[(T\to\Bun_G)\mapsto \Igs^\ast_T,\] 
commutes with base-change along maps between strictly totally disconnected perfectoid spaces over $\Bun_G$, and hence descends to a v-stack $\Igs_{K^p}^\ast$ with a $0$-truncated map to $\Bun_G$. We call it the \textit{minimal compactification} of $\Igs$. 
\end{prop}

The name ``minimal compactification" is justified by Theorem \ref{cpf} and Corollary~\ref{cor: ContainsIgs} below.  

\begin{proof}
Let $\varpi\in \CO(\Igs_T)$ be a pseudo-uniformizer. For any map $T'\xrightarrow{f} T$ between strictly totally disconnected spaces over $\Bun_G$, the map
\[\Igs_{T'} \rightarrow \Igs_T \rightarrow(\Igs_T)_0\]
factors through $(\Igs_{T'})_0$ by Lemma \ref{affinization} and \ref{IgsPerfectoid}, giving rise to a map
\[(\Igs_{T'})_0
\rightarrow(\Igs_T)_0\times_T T'.\]
We need to show that it induces an isomorphism on their canonical compactifications towards $T'$.

Consider the map
\[\CO^+(\Igs_T)\hat{\otimes}_{\CO^+(T)} \CO^+({T'})\rightarrow \CO^+(\Igs_{T'}).\]
Modulo $\varpi$, the left hand side becomes
\[\CO^+(\Igs_T)/\varpi\otimes_{\CO^+(T)/\varpi} \CO^+({T'})/\varpi
     \overset{\ref{modvarpi}}{=^a}(\CO^+/\varpi)(\Igs_{T,\text{an}})\otimes_{\CO^+(T)/\varpi} \CO^+({T'})/\varpi.\] 
To compare it with $\CO^+(\Igs_{T'})/\varpi$, take a finite (analytic) open cover $\{U_i\}$ of $\Igs_T$ by affinoid perfectoids. By the sheaf condition plus flatness of the map $\CO^+(T)/\varpi\rightarrow \CO^+(T')/\varpi$ (\cite[Proposition 7.23]{Sch18}), we compute that
\small
\begin{align*}
&(\CO^+/\varpi)(\Igs_{T,\text{an}})\otimes_{\CO^+(T)/\varpi} \CO^+({T'})/\varpi\\
&=^a \mathrm{eq}(\prod_i (\CO^+/\varpi)(U_i) \double \rightarrow \prod_{i,j}(\CO^+/\varpi)(U_{ij}))\otimes_{\CO^+(T)/\varpi} \CO^+({T'})/\varpi\\
&=^a \mathrm{eq}(\prod_i\CO^+(U_i)/\varpi\double \rightarrow \prod_{i,j}\CO^+(U_{ij})/\varpi)\otimes_{\CO^+(T)/\varpi} \CO^+({T'})/\varpi\\
&=^a \mathrm{eq}(\prod_i(\CO^+(U_i)\hat{\otimes}_{\CO^+(T)}\CO^+(T'))/\varpi\double \rightarrow \prod_{i,j}(\CO^+(U_{ij})\hat{\otimes}_{\CO^+(T)}\CO^+(T'))/\varpi)\\
&=^a\mathrm{eq}(\prod_i\CO^+(U_i\times_TT')/\varpi\double \rightarrow \prod_{i,j}\CO^+(U_{ij}\times_TT')/\varpi)\\
&=^a (\CO^+/\varpi)(\Igs_{T',\text{an}})\\
&\overset{\ref{modvarpi}}{=^a} \CO^+(\Igs_{T'})/\varpi.
\end{align*}
\normalsize
Apply the above argument to all $\varpi^n$ and pass to the inverse limit. We conclude by $\varpi$-adically completeness of both sides that 
\[\CO^+(\Igs_T)\hat{\otimes}_{\CO^+(T)} \CO^+({T'})\rightarrow \CO^+(\Igs_{T'})\]
is an almost isomorphism. Invert $\varpi$ and we get
\[\CO(\Igs_T)\hat{\otimes}_{\CO(T)} \CO({T'})\cong \CO(\Igs_{T'}).\]

The left hand side is the global section of the structure sheaf\footnote{All morphisms are $\varpi$-adic, so the fiber product exists and on structure sheaves it is given by the completed tensor product, while on the integral structure sheaves it is taking the integral closure of the tensor product in the structure sheaf, followed by a completion, see \cite[Proposition 6.18]{Sch12}.} on $(\Igs_T)_0\times_T T'$ and the right hand side is that on $(\Igs_{T'})_0$. Since the canonical compactification depends only on the structure sheaf and not on the integral structure sheaf, we get
\[\overline{(\Igs_{T'})_0}^{/T'} \cong \overline{(\Igs_T)_0\times_T T'}^{/T'}\cong \overline{(\Igs_T)_0}^{/T}\times_T T'\]
as wished. As totally disconnected affinoid perfectoids form a basis of the v-topology on $\Perf/{\Bun_G}$, these descends to a v-stack $\Igs^\ast:= \Igs^\ast_{K^p}$ on $\Perf$ with a $0$-truncated morphism to $\Bun_G$.
\end{proof}

\begin{theorem}
\label{cpf}
The fiber product $\Igs^\ast\times_{\Bun_G}\Gr_{G,\mu}$ is isomorphic to the minimal compactification $\Shib$ and under this identification, the pullback to $\Gr_{G,\mu}$ of the structure morphism $\Igs^\ast\rightarrow\Bun_G$ is the Hodge-Tate period map on $\Shib$. 
\end{theorem}

\begin{proof}
Take a strictly totally disconnected perfectoid space $T=\Spa(R,R^+)$ in $\Perf$ with a map $T\rightarrow \Gr_{G,\mu}$. The inclusion of the good reduction locus into the minimal compactification of the Shimura variety
\[\Shio\times_{\Gr_{G,\mu}}T\hookrightarrow \Shib\times_{\Gr_{G,\mu}}T,\]
induces a unique map 
\[\Igs^\ast_T\cong \overline{(\Shio\times_{\Gr_{G,\mu}}T)_0}^{/T} \rightarrow \Shib\times_{\Gr_{G,\mu}}T,\]
since the target is affinoid perfectoid and is partially proper over $T$. 

To see this map is an isomorphism, since it is proper \cite[Corollary 18.8(vi)]{Sch18}, it suffices to check it is bijective on rank one geometric points. In particular, it suffices to prove in the case where $T=x=\Spa(C,\CO_C)$ is a rank one geometric point, where it follows from Lemma~\ref{lemma: AffinizationFibers}. 
\end{proof}

\begin{cor}
$\Igs^\ast$ is an Artin v-stack.
\end{cor}
\begin{proof}
    It follows from \cite[Proposition IV.1.8(iii)]{FS}, since $\Bun_G$ is an Artin v-stack and the map $\Igs^\ast\to \Bun_G$ is representable in spatial diamonds according to Theorem~\ref{cpf}.
\end{proof}
\begin{cor}\label{cor: ContainsIgs}
$\Igs\rightarrow \Igs^\ast$ is an open immersion.
\end{cor}
\begin{proof}
Being an open immersion can be checked v-locally by \cite[Proposition 10.11]{Sch18}. In particular, we can check the statement by pulling back to $\Gr_{G,\mu}$, where this map becomes the open immersion of the good reduction locus into the minimal compactification.
\end{proof}

\subsection{Newton stratification}

Fix an algebraically closed field $k$ containing the residue field $\F_q$ of $E$. Let $B(G)$ be the Kottwitz set for $G=G_{\Q_p}$ and $B(G,\mu)$ be the subset of $\mu$-admissible elements. We have discussed in Section 7 the Newton stratification on $\Bun_{G, k}$ labeled by $B(G)$. We can pull it back to define the Newton stratification on the Igusa stack. As before, we fix the level subgroup $K^p$ and write $\Igs$ for $\Igs^\circ_{K^p}$, $\Igs^\ast$ for $\Igs^\ast_{K^p}$. For each $[b]\in B(G,\mu)$, we choose a representative $\mathbb{X}_b$ of the corresponding isogeny class of $p$-divisible groups over $k$ and use them to define the perfect Igusa varieties $\mathrm{Ig}^b$ over $k$. We have an action of the formal group scheme $\underline{\mathrm{Aut}}_G(\tilde{\mathbb{X}}_b)$ on $\mathrm{Ig}^b$ as described in Proposition~\ref{prop:ActionofAut}. We also recall the v-sheaf of groups $\widetilde{G}_b$ from Definition/Proposition~\ref{BunGStrata}, which we think of as the sheaf of automorphisms of the $\Spd k$-point of $\Bun_G$ corresponding to $\tilde{\mathbb{X}}_b$.

\begin{defn}
    For any $[b]\in B(G,\mu)$, we define the following locally closed substacks of $\Igs_{k}^\ast$
    \[\Igs^{\ast,b}:= \Igs^{\ast}_k \times_{|\Bun_G|}\{[b]\}.\] 
    \[\Igs^{b}:= \Igs_k \times_{|\Bun_G|}\{[b]\}.\]
\end{defn}

Before we give a more explicit description of these strata, we need two lemma.

\begin{lemma}\label{lemma: UniqueExistenceJbAction}
Under Assumption~\ref{Assumption: codim}, there is a unique extension of the action of the formal group scheme $\underline{\mathrm{Aut}}_G(\tilde{\mathbb{X}}_b)$ on $\mathrm{Ig}^b$ to the partial minimal compactification $\mathrm{Ig}^{b,*}$.     
\end{lemma}
\begin{proof}
We denote $\underline{\mathrm{Aut}}_G(\tilde{\mathbb{X}}_b)$ by $\widetilde{J}_b$ in this proof to simplify notation. The action of $\widetilde{J}_b$ on $\mathrm{Ig}^b$ is given by a morphism 
\[\widetilde{J}_b\times_{k} \mathrm{Ig}^b\to \mathrm{Ig}^b.\]
Taking global sections gives a ring homomorphism
\[\mathcal{O}(\mathrm{Ig}^{b,\ast})\simeq \mathcal{O}(\mathrm{Ig}^b)\to \mathcal{O}(\widetilde{J}_b\times_{k} \mathrm{Ig}^b)\simeq \mathcal{O}(\widetilde{J}_b\times_{k} \mathrm{Ig}^{b,\ast}),\]
which gives the desired action map $\widetilde{J}_b\times_{k} \mathrm{Ig}^{b,\ast}\to \mathrm{Ig}^{b,\ast}$ by affineness of $\mathrm{Ig}^{b,\ast}$. The uniqueness is clear.
\end{proof}

Let $S=\Spa(R,R^+)$ in $\Perf_k$ be affinoid perfectoid, with chosen pseudo-uniformizer $\varpi$. Recall from the construction in Proposition~\ref{prop: MapIgstoBunG} that we have a functor from $p$-divisible group over $R^+/\varpi$ up to isogenies to vector bundles on the Fargues-Fontaine curve $X_S$ via their rational crystalline Dieudonn\'e modules. We denote this functor by $\mathcal{G}\mapsto \mathscr{E}(\mathcal{G})$. 
\begin{lemma}[{Cf. \cite[Corollary 6.3]{Far16}}]\label{lemma: EFullFaithful}
    The functor $\mathscr{E}$ is fully faithful on $p$-divisible groups, whose Newton polygons are fiber-wise constant over $\Spec(R^+/\varpi)$.
\end{lemma}
\begin{proof}
     Let $\mathcal{G}$ and $\mathcal{H}$ be two $p$-divisible groups over $R^+/\varpi$. Then by \cite[Theorem A]{SW}, homomorphisms between them in the isogeny category are the same as the homomorphisms between their rational crystalline Dieudonn\'e modules, which we denote by $M(\mathcal{G})[1/p]$ and $M(\mathcal{H})[1/p]$ accordingly. We have 
     \[\Hom(\mathcal{G}, \mathcal{H})[1/p] = (M(\mathcal{G})^\vee\otimes M(\mathcal{H})[1/p])^{\varphi_\mathcal{G}^\vee\otimes\varphi_\mathcal{H}=1}.\]
     But by \cite[Lemma 4.2.15, 4.2.16]{CS17}, we may assume $R^+ = R^\circ$, in which case the latter is also the global sections of $\mathscr{E}(\mathcal{G})^\vee\otimes \mathscr{E}(\mathcal{H})$.
\end{proof}

\begin{cor}\label{cor: CompareAut}
The v-sheaf $\underline{\mathrm{Aut}}_G(\tilde{\mathbb{X}}_b)^\diamond$ on $\Perf_k$ is isomorphic to $\widetilde{G}_b$.   
\end{cor}
\begin{proof}
    Let $S=\Spa(R,R^+)$ be an affinoid perfectoid in $\Perf_k$. We have
    \[\underline{\mathrm{Aut}}_G(\tilde{\mathbb{X}}_b)^\diamond (S)= \varinjlim_n \operatorname{Aut}_G(\tilde{\mathbb{X}}_b\times_k R^+/\varpi^n)\simeq \operatorname{Aut}_G(\tilde{\mathbb{X}}_b\times_k R^+/\varpi),\]
    where the first identity follows from the definition of the small diamond functor and the second isomorphism follows from Theorem~\ref{ST0}. By Lemma~\ref{lemma: EFullFaithful} above, the latter agrees with the automorphism group of the $G$-bundle $\mathscr{E}(\mathbb{X}_b\times_kR^+/\varpi)=\mathscr{E}(\mathbb{X}_b)|_{X_S}$, which is $\widetilde{G}_b(S)$ as desired.
\end{proof}

Let $\mathrm{Ig}^{b,\ast,\diamond}$ be the v-sheaf attached to $\mathrm{Ig}^b$ using the small diamond functor, and $\overline{\mathrm{Ig}^{b,\ast,\diamond}}$ be its canonical compactification towards $\Spd k$. Then $\widetilde{G}_b$ acts on it through the isomorphism in Corollary~\ref{cor: CompareAut} and the action of $\underline{\mathrm{Aut}}_G(\tilde{\mathbb{X}}_b)$ on $\mathrm{Ig}^{b,\ast}$ in Lemma~\ref{lemma: UniqueExistenceJbAction}. Since $\widetilde{G}_b$ is partially proper by Remark~\ref{Y}, it extends to an action on $\overline{\mathrm{Ig}^{b,\ast,\diamond}}$.
\begin{prop}\label{stratum}
    We have the following identifications
    \begin{align*}
    \overline{\Igs^b}& \simeq [\overline{\mathrm{Ig}^{b,\diamond}}/\widetilde{G}_b]\\
    \Igs^{\ast,b}& \simeq [\overline{\mathrm{Ig}^{b,\ast,\diamond}}/\widetilde{G}_b],
    \end{align*}
    where $\overline{\Igs^b}$ is the canonical compactification of $\Igs^b$ towards $\Bun_G^b$ and the $\widetilde{G}_b$-actions are as described above.
\end{prop}
\begin{proof}
    The element $b$ defines a surjective map $\Spd k\rightarrow \Bun_G^b$, whose fiber is the v-sheaf $\widetilde{G}_b$. We first check that for $\overline{\Igs^b}$, there is a canonical identification
    \[\overline{\mathrm{Ig}^{b,\diamond}} \simeq \overline{\Igs^{b}} \times_{\Bun_G^b}\Spd k. \]
    Indeed, for any $S=\Spa(R,R^+)$ in $\Perf_k$, using Lemma~\ref{lemma: EFullFaithful}, an $S$-point of the right hand side amounts to a pair $(A_0, \alpha)$, where $A_0$ is an abelian scheme over $R^\circ/\varpi$, considered up to isogeny, for some pseudo-uniformizer $\varpi \in R^+$; and $\alpha: A_0[p^\infty]\dashrightarrow \mathbb{X}_b\times_{k}R^\circ/\varpi$ is a quasi-isogeny. This is the same as a $\Spec R^\circ/\varpi$-point of $\mathrm{Ig}^b$. Since $\mathrm{Ig}^b$ is a perfect scheme, such a point lifts automatically to the perfection $\Spec R^\circ\simeq \Spec (R^\circ/\varpi)^\mathrm{perf}$. But $\Spec R^\circ\to \mathrm{Ig}^b$ gives precisely an $S$-point of the v-sheaf $\overline{\mathrm{Ig}^{b,\diamond}}$. Clearly the $\widetilde{G}_b$-torsor structure on $\overline{\mathrm{Ig}^{b,\diamond}}\to \overline{\Igs^b}$ is given by the action of $\widetilde{G}_b$ via its identification with $\underline{\operatorname{Aut}}_G(\tilde{\mathbb{X}}_b)^\diamond$.
\end{proof}

\hfill

\section{Hecke action} 	 
    This short section is devoted to the part of Conjecture \ref{conjecture} regarding the away-from-$p$ Hecke action on the Igusa stack. We adopt the notation from Section 5 and fix the level at $p$ to be $K_p=G_{\Z_p}(\Z_p)$. The adelic group $G(\Af^p)$-acts on the inverse system $\{S_{K_pK^p}\}_{K^p}$ as follows: for any $g\in G(\Af^p)$, there is a map between $\CO_E$-schemes
\[\gamma_g: S_{K_pK^p}\rightarrow S_{K_pg^{-1}K^pg}\]
sending a tuple $(A,\iota,\lambda,\bar{\eta}=K^p\cdot \eta)$ to $(A,\iota, \lambda, g^{-1}\bar{\eta}=g^{-1}K^pg\cdot g^{-1}\eta)$. By Lan \cite[Section 7.2.5]{Lan13}, this action extends to the system of minimal compactifications $\{S_{K_pK^p}^*\}_{K^p}$. By taking attached formal schemes, adic spaces or diamonds, we have a $G(\Af^p)$-action on all these variants of Shimura varieties.

Similarly, since the action is simply on the away-from-$p$ level structures, we can define an action of $G(\Af^p)$ on the system $\{\Igs_{K^p}^\circ\}_{K^p}$ where $K^p$ runs through compact open subgroups of $G(\Af^p)$: for any $g\in G(\Af^p)$, define a map of fibered categories
\[\Igs^{\text{pre},\circ}_{K^p}\rightarrow \Igs^{\text{pre},\circ}_{g^{-1}K^pg}\]
which sends a tuple $(A_0,\iota,\lambda,\bar{\eta})$ to $(A_0,\iota, \lambda, g^{-1}\bar{\eta})$ and a quasi-isogeny between two tuples to the same quasi-isogeny. This induces a map of v-stacks
\[\bar{\gamma}_g^\circ:\Igs^\circ_{K^p}\rightarrow \Igs^\circ_{g^{-1}K^pg}.\]

Under the assumption that the minimal compactification boundary of the Shimura variety has codimension at least two, see our classification in Proposition \ref{classification}, the Hecke action extends to the minimal compactifications 
$\{\Igs_{K^p}^\ast\}_{K^p}$, which is clear from the formula of $\Igs^*_{K^p}$ given in Definition \ref{construction}. By checking on the moduli problems, we have

\begin{prop}\label{Hecke}
Let $K=K_pK^p$ with $K_p=G_{\Z_p}(\Z_p)$ as before. For any $g\in G(\Af^p)$, the following diagrams of v-stacks over $\Bun_G$ commute (on the nose, as all maps to $\Bun_G$ are $0$-truncated)
\[
\begin{tikzcd}
\Shi^\ast_{K}\ar[r,"\gamma^\ast_g"]\ar[d] & \Shi^\ast_{g^{-1}Kg}\ar[d] & & \Igs^\ast_{K^p} \ar[d,"\bar{\gamma}^\ast_g"] \ar[dr,"\bar{\pi}^\ast_{HT}"]&\\
\Igs^\ast_{K^p}\ar[r,"\bar{\gamma}^\ast_g"] & \Igs^\ast_{g^{-1}K^pg}& & \Igs^\ast_{g^{-1}K^pg} \ar[r,"\bar{\pi}^\ast_{HT}"]& \Bun_G
\end{tikzcd}
\]
\end{prop}
\begin{proof}
The right diagram is commutative by construction. For the commutativity of the left diagram, it suffices to pullback the right diagram along $BL:\Gr_{G,\mu}\rightarrow \Bun_G$. When restricted to the good reduction locus of the Shimura variety and the open substack $\Igs^\circ_{K^p}$, it is clear that $\bar{\gamma}_g^\circ$ pulls back to $\gamma_g$. But as we have shown in \Cref{cpf}, on a strictly totally disconnected test object $T$, we have 
\[\Shi_{K^p}^\ast\times_{\Gr_{G,\mu}}T \cong \overline{(\Shi_{K^p}^\circ\times_{Gr_G}T)_0}^{/T},\]
so the map $\gamma_g^\ast$ is uniquely determined by its restriction to $\Shi_{K^p}^\circ$. Hence it must agree with the pullback of $\bar{\gamma}^\ast_g$.
\end{proof}


\hfill

\section{Integral model}
     Consider the formal integral model $\mathscr{S}_K$ over $\Spf \CO_E$ of the Shimura variety at hyperspecial level and its attached $v$-sheaf $\SK$ over $\Spd \CO_E$ as in Section 5. We have the reduction map
\[\mathrm{red}: \SK\rightarrow \Igs_{K^p}^\circ\]
from Remark \ref{IntRed}. The goal of this section is to extend the fiber product structure on $\Shi_{K}^\circ$ to this v-sheaf integral model, substituting the left vertical map in the cartesian diagram in Corollary \ref{LevelK} by the above map. In the integral model diagram the map $\overline{\pi}^\circ_{HT}: \Igs^\circ_{K^p}\rightarrow \Bun_G$ remains unchanged, but the minuscule Schubert cell $\Gr_{G,\mu,K_p}$ will be replaced by a moduli stack of shtukas with extra structures. The first two subsections are purely local and work for general reductive groups. Starting from Section 11.3 we switch back to the global PEL-setup. We begin by some recollection, see Remark \ref{Y} for the notation $\mathcal{Y}$, $\mathcal{Y}_I$. 

\subsection{\texorpdfstring{$\G$}{}-torsors, shtukas and BKF-modules}

Let $G_{\Z_p}$ over $\Z_p$ be a reductive group scheme. We denote by $\G$ the adic space representing hte functor that takes an adic space $S$ over $\Z_p$ to $G_{\Z_p}(\CO_S(S))$. If $G_{\Z_p}=\Spec(A)$, then $\G=\Spa(A,A^+)$, where $A^+$ is the integral closure of $\Z_p$ in $A$, equipped with discrete topology. We also write $\G(\Z_p)$ for $G_{\Z_p}(\Z_p)$ and the notation is understood as the $\Spa \Z_p$ points of $\G$. We have the notion of $\G$-torsors on sousperfectoid analytic adic spaces as in Definition \ref{G-bundle}.

\begin{defn}[{\cite[Definition 19.5.2]{Berkeley}}]
Let $X$ be a sousperfectoid analytic adic space over $\Z_p$. A $\G$-torsor $\CP$ is an \'etale sheaf on X with a $\G$-action which is \'etale locally $\G$-equivariantly isomorphic to $\G$. Equivalently, this is an exact tensor functor 
\[\mathrm{Rep}G_{\Z_p}\rightarrow \Bun(X),\]
from the exact tensor category of algebraic representations of $G_{\Z_p}$ on finite free $\Z_p$-modules to that of vector bundles on $X$. 
\end{defn}

Let $S=\Spa(R,R^+)\in \Perf/\Spd \Z_p$ be an affinoid perfectoid space of characteristic $p$ with an untilt $S^\sharp$ over $\Z_p$. We view it as a closed Cartier divisor on the analytic adic space $\mathcal{Y}_{[0,\infty)}(S)$. Write $\varphi_S$ for the Frobenius on $\mathcal{Y}_{[0,\infty)}(S)$.

\begin{defn}[{\cite[Definition 11.4.1]{Berkeley}}]
A shtuka over $S$ with one leg at $S^\sharp$ is a pair $(\CP,\varphi_\CP)$, where $\CP$ is a vector bundle on $\mathcal{Y}_{[0,\infty)}(S)$; and $\varphi_\CP$ is an isomorphism
\[\varphi_S^\ast\CP|_{\mathcal{Y}_{[0,\infty)}(S)\backslash S^\sharp}\cong \CP|_{\mathcal{Y}_{[0,\infty)}(S)\backslash S^\sharp},\]
which is meromorphic along $S^\sharp$ in the sense of \cite[Definition 5.3.5]{Berkeley}. A map between two shtukas $(\CP,\varphi_\CP)$ and $(\CP',\varphi_{\CP'})$ is a map of vector bundles $f:\CP\rightarrow \CP'$ such that $\varphi_{\CP'}\circ \varphi_S^*f = f\circ \varphi_\CP$.
\end{defn}

\begin{defn}[{\cite[Section 23.1]{Berkeley}}]\label{shtuka}
A $\G$-shtuka over $S$ with one leg at $S^\sharp$ is a pair $(\CP,\varphi_\CP)$,
where $\CP$ is a $\G$-torsor on $\mathcal{Y}_{[0,\infty)}(S)$; and $\varphi_\CP$ is an isomorphism

\[\varphi_S^\ast \CP|_{\mathcal{Y}_{[0,\infty)}(S) \backslash S^\sharp} \cong \CP|_{\mathcal{Y}_{[0,\infty)}(S) \backslash S^\sharp},\]
which is meromorphic along $S^\sharp$. A map between two $\G$-shtukas $(\CP,\varphi_\CP)$ and $(\CP',\varphi_{\CP'})$ is a map of $\G$-torsors $f:\CP\rightarrow \CP'$ such that $\varphi_{\CP'}\circ \varphi_S^\ast f = f\circ \varphi_\CP$.
\end{defn}

We have a notion of boundedness: recall from \cite[Definition 20.3.1]{Berkeley} that the mixed characteristic Beilinson-Drinfeld affine Grassmannian attached to $\G$ is the small v-sheaf $\Gr_{\G}$ over $\Spd\Z_p$ sending $S\in \Perf$ to the set of triples consisting of an untilt $S^\sharp$, a $\G$-torsor $\mathcal{F}$ on $\mathcal{Y}_{[0,\infty)}(S)$ and a trivialization $\alpha$ of $\mathcal{F}|_{\mathcal{Y}_{[0,\infty)}(S)\backslash S^\sharp}$, meromorphic along $S^\sharp$. Or equivalently this sends $S$ to the set of triples $\{(S^\sharp, \mathcal{F},\alpha)\}$, where $S^\sharp$ is an untilt, $\mathcal{F}$ is a $G_{\Z_p}$-torsor on $\Spec(\Bdr^+(R^\sharp))$ and $\alpha$ is a trivialization over $\Spec(\Bdr(R^\sharp))$, where if $R^\sharp$ has characteristic $p$, $\Bdr^+(R^\sharp)$ is defined to be $W(R^\sharp)$ and $\Bdr(R^\sharp)$ is $W(R^\sharp)[1/p]$.

Fix a maximal torus and a Borel $T\subset B\subset G_{\overline{\Q}_p}$. Let $\mu: \Gm_{\overline{\Q}_p}\rightarrow G_{\overline{\Q}_p}$ be a dominant cocharacter, whose $G(\overline{\Q}_p)$-conjugacy class is defined over some finite extension $E$ over $\Q_p$. (Later we will take $\mu$ to be the dominant cocharacter and $E$ the local field determined by the global PEL-datum, so we do not introduce new notation.) We have as in \cite[Definition 20.3.5]{Berkeley}, a Schubert variety $\Gr_{\G, \CO_E, \leq \mu}\subset \Gr_{\G, \CO_E}$ over $\Spd \CO_E$, which is the subfunctor where geometric-pointwise on a test object $S$, the relative position of $\mathcal{F}$ and the trivial $\G$-torsor is bounded by $\mu$, i.e. if the point is of characteristic $p$, the pair $(\mathcal{F},\alpha)$ lies in the closure of the Schubert cell labeled by $\mu$ in the Witt vector affine Grassmannian; otherwise in that of the $\Bdr^+$-affine Grassmannian. 

\begin{defn}
    Given $S\in \Perf$ with an untilt $S^\sharp$ over $\Spa \CO_E$, we say a $\G$-shtuka $(\CP,\varphi_\CP)$ over $S$ with one leg at $S^\sharp$ is \textit{bounded by $\mu$}, if geometric-pointwise on $S$, the relative position of $\varphi_S^\ast\CP$ and $\CP$, completed along $S^\sharp$ is bounded by $\mu$. Namely, for any geometric point $\bar{s}\in S$, choose a trivialization $\G\cong (\varphi_S^\ast\CP)_{\bar{s}^\sharp}$ of the stalk of the $\G$-torsor $\varphi_S^\ast\CP$ at $\bar{s}^\sharp$, \footnote{This is possible by smoothness of the group scheme $G_{\Z_p}$ and henselianness of the completion of $\CO_{\mathcal{Y}_{[0,\infty)}(S)}$ at $\bar{s}^\sharp$.} so that the pair $(\CP_{\bar{s}^\sharp}, \varphi_{\CP,\bar{s}^\sharp})$ defines an $\bar{s}$-point of $\Gr_{\G,\CO_E}$. Then this point lies in $\Gr_{\G,\CO_E,\leq \mu}$.
\end{defn}

\begin{remark}\label{rem: SignShtuka}
    If $G_{\Z_p}$ is a smooth parahoric group scheme, one can also define $\G$-shtukas with a notion of boundedness by $\mu$ using the local model $\mathbb{M}_{\mathcal{G},\mu}$, see \cite[Section 2.4.4]{PR}. Note also that our sign convention agrees with \textit{loc. cit.}. Namely, we  trivialize the stalk of $\varphi_S^\ast \mathcal{P}$ instead of that of $\mathcal{P}$. This leads to the slightly confusing fact that the generic fiber of $\mathrm{Sht}_{\mathcal{G},\mu}$ we consider in Definition~\ref{def: ShtukaGmu} is the Schubert cell $[\Gr_{G,\mu^{-1}}/\underline{K_p}]$. We choose this sign convention, so that the local model has the same cocharacter as the Hodge cocharacter of the Shimura variety, as in \cite[Section 4.9]{PR}. This sign convention is necessary, if one wants the local model diagram for the Shimura variety to be compatible with the classical Borel embedding on the generic fiber, as stated in \cite[Section 4.9.2]{PR}, cf. Remark~\ref{rem: SignHT}. 
\end{remark}

As objects with more favorable algebraic properties, we introduce Breuil-Kisin-Fargues modules with $\G$-structure and discuss their relation to $\G$-shtukas. 

\begin{defn}
\label{BKF}
    Let $S$, $S^\sharp$ be as above. Write $\varphi_S$ for the Frobenius on $W(R^+)$ and $\xi$ is a chosen generator of the kernel of Fontaine's map $\theta: W(R^+)\rightarrow R^{\sharp+}$.\footnote{Recall from \cite[Definition 3.5, Lemma 3.20]{BS} that $\operatorname{ker}\theta$ is principal.} A \textit{Breuil-Kisin-Fargues module (BKF-module)}, respectively a $\G$-BKF-module, over $S$ with a leg at $S^\sharp$ is a pair $(M,\varphi_M)$ consisting of a finite projective $W(R^+)$-module $M$, respectively a $G_{\Z_p}$-torsor $M$ over $\Spec(W(R^+))$, together with an isomorphism 
    \[\varphi_M: (\varphi_S^\ast M)[1/\xi]\xrightarrow{\sim}M[1/\xi].\]
\end{defn}

A $\G$-BKF module $(M,\varphi_M)$ defines a $\G$-shtuka by the following construction: restrict $(M,\varphi_M)$ to the punctured spectrum $\Spec(W(R^+))\backslash V(I)$, where $I=(p,[\varpi])$. Using Tannakian formalism and the equivalence between the exact tensor category of vector bundles on $\Spec(W(R^+))\backslash V(I)$ and that on $\mathcal{Y}(S)$ due to Kedlaya (cf. \cite[Theorem 2.1.6 a)]{PR}), this gives a $\G$-torsor with meromorphic Frobenius on the latter. Further restricting to $\mathcal{Y}_{[0,\infty)}(S)$, one obtains a $\G$-shtuka. In the rest of this subsection, we will show that when $S$ is a product of rank one geometric points, all $\G$-shtukas come from $\G$-BKF-modules in such a way. For this we need the following key input from the recent work of Gleason--Ivanov--Zillinger \cite{GIZ}.\footnote{During the time this paper is being reviewed, Gleason--Ivanov--Zillinger's paper has been updated. We refer here to results in the first Arxiv version of their paper. This is to avoid circulation, since in Remark 1.10 in their second version, they have referred to this paper.}

\begin{theorem}[{\cite[Version 1, Theorem 8.6]{GIZ}, cf. \cite[Version 1, Corollary 1.9]{GIZ}}]\label{VBextend}
    Let $S$ be a product of rank one geometric points with untilt $S^\sharp$, then the tensor category of shtukas over $S$ with a leg at $S^\sharp$ is equivalent to that of vector bundles with meromorphic Frobenius over $\mathcal{Y}(S)=\mathcal{Y}_{[0,\infty]}(S)$.
    
\end{theorem}

\begin{cor} \label{Gextend}
    Let $S$ be a product of rank one geometric points with untilt $S^\sharp$, then the category of $\G$-shtukas over $S$ with a leg at $S^\sharp$ is equivalent to that of $\G$-torsors with meromorphic Frobenius over $\mathcal{Y}(S)$.
\end{cor}
\begin{proof}
Combine Lemma \ref{VBextend} and Tannakian formalism: since the locus where $p=0$ has nothing to do with the extension we are interested in, we can restrict a given shtuka to $\mathcal{Y}_{(0,\infty)}(S)$ and consider the statement there. But $\G$-torsors on $\mathcal{Y}_{(0,\infty]}(S)$ are the same as exact tensor functors from the category of finite dimensional algebraic $\Q_p$-representations of $G_{\Q_p}$ to vector bundles on $\mathcal{Y}_{(0,\infty]}(S)$. Since the representation category is semi-simple, the condition on exactness is vacuous and the tensor equivalence for vector bundles implies the equivalence for $\G$-torsors.
\end{proof}

\begin{lemma}\label{trivial}
    Let $S=\Spa(R,R^\circ)$ in $\Perf$ be a product of rank one geometric points. Then $G_{\Z_p}$-torsors on $\Spec(R)$ are trivial.
\end{lemma}
\begin{proof}
    Pullback along the map of locally ringed spaces
    \[S\rightarrow \Spec(R)\]
    defines an exact equivalence between the category of vector bundles on both sides by \cite[Theorem 1.4.2]{Kedlaya}. One checks directly that this equivalence is symmetric monoidal. Hence using Tannakian formalism, we obtain an equivalence between the category of $\G$-torsors on $S$ and that of $G_{\Z_p}$-torsors on $\Spec(R)$. Since $S$ is strictly totally disconnected, \'etale $\G$-torsors on $S$ are trivial, so are $G_{\Z_p}$-torsors on $\Spec(R)$. 
\end{proof}

\begin{prop}
\label{extension}
    Let $S=\Spa(R,R^\circ)$ in $\Perf$ be a product of rank one geometric points with an untilt $S^\sharp$ over $\Z_p$. Then pullback along 
    \[\mathcal{Y}_{[0,\infty)}(S)\rightarrow \Spa(W(R^\circ), W(R^\circ)) \rightarrow \Spec(W(R^\circ)),\]
    defines an equivalence of categories between $\G$-BKF-modules over $S$ with a leg at $S^\sharp$ and $\G$-shtukas over $S$ with a leg at $S^\sharp$.
    
    If we write $R^\circ=\prod_{i\in I} \CO_{C_i}$, where $\CO_{C_i}$ is the ring of integers in some complete algebraically closed non-archimedean field $C_i$, and let $s_i:=\Spa(C_i,\CO_{C_i})$. Then the inverse equivalence is given by the following construction: given a $\G$-shtuka $(\CP,\varphi_{\CP})$, restrict it to $\mathcal{Y}_{[0,\infty)}(s_i)$ for each $i$. The restriction $(\CP_i,\varphi_{\CP_i})$ extends uniquely to a $\G$-BKF-module $(M_i,\varphi_{M_i})$ by \cite[Proposition 2.4.6]{PR}. Their product $(M,\varphi_M)$ 
    is the desired $\G$-BKF module.
\end{prop}

\begin{proof}
      The restriction functor from $\G$-BKF-modules to $\G$-torsors on $\mathcal{Y}_{[0,\infty]}(S)$ is fully faithful, by combining Tannakian formalism with the following results: the restriction of $G_{\Z_p}$-torsors from $\Spec(W(R^\circ))$ to the punctured spectrum \[\Spec(W(R^\circ))\backslash V(I)\] is fully faithful, see \cite[Lemma 8.4]{Ans22}, and the pullback functor on the category of vector bundles along the map of locally ringed spaces
      \[\mathcal{Y}(S)\rightarrow \Spec(W(R^\circ))\backslash V(I)\]
      is an exact tensor equivalence due to Kedlaya, see \cite[Theorem 2.1.6]{PR}.
      
      To show essential surjectivity, given a $\G$-shtuka $(\CP,\varphi_\CP)$ over $S$ with one leg at $S^\sharp$, we apply Lemma \ref{Gextend} to extend it to a $\G$-torsor with meromorphic Frobenius $(\tilde{\CP},\varphi_{\Tilde{\CP}})$ over $\mathcal{Y}(S)$. It suffices to show that the underlying $\G$-torsor of $\Tilde{\CP}$ is trivial, so that we can take the trivial extension to get a $\G$-BKF-module. For this, we adapt an argument from \cite[Proposition 9.2]{Ans22}, which deals with the case of $S$ being a point.
     
     By GAGA and Beauville-Laszlo gluing, we may view $\tilde{P}$ as a $G_{\Z_p}$-torsor on the scheme 
     \[\Spec(W(R^\circ))\backslash V(I),\] 
     glued from $G_{\Z_p}$-torsors on $\Spec(W(R^\circ)[1/p])$ and on $\Spf(W(R))$ over $W(R)[1/p]$. But $G_{\Z_p}$-torsors on both are trivial: on $W(R)$ it follows from Lemma \ref{trivial}, smoothness of $G_{\Z_p}$ and henselianness of the pair 
     \[(W(R),\text{ker}(W(R)\twoheadrightarrow R)),\] 
     while triviality on $W(R^\circ)[1/p]$ is proven in \cite[Proposition 11.5]{Ans22}.

     Hence $\tilde{P}$ can be described by an element of the double coset 
     \[G_{\Z_p}(W(R))\backslash G_{\Q_p}(W(R)[1/p])/G_{\Q_p}(W(R^\circ)[1/p]).\] By comparing with the presheaf of the Witt vector affine Grassmannian $\Gr_{G_{\Z_p}}^W$ (\cite[Definition 20.3.3]{Berkeley}), the above set (in fact its \'etale sheafification on $\Spec(R)$, which has the same value) measures $R$-points of $\Gr_{G_{\Z_p}}^W$ that do not come from $R^\circ$-points. But given any $R$-point of $\Gr_{G_{\Z_p}}^W$, its restriction to $C_i$ extends uniquely to an $\CO_{C_i}$-point, since $\Gr_{G_{\Z_p}}^W$ is ind-proper (see \cite[Theorem 1.4]{Zhu}). These $\CO_{C_i}$-points in turn define a unique $\Spec(R^\circ)$-point of $\Gr_{G_{\Z_p}}^W$ as below: assume the original $R$-point lies in some proper subscheme $X$. Take a finite affine open cover $\{X_j\}$, $j=1,\dots,n$ of $X$. We obtain a finite partition 
     \[I=\coprod_{j=1}^n I_j\]
     of the index set $I$, where $i\in I_j$ if the image of $\Spec(\CO_{C_i})$ lies in $X_j$. Define $R^\circ_j:=\prod_{i\in I_j}\CO_{C_i}$. Since $X_j$ is affine, the collection of maps $\Spec(\CO_{C_i})\rightarrow X_j$, $i\in I_j$ determines a unique map 
     \[\Spec(R^\circ_j)\rightarrow X_j.\]
     As $\{\Spec(R^\circ_j)\}$, $j=1,\dots, n$ is a cover of $\Spec(R^\circ)$ by open-and-closed subspaces, we get a unique map
     \[\Spec(R^\circ)\rightarrow \Gr_{G_{\Z_p}}^W.\]
     Moreover, by checking for each $j$ we see that when restricted along $\Spec(R)\hookrightarrow \Spec(R^\circ)$, we get back the given $R$-point. This shows that the double coset we considered is a singleton and $\tilde{\CP}$ is trivial as desired.

     To see that the inverse equivalence takes the stated form, assume $(\CP,\varphi_{\CP})$ extends to $(M',\varphi_{M'})$. Then the restriction $(M'_i,\varphi_{M'_i})$ to each $s_i$ extends the restriction $(\CP_i,\varphi_{\CP_i})$ of the given $\G$-shtuka. By uniqueness, $(M'_i,\varphi_{M'_i})=(M_i,\varphi_{M_i})$ and hence $(M',\varphi_{M'}) = (M,\varphi_{M})$.
\end{proof}

\begin{remark} \label{ShtProd}
    Over a product of rank one geometric points $S$ as in the proposition, by going through the equivalences 
    \begin{align*}
        \{\G\text{-shtukas over }S\}
    &\Leftrightarrow \{\G\text{-BKF-modules over }S\}\\
    &\Leftrightarrow {\prod_i}'\{\G\text{-BKF-modules over }s_i\}\\
    &\Leftrightarrow {\prod_i}'\{\G\text{-shtukas over }s_i\},
    \end{align*}
    we see that a $\G$-shtuka over $S$ is uniquely determined by its restriction to each $s_i$. Here we use ${\prod_i}'$ to mean the full subcategory of the product category in which an object is a collection of $\G$-BKF-modules, respectively $\G$-shtukas, for all $i\in I$, that are commonly bounded by some cocharacter $\lambda$.\footnote{We define boundedness of $\G$-BKF-modules by that of their attached $\G$-shtukas.} If a $\G$-shtuka $(\CP,\varphi_{\CP})$ is sent to the collection $(\CP_i,\varphi_{\CP_i})$ under the above composition of equivalences, we will call $(\CP,\varphi_{\CP})$ the product of $(\CP_i,\varphi_{\CP_i})$'s. 
\end{remark}

\subsection{Moduli of \texorpdfstring{$\G$}{}-shtukas}
We now define the moduli stack of $\G$-shtukas and record some of its geometric properties.

\begin{defn}
Let $\Sht$ be the presheaf of groupoids on the v-site of $\Perf/\Spd \Z_p$ sending $S$ to the groupoid of $\G$-shtukas over $S$ with a leg at $S^\sharp$, where $S^\sharp$ is the untilt of $S$ over $\Spa \Z_p$ determined by the structure map $S\rightarrow \Spd \Z_p$. This is a v-stack by \cite[Proposition 19.5.3]{Berkeley}.
\end{defn}

\begin{prop}\label{qsness}
    The structure map $\Sht\rightarrow\Spd \Z_p$ is quasi-separated. For any perfectoid Tate algebra $R$ with an open bounded integrally closed subring $R^+$ and any commutative diagram with solid arrows
    \[
    \begin{tikzcd}
    \Spa(R,R^\circ)\ar[r,"f"]\ar[d]
    & \Sht \ar[d]\\
    \Spa(R,R^+)\ar[r]\ar[ru,dashed]
    & \Spd \Z_p,
    \end{tikzcd}
    \]
    there is a unique (up to isomorphism) dotted arrow making the whole diagram commute up to a natural automorphism of $f$.
\end{prop}

\begin{proof}
    We first check that the diagonal map is quasi-separated. 
    
    For this, we need to show that the diagonal map from $\Sht$ to the inertia stack is quasi-compact. This is equivalent to the fact that over an affinoid perfectoid test object $T$ with untilt $T^\sharp$, if we are given a $\G$-shtuka $(\CP,\varphi_{\CP})$ over $T$ with one leg at $T^\sharp$ together with an automorphism $g\in \underline{\mathrm{Aut}}(\CP)(T)$, then the locus on $T$ where $g$ equals the neutral element $e$ is quasi-compact. Here we use $\underline{\mathrm{Aut}}(\CP)$ to denote v-sheaf over $S$ of automorphisms of the $\G$-shtuka $(\CP,\varphi_\CP)$:
    \[S\mapsto \mathrm{Aut}((\CP,\varphi_\CP))(\mathcal{Y}_{[0,\infty)}(S)).\]
    It suffices to show that the structure map $\underline{\mathrm{Aut}}(\CP)\to S$ is separated, because then the sections $g$ and $e$ defines closed immersions $S\to\underline{\mathrm{Aut}}(\CP)$ and the intersection
    \[\lim(S\xrightarrow{g} \underline{\mathrm{Aut}}(\CP) \xleftarrow{e} S)\]
    is a closed (in particular quasicompact) sub-v-sheaf in $S$. For this, pick a faithful representation $V$ of $\G$ and consider the vector bundle shtuka $\mathcal{V}:=(\CP,\varphi_\CP)\times^\G V$. Since $\underline{\mathrm{Aut}}(\CP)$ is a sub-v-sheaf of $\underline{\mathrm{End}}(\mathcal{V})$, it suffices to prove separatedness of the latter over $S$. This again is a sub-v-sheaf of the Banach-Colmez space for the vector bundle on $X_S$ obtained by restricting $\mathcal{V}$ to $\mathcal{Y}_{[r,\infty)}(S)$ for large enough $r$ (see Proposition \ref{generic fiber} below). But this Banach-Colmez space is separated over $S$ according to \cite[Proposition II.2.16]{FS}. This finishes the proof of quasiseparatedness.

    For quasi-compactness, take any map from an affinoid perfectoid space $X$ (without loss of generality of the form $\Spa(R,R^\circ)$ with a pseudo-uniformizer $\varpi\in R$) to $\Sht\times_{\Spd \Z_p}\Sht$ and consider the cartesian diagram
    \[
    \begin{tikzcd}
        Y\ar[r]\ar[d] & X\ar[d]\\
        \Sht\ar[r,"\Delta"] & \Sht\times_{\Spd \Z_p}\Sht.
    \end{tikzcd}
    \]
    We need to show $Y$ is quasi-compact. For this, combine Remark \ref{ShtProd} and Criterion \ref{qcqs}: Fix a representative $y=\Spa(C_y, C_y^+)$ for each geometric point of $Y$, the composition $y\rightarrow Y \rightarrow X$ determines a map of affinoid rings $(R,R^\circ)\rightarrow (C_y,C_y^+)$ such that the image of $\varpi$ is some pseudo-uniformizer $\varpi_y$. Form a product of points $S$ by letting 
    \[A^+:= \prod_{y\in|Y|} C_y^+,\, \varpi_A:=(\varpi_y), \, A:=A^+[\frac{1}{\varpi_A}],\]
    and taking $S$ to be $\Spa(A,A^+)$. This comes equipped with a map $S\rightarrow X$ induced by $R^+\rightarrow \prod_{y\in|Y|} C_y^+$. It suffices to show this map factors through a surjection of v-sheaves $S\rightarrow Y$. The map factoring through $Y$ amounts to the fact that the two $\G$-shtukas $(\CP,\varphi_{\CP})$, $(\CP',\varphi_{\CP'})$ on $S$ obtained by pulling back along $S\rightarrow X\rightarrow \Sht\times_{\Spd \Z_p}\Sht$ are isomorphic. But by construction, $(\CP_y,\varphi_y)\cong(\CP'_{y},\varphi_{\CP'_y})$ for each $y\in |Y|$ and hence Remark \ref{ShtProd} shows the product of these isomorphisms gives an isomorphism $(\CP,\varphi_\CP)\cong(\CP',\varphi_{\CP'})$. For surjectivity, since $S\rightarrow Y \rightarrow X$ is qcqs and the second map is quasi-separated by quasi-separatedness of $\Delta$, $S\rightarrow Y$ is qcqs by cancellation. It is surjective on topological spaces by construction and hence a surjection of v-sheaves \cite[Lemma 17.4.9]{Berkeley}. This finishes the proof of quasi-compactness.

    For the last claim, use Tannakian formalism and the tensor exact equivalence for restriction of vector bundles along
    \[\mathcal{Y}_{[0,\infty)}(R,R^\circ)\rightarrow\mathcal{Y}_{[0,\infty)}(R,R^+).\]
\end{proof}

\begin{cor}\label{ProperDiag}
    The diagonal $\Delta: \Sht\rightarrow \Sht\times_{\Spd\Z_p}\Sht$ is proper.
\end{cor}
\begin{proof}
    Use the characterization of properness by \cite[Proposition 18.3]{Sch18}. Since the diagonal is 0-truncated, qcqs, it suffices to check the valuative criterion, i.e. for any perfectoid field $K$, with ring of integers $\CO_K$ and open bounded valuation subring $K^+\subset K$ and any commutative diagram 
    \[
    \begin{tikzcd}
        \Spa(K,\CO_K)\ar[r,"x"]\ar[d]& \Sht \ar[d,"\Delta"]\\
        \Spa(K,K^+)\ar[r]\ar[ru,dashed,"\tilde{x}"]& \Sht\times_{\Spd \Z_p}\Sht,
    \end{tikzcd}
    \]
    there should be a unique dotted arrow making the diagram commute. Post-composing with the structure maps to $\Spd \Z_p$, we obtain a unique up to isomorphism $\tilde{x}$ from the valuative criterion for $\Sht\rightarrow \Spd \Z_p$, making the upper left triangle commute up to an automorphism of $x$. But the commutativity of the lower right triangle rigidifies the situation, namely we can modify $\tilde{x}$ with a unique isomorphism making the diagram of sheaves over $\Sht\times_{\Spd \Z_p}\Sht$ commute. 
\end{proof}

Let $K_p$ be $G_{\Z_p}(\Z_p)$. To relate $\Sht$ to $\Bun_G$, we observe that the generic fiber of $\Sht$ can be identified with the quotient $\Gr_{G,K_p}:=[\Gr_G/\underline{K_p}]$, by rephrasing the moduli interpretation of $\Gr_{G,K_p}$ as below.

\begin{lemma}
The value of $\Gr_{G,K_p}/\Spd \Q_p$ on an affinoid perfectoid $S\in \Perf$, is the groupoid of isomorphism classes of tuples 
\[(S^\sharp, \mathscr{E}_0, \mathscr{E}, \mathbb{T}, \alpha: \mathscr{E}_0\dashrightarrow \mathscr{E}),\]
where $S^\sharp$ is an untilt of $S$ over $\Spa \Q_p$; $\mathscr{E}_0$, $\mathscr{E}$ are $G$-torsors on the relative Fargues-Fontaine curve $X_S$, with $\mathscr{E}_0$ being geometric pointwise on $S$ trivial; $\mathbb{T}$ is a pro-\'etale $\underline{K}_p$-torsor over $S$ such that 
\[\mathscr{E}_0 = \mathbb{T}\times^{\underline{K}_p}(\G\times X_S);\]
and $\alpha$ is a meromorphic isomorphism over $X_S\backslash S^\sharp$.
\end{lemma}
\begin{proof}
Note that one can $\underline{K}_p$-equivariantly identify the moduli interpretation of $\Gr_G$ with the v-sheaf of trivializations $\underline{\mathrm{Isom}}_S(\mathbb{T},\underline{K}_p)$ over the above moduli problem.
\end{proof}

\begin{prop}
\label{generic fiber}
The generic fiber $\mathrm{Sht}_{\G,\Q_p}$ can be identified with $\Gr_{G,K_p}/\Spd \Q_p$. In particular, the Beauville-Laszlo map on $\Gr_G$ factors through a map 
\[\mathrm{Sht}_{\G,\Q_p}\rightarrow \Bun_G.\]
\end{prop}
\begin{proof}
Take $S=\Spa(R,R^+)$ in $ \Perf$. An $S$-point of $\mathrm{Sht}_{\G,\Q_p}$ gives an untilt $S^\sharp$ of $S$ over $\Spa \Q_p$ and a $\G$-shtuka $(\CP, \varphi_\CP)$ over $S$ with one leg at $S^\sharp$. Restrict $(\CP, \varphi_\CP)$ to $\mathcal{Y}_{(r,\infty)}(S)$ for $r$ large enough such that $\mathcal{Y}_{(r,\infty)}(S)$ does not meeting $S^\sharp$. This descends to a $\G$-torsor $\mathscr{E}$ on $X_S$. Similarly the restriction of $(\CP, \varphi_\CP)$ to $\mathcal{Y}_{(0,\epsilon]}(S)$ for some $\epsilon$ with $\mathcal{Y}_{(0,\epsilon]}(S)$ not meeting $S^\sharp$, descends to a $\G$-torsor $\mathscr{E}_0$ on $X_S$. By \cite[Proposition 22.6.1, 23.3.1]{Berkeley}, as its pullback to $\mathcal{Y}_{(0,\epsilon]}(S)$ extends $\varphi_S^{-1}$-equivariantly over the locus $p=0$, $\mathscr{E}_0$ is geometric pointwise on $S$ trivial and there is a pro-\'etale $\underline{K}_p$-torsor $\mathbb{T}$ on $S$ such that 
\[\mathscr{E}_0 = \mathbb{T}\times^{\underline{K}_p}(\G\times X_S).\]
Furthermore, for $n\in \N$ large enough, $\varphi_\CP^n$ induces an isomorphism
\[\alpha: \mathscr{E}_0|_{X_S\backslash S^\sharp}\cong(\varphi_S^{n})^*\CP|_{\mathcal{Y}_{(r,pr)}}\xrightarrow{\varphi_\CP^n} \CP|_{\mathcal{Y}_{(r,pr)}} \cong \mathscr{E}|_{X_S\backslash S^\sharp}.\] 
The tuple $(S^\sharp, \mathscr{E}_0, \mathscr{E}, \mathbb{T}, \alpha)$ is an object of $\Gr_{G,K_p}(S)$ and the assignment $(S^\sharp, \CP,\varphi_{\CP})$ to $(S^\sharp, \mathscr{E}_0, \mathscr{E}, \mathbb{T}, \alpha)$ preserves isomorphisms. 

Conversely, pull back $\mathscr{E}_0$ to $\mathcal{Y}_{(0,\infty)}(S)$. By \cite[Section 22.6]{Berkeley} and the existence of the $\underline{K}_p$-torsor $\mathbb{T}$, this extends to a $\G$-shtuka $\CP'$ with no legs. On the other hand, the completion of $\mathscr{E}$ along $S^\sharp \subset X_S$ gives a $\Bdr^+(R^\sharp)$-lattice. We may use $\alpha$ to modify $\CP'$ by this lattice at $\varphi_S^n(S^\sharp)$ for all $n\geq 1$. By doing so we obtain a new $\G$-torsor $\CP$, together with a meromorphic map
\[\varphi_\CP: \varphi_S^\ast\CP\dashrightarrow \CP,\]
which is $\varphi_{\CP'}$ at $\mathcal{Y}_{[0,\infty)}(S)\backslash \varphi_S^n(S^\sharp)$, $n\geq 0$, identity at $\varphi_S^n(S^\sharp)$, $n\geq 1$ and $\alpha^{-1}$ at the leg. The pair $(\CP,\varphi_\CP)$ is a $\G$-shtuka over $S$ with one leg at $S^\sharp$. Clearly these two constructions are inverse to each other. 
\end{proof}

\begin{prop}
\label{VB}
There is a map 
\[BL^\mathrm{int}_{K_p}:\Sht\rightarrow \Bun_G\]
extending the map $\mathrm{Sht}_{\G,\Q_p}\rightarrow \Bun_G$ from Proposition \ref{generic fiber}.
\end{prop}
\begin{proof}[Construction.]
For $S=\Spa(R,R^+)$ in $\Perf$ with an untilt $S^\sharp$ over $\Spa\Z_p$, given a $\G$-shtuka $(\CP,\varphi_\CP)$ over $S$ with a leg at $S^\sharp$, restrict $\CP$ to $\mathcal{Y}_{[r,\infty)}(S)$ for large enough $r$ such that $S^\sharp$ does not lie in $\mathcal{Y}_{[r,\infty)}(S)$. Then the restriction $\CP|_{\mathcal{Y}_{[r,\infty)}(S)}$ with the descent datum provided by $\varphi_\CP$ defines an $S$-point of $\Bun_G$. Clearly if the leg is not in characteristic $p$, then this is the map induced by the Beauville-Laszlo map.
\end{proof}

We can define bounded substacks of $\Sht$. Fix $T\subset B\subset G_{\overline{\Q}_p}$ and let $\mu \in X_\ast(T)$ be a dominant cocharacter defined over a finite extension $E/\Q_p$ with ring of integers $\CO_E$.
\begin{defn}\label{def: ShtukaGmu}
    $\mathrm{Sht}_{\G,\leq \mu}$ is the closed substack of $\mathrm{Sht}_{\G,\CO_E}$ where the $\G$-shtukas are bounded by $\mu$. We write $\mathrm{Sht}_{\G,\mu}$ for $\mathrm{Sht}_{\G,\leq \mu}$ if $\mu$ is minuscule.
\end{defn}

\color{black}
\begin{theorem}\label{qcqsSchubert}
    Let $\mu$ and $E$ be as above. The structure map 
    \[\mathrm{Sht}_{\G,\leq \mu}\rightarrow \Spd\CO_E\]
    is qcqs, with proper diagonal, and for any perfectoid Tate algebra $R$ with an open bounded integrally closed subring $R^+$ and any commutative diagram with solid arrows
    \[
    \begin{tikzcd}
    \Spa(R,R^\circ)\ar[r,"f"]\ar[d]
    & \mathrm{Sht}_{\G,\leq \mu} \ar[d]\\
    \Spa(R,R^+)\ar[r]\ar[ru,dashed]
    & \Spd \CO_E,
    \end{tikzcd}
    \]
    there is a unique (up to isomorphism) dotted arrow making the whole diagram commute up to a natural automorphism of $f$.
\end{theorem}
\begin{proof}
    We only have to prove quasi-compactness; the rest is Proposition \ref{qsness} and Corollary \ref{ProperDiag}. But this follows from the qcqsness criterion Proposition~\ref{qcqs} and Remark \ref{ShtProd}.
\end{proof}

\subsection{The crystalline period map}

Back to the PEL-setup: $G_{\Z_p}$ is the reductive group determined by the quadruple $(\CO_{B_{\Q_p}}, \ast, \Lambda, (\cdot,\cdot))$; $E$ is the completion of the global reflex field at a fixed prime above $p$, and $\mu$ is a dominant representative in the conjugacy class of the inverse of the Hodge cocharacter. We show that the universal formal abelian scheme over the integral model of the Shimura variety gives rise to a $\G$-shtuka with one leg bounded by $\mu^{-1}$.\footnote{Note that the generic fiber of $\mathrm{Sht}_{\G,\mu^{-1}}$ is $\Gr_\mu$. This switch from $\mu^{-1}$ to $\mu$ reflects that the type of the Hodge filtration on the de Rham cohomology and the Hodge-Tate filtration on the \'etale cohomology are inverse to each other.} This defines a map $\SK\rightarrow \mathrm{Sht}_{\G,\mu^{-1}}$. We call it the crystalline period map. \footnote{The name is not optimal, since shtukas do not see the crystalline point on the spectrum of $\mathbb{A}_\mathrm{inf}$. Our justification is that the generic fiber of $\SK$ is the potentially crystalline locus.}

\begin{prop}
Assume $p\neq 2$, then there is a map of small v-stacks over $\Spd \CO_E$
\[\pi_\text{crys}: \SK\rightarrow \mathrm{Sht}_{\G,\mu^{-1}}.\]
whose base change to $\Spd E$ is the Hodge-Tate period map $\overline{\pi}^\circ_{HT}$ at level $K_p$. For $p=2$, the same statement is true if \cite[Theorem 3.16]{RZ} holds. 
\end{prop}

\begin{proof}[Construction]
View $\SK/\Spd \CO_E$ as the sheafification of the presheaf on $\Perf$
\[S=\Spa(R,R^+)\mapsto  \mathscr{S}_K(\Spf(R^{\sharp+})),\]
where $S^\sharp=\Spa(R^\sharp,R^{\sharp+})$ is the untilt of $S$ over $\CO_E$ determined by the structure morphism to $\Spd \CO_E$. 

For $S=\Spa(R,R^+)$ as above, we denote by $\varphi$ the Frobenius on $W(R^+)$, $\xi$ a generator of the kernel of Fontaine's theta map and write $X$ for $\Spec(W(R^+))$. Assume we have a map $\Spf(R^{\sharp+})\rightarrow \mathscr{S}_K$, denote the pullback of the universal formal abelian scheme by $\mathfrak{A}$. Its prismatic Dieudonn\'e module is a BKF-module $(M,\varphi_M)$ over $S$ with a leg at $\varphi(S^\sharp)$, equipped with an alternating form $(\cdot,\cdot)$ and an $\CO_B$-action. Let $\tilde{M}$ be the coherent sheaf on $X_\text{\'et}$ attached to the $W(R^+)$-module $M$. Consider the sheaf on $X_\text{\'et}$
\[\CP: T \mapsto \{g\in \mathrm{Isom}_{\CO_B}(\tilde{M}_T, \Lambda\otimes_{\Z_p} \CO_T)| g^\ast(\cdot,\cdot)=c(g)(\cdot,\cdot), c(g)\in \CO_T(T)^\times\}.\]
We want to show that this is a $G_{\Z_p}$-torsor, i.e. $\tilde{M}$ is \'etale locally on $X$ isomorphic to $\Lambda\otimes_{\Z_p} \CO_X$ as polarized $\CO_B\otimes_{\Z_{(p)}} \CO_X$-modules. 

Since both $M$ and $\Lambda\otimes_{\Z_p}W(R^+)$ are $\xi$-adically complete and $\xi$-adically separated, we can check this after reducing modulo $\xi$. (Here we use the smoothness part of \cite[Theorem 3.16]{RZ} and hence have to exclude $p=2$.) But $M\otimes_{W(R^+),\theta}R^{\sharp+}$
 agrees with the de Rham homology $H_{1,\text{dR}}(\mathfrak{A}/R^{\sharp+})$ of the formal scheme $\mathfrak{A}$ over $R^{\sharp+}$.
 For any chosen pseudo-uniformizer $\varpi\in R^{\sharp+}$ of $R^\sharp$ and each integer $n$, denote the reduction of $\mathfrak{A}$ modulo $\varpi^n$ by $A_n$ and $R_n:=R^{\sharp+}/\varpi^n$. We have 
 \[H_{1,\text{dR}}(\mathfrak{A}/R^{\sharp+})\cong \varprojlim_n H_{1,\text{dR}}(A_n/R_n).\]
 Since the Hodge filtration on $H_{1,\text{dR}}(A_n/R_n)$ is $\CO_B$-linear with graded pieces given by the Lie algebra of the abelian scheme $A_n$ and the dual of the Lie algebra of $A_n^\vee$, it is implied by the Kottwitz condition that each $H_{1,\text{dR}}(A_n/R_n)$ is isomorphic to $\Lambda\otimes_{\Z_p} R_n$ as polarized $\CO_B\otimes_{\Z_{(p)}}R_n$-modules. Passing to the limit we have the desired statement.

Now restrict $\varphi^\ast\CP$ to $\mathcal{Y}_{[0,\infty)}(S)$ and equip it with the Frobenius semi-linear endomorphism $\varphi_M\circ \varphi$, we obtain a $\G$-shtuka over $S$ with a leg at $S^\sharp$. It is bounded by $\mu^{-1}$ because of the shape of the Hodge-filtration as explained in Theorem \ref{HT}. This 
induces a map of v-stacks 
\[\SK\rightarrow \mathrm{Sht}_{\G,\mu^{-1}}.\]
Compare with the construction in Theorem \ref{HT} and use Proposition \ref{generic fiber}, we see that when restricted to the generic fiber of $\SK$ this is the Hodge-Tate period map. 
\end{proof}

\begin{remark}
    In \cite[Section 4.5-4.6]{PR}, Pappas-Rapoport showed in the more general case of Hodge type Shimura varieties at parahoric levels, integral models satisfying \cite[Conjecture 4.2.2 a)-c)]{PR} exist, and the universal $\G$-shtuka over the generic fiber extends uniquely over the integral model.
\end{remark}

\begin{prop}
    The map $\pi_\text{crys}$ is qcqs.
\end{prop}
\begin{proof}
    Since the composition 
    \[\SK \xrightarrow{\pi_\text{crys}} \mathrm{Sht}_{\G,\mu^{-1}} \rightarrow \Spd \CO_E\]
    is qcqs and the second map is quasi-separated by Theorem \ref{qcqsSchubert}, the first map is qcqs by cancellation.
\end{proof}

\subsection{Integral model of the cartesian diagram}
\begin{theorem}\label{integral}
The following diagram of small v-stacks on $\Perf/\Spd\CO_E$ is 2-cartesian. 
    \[
    \begin{tikzcd}
        \SK \ar[r,"\pi_\text{crys}"] \ar[d,"\mathrm{red}"] & \mathrm{Sht}_{\G,\mu^{-1}}\ar[d, "BL^\mathrm{int}_{K_p}"]\\
        \Igs^\circ_{K^p}\ar[r, "\overline{\pi}^\circ_\text{HT}"] & \Bun_G
    \end{tikzcd}
    \]
 Its base change to $\Spd E$, identifies with the diagram in Corollary \ref{LevelK}.
\end{theorem}

\begin{proof}
The last statement is clear. We only need to show that the diagram is 2-cartesian. For convenience we denote the fiber product $\Igs^\circ_{K^p}\times_{\Bun_G}\mathrm{Sht}_{\G,\mu^{-1}}$ by $F$. The diagram commutes up to a natural isomorphism (comparison between prismatic and crystalline Dieudonn\'e modules). Hence there is a unique map $\SK\rightarrow F$ by the universal property. We need to show that this is an isomorphism.

We know by Corollary \ref{nice} that the map $\overline{\pi}^\circ_{HT}$ is qcqs, and hence its base change $F \rightarrow \mathrm{Sht}_{\G,\mu^{-1}}$ is also qcqs. Since the composition 
\[\SK\rightarrow F \rightarrow \mathrm{Sht}_{\G,\mu^{-1}}\]
is naturally isomorphic to $\pi_\text{crys}$ and is quasi-separated, the map $\SK\rightarrow F$ is qcqs by cancellation. In particular, that this is an isomorphism can be checked on geometric points. On a geometric point $s:=\Spa(C,C^+)$, a $\G$-shtuka over $s$ with a leg at $s^\sharp$ can be uniquely extended to a $\G$-BKF-module over $\CO_C$ with one leg at $\CO_{C^\sharp}$. Using Dieudonn\'e theory, it is the same as a $p$-divisible group with $G$-structure over $\CO_{C^\sharp}$. One can now argue as in Theorem \ref{cartesian}. Namely, an $s$-point of $F$ amounts to a triple $(\mathcal{A}_0, (s^\sharp,\mathcal{H}),\rho)$, where $\mathcal{A}_0$ is an abelian scheme over $C^+/\varpi$ with $G$- and $K$-level structures; $s^\sharp=\Spa(C^\sharp,C^{\sharp+})$ is an untilt of $s$ over $\Spd \CO_E$; $\mathcal{H}$ is a $p$-divisible group with $G$-structure over $\CO_{C^\sharp}$ and $\rho$ is a quasi-isogeny between $\mathcal{A}_0[p^\infty]\times_{C^+/\varpi} \CO_C/\varpi$ and $\mathcal{H}\times_{\CO_{C^\sharp}}\CO_C/\varpi$, compatibly with the $G$-structures. The argument in Theorem~\ref{cartesian} shows that this datum is equivalent to a formal abelian scheme with $G$- and $K$-level structures $\mathfrak{A}$ over $C^{\sharp+}$, i.e. an $s$-point of $\mathscr{S}_K^\diamond$.
\end{proof}

\subsection{Newton stratification}
Parallel to the discussion in Sections 7 and 9, we discuss the Newton stratification on the cartesian diagram in Theorem \ref{integral}. This recovers (at least on the level of v-sheaves) the almost product formula on Newton strata of PEL-type Shimura varieties due to Mantovan \cite[Proposition 11]{Mantovan}, cf. \cite[Section 4.3]{CS17}, as well as the $p$-adic uniformization of Rapoport and Zink \cite[Theorem 6.30]{RZ}.

Fix an algebraically closed field $k$ containing $\F_q$. Let $B(G)$ be the Kottwitz set for $G_{\Q_p}$. We base-change the cartesian diagram in Theorem \ref{integral} to $\Spd k$ and consider the absolute version of it. Namely, we forget the structure map $\mathscr{S}_K^\diamond\rightarrow \Spd W(k)$, and view it as the v-sheaf
\[S (\in \Perf_k)\mapsto \{(S^\sharp, S^\sharp \rightarrow \mathscr{S}_K^\text{ad})\mid S^\sharp \text{ is  an untilt of }S\},\]
similarly for $\mathrm{Sht}_{\G,\mu}$.

The Newton stratification on $\Bun_{G,k}$ define stratifications on all of $\Igs^\circ_{K_p}$, $\mathrm{Sht}_{\G,\mu^{-1}}$ and $\SK$. For $[b]\in B(G)$, we label the corresponding strata by a superscript $b$ (on $\Igs^\circ_{K_p}$ and $\SK$ this is empty unless $[b]\in B(G,\mu)$ by Proposition \ref{IgsImage}). Then it follows from Theorem \ref{integral} that

\begin{cor}
    The following diagram on $\Perf_k$ is 2-cartesian.
    \[
    \begin{tikzcd}
    \mathscr{S}_K^{\diamond,b} \ar[r,"\pi^b_\mathrm{crys}"] \ar[d] & \mathrm{Sht}^b_{\G,\mu^{-1}}\ar[d, "BL^\mathrm{int,b}_{K_p}"]\\
    \Igs^{\circ,b}_{K^p}\ar[r, "\overline{\pi}^{\circ,b}_\text{HT}"] & \Bun^b_G
    \end{tikzcd}
    \]
\end{cor}

In the rest of this section, we describe the strata and explain the relation to Mantovan's formula and Rapoport-Zink uniformization.

Denote by $\mathcal{D}$ the integral local PEL datum $(\CO_B\otimes \Z_p,\ast, \Lambda, (\cdot,\cdot), \mu^{-1}, [b])$, where the first five entries are determined by the integral global PEL-datum and $[b]\in B(G,\mu)$. Fix a representative $\mathbb{X}_b$ of the isogeny class of $p$-divisible groups with $G$-structure over $k$ labeled by $[b]$. This defines a formal scheme $\mathfrak{M}_\mathcal{D}$ (the Rapoport-Zink space attached to $\mathcal{D}$, see. \cite[Definition 3.21]{RZ}) (pro)-representing the following deformation functor of $p$-divisible groups:
\[\mathrm{Nilp}^\text{op}_{W(k)} \rightarrow \mathrm{Sets}\]
\[R\mapsto \{(\mathcal{H},f)\}/\sim,\]
where $\mathcal{H}$ is a $p$-divisible group over $R$ with principal polarization and $\CO_B\otimes \Z_p$-action (satisfying the Kottwitz condition and compatibility with $\ast$, see. Remark \ref{G-p-div}), and $f: \mathcal{H}\times_R R/p\dashrightarrow \mathbb{X}_b\times_k R/p$ is a $B_{\Q_p}$-linear quasi-isogeny, preserving the polarization up to a scalar in $\underline{\Q}^\times_p(\Spec(R/p))$. The equivalence relation is given by isomorphisms of such pairs. We consider its attached v-sheaf and view it as defined absolutely over $\Spd k$.

Recall from Definition~\ref{def:Aut} the formal group scheme $\underline{\mathrm{Aut}}_G(\tilde{\mathbb{X}}_b)$ of self-quasi-isogenies compatible with the extra structures on $\mathbb{X}_b$. Note that its attached v-sheaf $\underline{\mathrm{Aut}}_G(\tilde{\mathbb{X}}_b)^\diamond$ acts on $\mathfrak{M}_{\mathcal{D}}^\diamond$ as follows: given $S=\Spa(R,R^+)$ in $\Perf_k$. An $S$-point of $\underline{\mathrm{Aut}}_G(\tilde{\mathbb{X}}_b)^\diamond$ is the same as an element $g\in \underline{\mathrm{Aut}}_G(\tilde{\mathbb{X}}_b\times_kR^+/\varpi)$ for any chosen pseudo-uniformizer $\varpi$, cf. the proof of Corollary~\ref{cor: CompareAut}. While an $S$-point of $\mathfrak{M}_{\mathcal{D}}^\diamond$ amounts to a triple $(S^\sharp, \mathcal{H}, f)$, where $S^\sharp=(R^\sharp,R^{\sharp+})$ is an untilt of $S$ over $W(k)$, and $(\mathcal{H},f)$ is a $\Spf R^{\sharp +}$-point of $\mathfrak{M}_{\mathcal{D}}$. We let $g.(S^\sharp, \mathcal{H}, f)$ be the triple $(S^\sharp, \mathcal{H},f')$, where $S^\sharp=(S^\sharp,\iota:S^{\sharp\flat}\simeq S)$ and $\mathcal{H}$ are as before, and $f'$ is the composition 
\[\mathcal{H}\times_{R^{\sharp+}} R^{\sharp+}/p\dashrightarrow \mathbb{X}_b\times_k R^{\sharp+}/p\dashrightarrow \mathbb{X}_b\times_k R^{\sharp+}/p,\]
where the first dashed arrow is the quasi-isogeny $f$ and the second is $g$, using the isomorphism 
\[{R^{\sharp+}}/p\simeq R^+/(\varpi:=p^\flat)\]
specified by the map $\iota$.

We let $\mathscr{E}_b$ denote the $\Spd k$-point of $\Bun_G$ determined by $\mathbb{X}_b$ and $\widetilde{G}_b$ be the v-sheaf of automorphism groups of it. We identify $\widetilde{G}_b$ with $\underline{\mathrm{Aut}}_G (\tilde{\mathbb{X}}_b)^\diamond$ by Corollary~\ref{cor: CompareAut}.
\begin{lemma}\label{ShtStrata}
    The locally closed substack $\mathrm{Sht}^b_{\G,\mu^{-1}}$ of $\mathrm{Sht}_{\G,\mu^{-1}}$ is isomorphic to the quotient stack
    \[\left[\mathfrak{M}_{\mathcal{D}}^\diamond /\underline{\mathrm{Aut}}_G (\tilde{\mathbb{X}}_b)^\diamond\right]\]
    and the quotient map $\mathfrak{M}_{\mathcal{D}}^\diamond\rightarrow \mathrm{Sht}^b_{\G,\mu^{-1}}$ identifies with the tautological $\widetilde{G}_b$-torsor coming from the (integral) Beauville-Laszlo map $\mathrm{Sht}^b_{\G,\mu^{-1}}\rightarrow \Bun_G^b\cong [\ast/\widetilde{G}_b]$.
\end{lemma}
\begin{proof}
    Consider the $\widetilde{G}_b$-torsor $\mathcal{M}^\text{int}_{(\G,b,
    \mu^{-1})}$ over $\mathrm{Sht}_{\G,\mu^{-1}}^b$, which parametrizes for a $\G$-shtuka $(\CP,\varphi_{\CP})$ Frobenius equivariant trivializations 
    \[\iota_r:\CP|_{\mathcal{Y}_{[r,\infty)}(S)}\cong \tilde{\mathscr{E}}_b,\]
    where $r\in (0,\infty)$ is large enough such that $S^\sharp$ does not intersect $\mathcal{Y}_{[r,\infty)}(S)$, and $\tilde{\mathscr{E}}_b$ is the pullback of $\mathscr{E}_b$ to $\mathcal{Y}_{[r,\infty)}(S)$.
    This is the integral local Shimura variety of \cite[Section 25.1]{Berkeley} and is isomorphic to $\mathfrak{M}_{\mathcal{D}}^\diamond$ as a v-sheaf by \cite[Corollary 25.1.3]{Berkeley}. The last statement follows directly from the definition of the (integral) Beauville-Laszlo map in Proposition \ref{VB}.
\end{proof}

Let $\mathrm{Ig}^b$ be the perfect Igusa variety over $k$ defined by $\mathbb{X}_b$ and $\overline{\mathrm{Ig}^{b,\diamond}}$ the canonical compactification of its attached v-sheaf towards $\Spd k$. Let $\overline{\mathscr{S}_K^{\diamond,b}}$ be the canonical compactification of $\mathscr{S}_K^{\diamond,b}$ towards $\mathrm{Sht}_{\G,\mu^{-1}}^b$. Combine the above with the description of the Newton strata of the Igusa stack in Proposition \ref{stratum}, we get the following formula of Newton strata on the Shimura variety.
\begin{cor}[{Mantovan's product formula, cf. \cite[Proposition 11]{Mantovan}, \cite[Section 4.3]{CS17}}]
   \begin{align*}
   \overline{\mathscr{S}_K^{\diamond,b}} & \cong [\overline{\mathrm{Ig}^{b,\diamond}}/\widetilde{G}_b]\times_{[\ast/\widetilde{G}_b]} [\mathfrak{M}_{\mathcal{D}}^\diamond/\widetilde{G}_b]\\
    &\cong [(\overline{\mathrm{Ig}^{b,\diamond}}\times \mathfrak{M}_{\mathcal{D}}^\diamond)/\underline{\mathrm{Aut}}_G (\tilde{\mathbb{X}}_b)^\diamond],
   \end{align*}
   where in the second line, the quotient is taken for the diagonal action of $\underline{\mathrm{Aut}}_G (\tilde{\mathbb{X}}_b)^\diamond$. 
\end{cor}

Choose any $k$-point $x$ on the fiber $\mathscr{S}_{K,k}$ of the Shimura variety over $k$, which is an abelian variety $A$ over $k$ with $\CO_B$-endomorphism $\iota$, polarization $\lambda$, and $K^p$-level structure $\bar{\eta}$. Consider the algebraic group $I_x$ over $\Q$ whose value on a $\Q$-algebra $R$ is 
\[\{g\in \mathrm{End}_{\CO_B}(A)\otimes_{\Z} R\mid gg^\ast \in R^\times\cdot \mathrm{id}_A\},\]
where $\ast$ denotes the Rosati involution induced by $\lambda$. Namely, we take $\CO_B$-linear self-quasi-isogenies of $A$ that preserve the polarization up to a scalar. Note that for a prime $\ell\neq p$,
\[\mathrm{End}_{\CO_B}(A)\otimes \Q_\ell \hookrightarrow \mathrm{End}_B(V_\ell(A)) \cong 
\mathrm{End}_B(V\otimes_{\Q}\Q_\ell),\]
so we have $I_x(\Q_\ell)\hookrightarrow G(\Q_\ell)$. We underline the topological groups $I_x(\Q)$, $G(\Af^p)$, $K^p$ to denote their attached v-sheaves. Consider the v-sheaf theoretic double quotient 
\[[\underline{I_x(\Q)}\backslash \underline{G(\Af^p)}/\underline{K^p}],\]
where $\underline{I_x(\Q)}$ acts on $\underline{G(\Af^p)}$ from the left via diagonal embedding
\[I_x(\Q)\hookrightarrow I_x(\Af^p)\hookrightarrow G(\Af^p)\]
and $\underline{K^p}$ acts from the right by the regular action.

We define a ``uniformization map" of small v-stacks
\[\Theta_x: [\underline{I_x(\Q)}\backslash \underline{G(\Af^p)}/\underline{K^p}]\rightarrow \Igs^\circ_{K^p},\]
as below: Given a totally disconnected test object $S=\Spa(R,R^+)$ with chosen pseudo-uniformizer $\varpi\in R^+$, write $R_0$ for $R^+/\varpi$. On the level of presheaves of groupoids, $\Theta_x$ sends a point of the left hand side, represented by the section 
\[g\in \underline{G(\Af^p)}(S)=\mathrm{Maps}(\pi_0(S), G(\Af^p)) = \mathrm{Maps}(\Spec(R_0), G(\Af^p))\]
to the tuple
\[(A_{R_0},\iota_{R_0}, \lambda_{R_0}, \overline{g^{-1}\eta_{R_0}}),\]
where $A_{R_0}:=A\times_k R_0$ and $\iota_{R_0}$, $\lambda_{R_0}$, $\overline{\eta_{R_0}}$ are the corresponding base changes. This preserves automorphisms and induces a map of v-stacks. 

Assume the $p$-divisible group of $A$ is in the isogeny class labeled by $[b]\in B(G,\mu)$. Then the image of the above map lies in the substack $\Igs^{\circ,b}_{K^p}$. If furthermore the element $[b]$ is basic as in Remark \ref{basic}, then the group $I_x$ is an inner form of $G$, agreeing with $G$ at all places but $p$ and infinity, see \cite[Theorem 6.30]{RZ}. In this case, by the proof of \textit{loc. cit.}, the set of isogeny classes of abelian varieties over $k$ with $G$-structure, whose $p$-divisible group has isogeny type $[b]$, is finite and is bijective to the Hasse kernel $\mathrm{ker}^1(\Q,I_x)$ (and hence to $\mathrm{ker}^1(\Q,G)$ by \cite[Equation (4.2.2)]{Kottwitz84}) of the map
\[\mathrm{res}: H^1(\Q, I_x)\rightarrow \prod_p H^1(\Q_p, I_x).\]
For each such isogeny class, we fix a point $x_i\in \mathscr{S}_{K}(k)$ lying in it, and define 
\[\Theta=\coprod_{i\in \mathrm{ker}^1(\Q,G)} \Theta_{x_i}: \coprod_{i\in \mathrm{ker}^1(\Q,G)} [\underline{I_{x_i}(\Q)} \backslash \underline{G(\Af^p)} /\underline{K^p}]\rightarrow \Igs^{\circ,b}_{K^p}.\]
We have the following reformulation of Rapoport-Zink uniformization:

\begin{prop}[{cf. \cite[Theorem 6.30]{RZ}}] \label{BasicStratum}
For $[b]$ basic, the uniformization map $\Theta$ is an isomorphism of small v-stacks over $\Spd k$.
\end{prop}

\begin{proof}
    This can be checked v-locally. Hence we may pull back along the v-cover $\mathscr{S}_K^{\diamond, b}\rightarrow \Igs^{\circ,b}_{K^p}$. To simplify notation, let us denote the source of $\Theta$ by $X$ and the pullback by $Y$. Consider the following diagram where all squares are cartesian
    \[
    \begin{tikzcd}
        Y\ar[r,"\tilde{\Theta}"]\ar[d] & \mathscr{S}_K^{\diamond, b} \ar[r]\ar[d] & \mathrm{Sht}_{\G,\mu^{-1}}^b\ar[d]\\
        X \ar[r,"\Theta"] & \Igs^{\circ,b}_{K^p} \ar[r] & \Bun_G^b.
    \end{tikzcd}
    \]

    By a direct computation, the tautological $\widetilde{G}_b$-torsor above $X$ corresponding to the map 
    \[X\xrightarrow{\Theta} \Igs^b_{K^p}\rightarrow \Bun_G^b\cong [\ast/\widetilde{G}_b]\]
    can be identified with
    \[X':= \coprod_{i\in \mathrm{ker}^1(\Q,G)} [\underline{I_{x_i}(\Q)} \backslash \widetilde{G}_b\times \underline{G(\Af^p)} /\underline{K^p}],\]
    where for each $i$, $\underline{I_{x_i}(\Q)}$ acts diagonally on the two middle terms: on $\widetilde{G}_b\cong \underline{\mathrm{Aut}}(\widetilde{\mathbb{X}}_b)^\diamond$ via the map 
    \[I_{x_i}(\Q)\rightarrow \mathrm{Aut}(\widetilde{A_{x_i}[p^\infty]})\cong \mathrm{Aut}(\widetilde{\mathbb{X}}_b),\] 
    and on $\underline{G(\Af^p)}$ as explained earlier; while $\underline{K^p}$ acts via the right regular action on $\underline{G(\Af^p)}$; $\widetilde{G}_b$ acts on $X'$ via the right regular action on itself.
    
    Identify $\mathrm{Sht}_{\G,\mu^{-1}}^b$ with  $[\mathfrak{M}_{\mathcal{D}}^\diamond
    /\widetilde{G}_b]$ using Lemma~\ref{ShtStrata}. Then using the outer cartesian square we can compute $Y$ to be the product of $X'$ with $\mathfrak{M}_{\mathcal{D}}^\diamond$ quotienting by the diagonal action of $\widetilde{G}_b$. This simplifies to the following formula
    \[\coprod_{i\in \mathrm{ker}^1(\Q,G)} [\underline{I_{x_i}(\Q)} \backslash (\mathfrak{M}_{\mathcal{D}}^\diamond\times \underline{G(\Af^p)}/\underline{K^p})],\]
    where $I_{x_i}(\Q)$ acts diagonally on the middle terms and the action on $\mathfrak{M}_{\mathcal{D}}^\diamond$ is via its map to $\mathrm{Aut}(\widetilde{\mathbb{X}}_b)$. The map $\tilde{\Theta}$ agrees with the $p$-adic uniformization map of \cite[Theorem 6.30]{RZ} by comparing the construction in \cite[p. 283, Equation (6.3)]{RZ} with ours, cf. the proof of Theorem \ref{integral}. Now $\tilde{\Theta}$ is an isomorphism by \cite[Theorem 6.30]{RZ}, upon identifying $\mathscr{S}_K^{\diamond,b}$ with the v-sheaf attached to the completion of $S_K$ (as a $W(k)$-scheme) along $T\subset S_{K,k}$, the closed subscheme where the universal $p$-divisible group is geometric fiberwise of isogeny class $b$. But this is fine, since both are the open sub-v-sheaf of $\SK$, obtained from sheafifying the presheaf
    \[S=\Spa(R,R^+)\mapsto \left\lbrace (S^\sharp, f \in \mathscr{S}_K(R^{\sharp+})) \;\middle|\;
  \begin{tabular}{@{}l@{}}
	$f^\ast \mathfrak{A} \times_{R^{\sharp+}} R^+/\varpi$ is geometric-\\
	pointwise of isogeny class $[b]$\\
   \end{tabular}
  \right\rbrace,
	\]
    where $S^\sharp$ is an untilt of $S$ over $\CO_E$ and $\mathfrak{A}$ is the universal formal abelian scheme on $\mathscr{S}_K$. 
\end{proof}

\begin{cor}
    If $[b]$ is basic, then the Igusa variety $\mathrm{Ig}^{b,\diamond}$ as a v-sheaf on $\Perf_k$ is isomorphic to
    \[\coprod_{i\in \mathrm{ker}^1(\Q,G)} [\underline{I_{x_i}(\Q)} \backslash \underline{G_b(\Q_p)}\times \underline{G(\Af^p)} /\underline{K^p}].\]
\end{cor}

\begin{proof}
    Combine the description of $\widetilde{G}_b$ in Remark \ref{basic}, description of the strata of the Igusa stack in Proposition \ref{stratum}, Proposition \ref{BasicStratum} and the identification of the tautological $\widetilde{G}_b$-torsor in its proof. (In this case partial minimal and canonical compactifications are unnecessary, since $\mathrm{Ig}^b$ is zero-dimensional.)
\end{proof}
\par

\newpage
\bibliographystyle{amsalpha}
\bibliography{bib}

\end{document}